\documentclass{article}
\usepackage[a4paper, total={6.5in, 8.5in}]{geometry}

\usepackage{amsmath}  
\usepackage{amssymb}  
\usepackage{amsthm}   
\usepackage{mathtools}

\usepackage{esvect}

\usepackage{amsfonts}
\usepackage[T1]{fontenc}
\usepackage{kpfonts}
\usepackage{enumerate}
\usepackage{amssymb}
\usepackage{graphicx}
\usepackage{lmodern}
\usepackage{dsfont}
\usepackage{pgfplots}

\usepackage{algorithm}
\usepackage{algorithmic}

\usepackage{cite}
\usepackage{multicol}

\usepackage{pdflscape}

\usepackage{url}

\newcommand\calC{{\mathcal C}}

\newcommand\calK{{\cal K}}

\newcommand\calN{{\cal N}}

\newcommand\calA{{\mathcal A}}

\newcommand\calD{{\mathcal D}}
\newcommand\calE{{\mathcal E}}

\newcommand\calQ{{\mathcal Q}}

\newcommand{\half}{{\frac{1}{2}}}

\newcommand{\Om}{\Omega}

\newcommand{\eps}{\varepsilon}
\newcommand{\ka}{\kappa}
\newcommand{\al}{\alpha}
\newcommand{\be}{\beta}
\newcommand{\ga}{\gamma}
\newcommand{\Lm}{L_{max}}

\renewcommand{\div}{\nabla \cdot}
\newcommand{\cd}{\cdot}
\newcommand{\nb}{\nabla}

\renewcommand{\l}{\ell}

\renewcommand{\k}{{(k)}}
\newcommand{\m}{{(m)}}

\renewcommand{\l}{{\ell}}
\newcommand{\MSE}{\text{MSE}}
\newcommand{\Cov}{\text{Cov}}
\newcommand{\Exp}{\text{Exp}}
\newcommand{\MLMC}{\text{MLMC}}
\newcommand{\CLMC}{\text{CLMC}}
\newcommand{\QCLMC}{\text{QCLMC}}

\newcommand{\dl}{\textup{ d}\ell}
\newcommand{\dQ}{\textup{ d}Q}
\newcommand{\dQl}{\frac{\textup{d}Q}{\dl}}
\newcommand{\dQll}{\frac{\textup{d}Q (\l)}{\dl}}

\newcommand{\inv}{^{-1}}

\def \to {\rightarrow}

\newcommand{\E}{{\mathbb E}}
\newcommand{\N}{{\mathbb N}}
\newcommand{\R}{{\mathbb R}}
\newcommand{\V}{{\mathbb V}}
\renewcommand{\P}{{\mathbb P}}

\newcommand\quadand{\quad\hbox{ and }\quad}

\newcommand\quadfor{\quad\hbox{ for }\quad}

\newtheorem{theorem}{Theorem}

\newtheorem{lemma}[theorem]{Lemma}
\newtheorem{proposition}[theorem]{Proposition}

\newtheorem{remark}[theorem]{Remark}

\numberwithin{equation}{section}
\numberwithin{figure}{section}
\numberwithin{table}{section}
\numberwithin{theorem}{section}

\usepackage{authblk}

\makeindex             

\usepackage{fancyhdr}

\begin{document}

\pagestyle{fancy}
\fancyhead[L]{Complexity analysis of quasi continuous level Monte Carlo}
\fancyhead[R]{C. A. Beschle and A. Barth}

\title{Complexity analysis of quasi continuous level Monte Carlo}
\author[$\dagger$]{Cedric Aaron Beschle}
\author[$\star$]{Andrea Barth}

\affil[$\dagger$]{cedric.beschle@ians.uni-stuttgart.de}
\affil[$\star$]{andrea.barth@ians.uni-stuttgart.de}

\date{\today}

\maketitle

\begin{abstract}
    Continuous level Monte Carlo is an unbiased, continuous version of the celebrated multilevel Monte Carlo method. The approximation level is assumed to be continuous resulting in a stochastic process describing the quantity of interest. Continuous level Monte Carlo methods allow naturally for samplewise adaptive mesh refinements, which are indicated by goal-oriented error estimators. The samplewise refinement levels are drawn in the estimator from an exponentially-distributed random variable. Unfortunately in practical examples this results in higher costs due to high variance in the samples. In this paper we propose a variant of continuous level Monte Carlo, where a quasi Monte Carlo sequence is utilized to ``sample'' the exponential random variable. We provide a complexity theorem for this novel estimator and show that this results theoretically and practically in a variance reduction of the whole estimator.  
\end{abstract}

\maketitle

\section*{Introduction}
During the last decade multilevel Monte Carlo methods and its variants as multiindex and multifidelity Monte Carlo have been successfully applied to reduce the costs of solving various uncertain problems (see e.g.~\cite{Heinrich2001_MLMC, Giles2008_MLMCPath, Cliffe2011_MLMCPDE, Assyr2013_MLMC, Peherstorfer2016_MFMC, Tempone2016_MIMC}). The multilevel Monte Carlo estimator combines discretizations of a quantity of interest on a hierarchy of refinements in such a way that many samples on coarse refinement levels are combined with few samples on fine discretization levels to reduce the variance of the estimator when compared to a naive Monte Carlo approach. The multilevel Monte Carlo estimator is asymptotically optimal and reduces the costs to compute the quantity of interest considerably. This reduction comes at a cost: The quantity of interest has to be available on the hierarchy of discretizations and the variance of the details (difference of subsequent discretizations) has to decrease faster than the costs increase. Further, the weak and the strong error have to fulfill a certain ratio for the optimal cost reduction (see~\cite{Giles2008_MLMCPath} for a detailed complexity theorem). However, in all cases the multilevel Monte Carlo estimator is biased (as is the singlelevel Monte Carlo estimator). 

An unbiased variant was introduced in~\cite{Detommaso2019_CLMC} with the continuous level Monte Carlo method. Here, the resolution levels are assumed to be continuous, resulting in a stochastic process describing the family of approximations of the quantity of interest. The refinement of each sample is determined by a (level) random variable. The samples are adaptively refined according to an a-posteriori error estimator. In practical terms the tail estimate of this level random variable is a crucial component in the performance of the estimator. In~\cite{BeschleBarth2023_QCLMC} a continuous level estimator was used to solve an elliptic problem with a discontinuous random coefficient and furthermore, the performance of the continuous level estimator was compared to its multilevel variant. The problem was chosen such that the continuous level estimator should have outperformed the multilevel method. Unfortunately the performance of the continuous level estimator is very sensitive to the tail estimate provided by the samples of the level random variable. To reduce this sensitivity the authors propose to use a quasi Monte Carlo sequence instead of i.i.d. samples of the level random variable to reduce the variance of the tail estimate and the whole estimator significantly. A similar idea was mentioned in a remark in \cite{Vihola2018_unbiasedMLMC} in the general framework of unbiased MLMC estimators. The so called quasi continuous level Monte Carlo estimator outperformed not only the continuous level but also the multilevel Monte Carlo estimator. 

In this paper we investigate theoretically the quasi continuous level Monte Carlo method. We provide a complexity theorem which shows that the quasi continuous level Monte Carlo method has the same optimal complexity as the continuous level and multilevel estimator, but with a potentially lower variance and therefore an overall improved time to error performance as demonstrated by numerical experiments, where we apply both methods --- continuous level Monte Carlo and quasi continuous level Monte Carlo --- to an elliptic PDE problem with a log-Gauss random coefficient and compare their performance in different hyperparameter settings. As we only exchange the one-dimensional random variable with a quasi Monte Carlo sequence the cost increase is negligible. The quasi Monte Carlo sequence provides a more accurate tail estimate than pseudo-random numbers for the level random variable. We emphasise that we are not proposing a quasi (multilevel) Monte Carlo method to solve the uncertain PDE (as e.g. in~\cite{Graham2011_QMC, Kuo2012_QMC, Kuo2015_MLQMC, Graham2015_QMClognormal}). The performance of those methods depends on the dimensionality of the problem, the proposed quasi continuous level Monte Carlo method is robust in this aspect.

The paper is organized as follows: In Section~\ref{sec:CLMC} we briefly recap the continuous level Monte Carlo method and its complexity. We introduce quasi-random sequences and the notion of $F$-discrepancy in Section~\ref{sec:Quasi-random sequences and F-discrepancy}. We use the $F$-discrepancy in Section~\ref{sec:QCLMC} to prove a complexity theorem for the quasi continuous level Monte Carlo method. In Section~\ref{sec:numerical_experiments} we introduce the classical two-dimensional random elliptic PDE model we use as a test case. The coefficient is given by a log-Gauss random field and for the $H^1$-norm as the quantity of interest we use a standard a-posteriori error estimator. For the performance comparison we estimate all parameters which are involved in the assumptions of the complexity theorem and compare the performances of the methods for different choices of the hyperparameters of the log-Gauss field  based on respective theoretical upper bounds to the mean squared error of the methods stemming from the proofs of the complexity theorems and in a time to error performance over several simulation runs via a proposed algorithm. 

\section{Continuous level Monte Carlo}
\label{sec:CLMC}
We consider a stochastic model and denote by $\calQ$ a real-valued quantity of interest of its solution. Denote by $Q_L$ an approximation of $\calQ$ by a discretization-based numerical scheme, for example a finite element method, to some resolution parameter $L \in \N_0$, e.g., corresponding to the degrees of freedom (DOF) of a mesh. By assuming that $\E[Q_L] \to \E[\calQ]$ for $L \to \infty$ $\mathbb{P}\textup{-almost surely}$, we are able to compute estimates $\widehat{Q}_L^{est}$ for the mean value $\E[\calQ]$ up to some desired accuracy with an average of independent approximation samples $(Q_L^\k; k=1,\dots,M)$ for $M \in \N$. 
%
The mean squared error (MSE) of an estimator $\widehat{Q}_L^{est}$ for the mean value $\E[\calQ]$ together with its decomposition into a variance and squared bias term is given by
\begin{equation}
\begin{aligned}
\MSE = \E\left[(\widehat{Q}_L^{est} - \E[\calQ])^2 \right] = \V[\widehat{Q}_L^{est}] + \left(\E[\widehat{Q}_L^{est} - \calQ]\right)^2.
 \end{aligned} 
 \label{eq:MSE}
\end{equation}

In this work we investigate the recently developed quasi continuous level Monte Carlo method (QCLMC) \cite{BeschleBarth2023_QCLMC}, which is an improved version of the continuous level Monte Carlo method developed in \cite{Detommaso2019_CLMC}.
The continuous level Monte Carlo (CLMC) method estimates the mean value of a quantity of interest with samplewise adaptive mesh hierarchies. This is realized by assuming a continuous resolution (level) $\l \in \R>0$  and a continuous family of approximations $(Q(\l);\l \geq 0)$ of $\calQ$ viewed as a stochastic process defined on a probability space $(\Omega, \calA, \P)$ with $\E\left[\left|\frac{\text{d}Q}{\dl}\right|\right] \in L^1((0,\infty);\R)$, such that $Q(\l) \to \calQ$ as $\l \to \infty$ $\P\textup{-almost surely}$.
With these considerations, the CLMC estimator is defined by
\begin{equation}
    \widehat{Q}_{0,\Lm}^\CLMC := \frac{1}{M} \sum_{k=1}^M \int_0^{\Lm} \frac{1}{\P(L_{r} \geq \l)} \left(\dQl\right)^\k (\l) \mathds{1}_{[0,L_{r}^\k]}(\l) \dl,
\label{eq:CLMC}
\end{equation}
with deterministic maximal level $\Lm \in (0,+\infty]$, total sample number $M \in \N$ and a random variable $L_{r}$ with finite expectation and $\P(L_r \geq \l) > 0$ for all $\l \in (0, \Lm)$, that is independent of the stochastic process $(Q(\l);\l \geq 0)$. For each sample $k=1,\dots,M$, the minimum of an i.i.d. copy $L_{r}^\k$ of $L_r$ and the predetermined $\Lm$ corresponds to the maximal computed resolution for this sample. 
Note, that the CLMC estimator \eqref{eq:CLMC} is defined as an estimator for the difference quantity $\E[\calQ - Q(0)]$, thus the indexing by ${0, \Lm}$, and in order to obtain an estimator for $\E[\calQ]$ it suffices to add an unbiased Monte Carlo estimator for $\E[Q(0)]$. 
The CLMC method is an unbiased estimator for $\E[Q(\Lm) - Q(0)]$ and in the case $\Lm = \infty$, it is an unbiased estimator for $\E[\calQ - Q(0)]$, i.e., $\E[\widehat{Q}_{0, \infty}^\CLMC] = \E[\calQ - Q(0)]$ and thus, its MSE expansion \eqref{eq:MSE} reduces to $\MSE^\CLMC = \V[\widehat{Q}_{\infty}^\CLMC]$.
As stated in \cite{Detommaso2019_CLMC}, under the assumption that there exist positive constants $\al$, $\be$, $\ga$, $c_1$, $c_2$, $c_3$ such that for any $\l > 0$ we have
\begin{equation}
   \E \left[\dQll\right] \leq c_1 e^{-\al \l},
  \quad
    \V \left[\dQll\right] \leq c_2 e^{-\be \l}, \quad
    \frac{\text{d} \calC[\l]}{\dl} \leq c_3 e^{\ga \l},
      \label{eq:CLMCassumptions}
\end{equation}   
where $\calC[\l]$ is the total accumulated cost to compute a sample of $Q(\l)$ (cf.~Remark \ref{rem:adapted_cost_growth_assumption}) and that $L_r \sim \Exp(r)$ is exponentially distributed with $r \in [\min\{2\al,\be, \ga\},\;\max\{\min\{\be,2\al\},\ga\}]$, then for any $\eps\in(0,\frac{1}{e})$ there exist and $M \in \N$ and $\widetilde{C} > 0$ such that 
\begin{equation}
    \MSE^\CLMC \leq \eps^2 \quadand \calC\left[\widehat{Q}_{0, \Lm}^\CLMC\right] \leq \widetilde{C} \eps^{-2-\max\{0, \frac{\ga - \min\{\be,2\al\}}{\al}}\} |\log(\eps)|^{\delta_{r,\be} + \delta_{r, 2\al} + \delta_{r,\ga}}.
    \label{eq:CLMC-complexity}
\end{equation}
In the case $\min\{\be,2\al\} > \ga$ and $r \in (\ga,\;\min\{\be,2\al\})$ with $\Lm = \infty$ this reduces to
\begin{equation*}
    \MSE^\CLMC \leq \eps^2 \quadand \calC[\widehat{Q}_{0, \infty}^\CLMC] \leq \widetilde{C} \eps^{-2}.
\end{equation*}

Next, we explain how CLMC generalizes MLMC. For any sample $1 \leq k \leq M$, suppose that $(Q_j^\k; j \geq 1)$ denotes a countable sequence of approximations of $Q^\k$ at levels $(\l_j^\k; j \geq 1)$. We choose a linear interpolation for the derivative samples
\begin{equation}
 \left(\dQl\right)^\k (\l) := \frac{Q_j^\k - Q_{j-1}^\k}{\l_{j}^\k - \l_{j-1}^\k} \quad\hbox{ for }\quad \in (\l_{j-1}^\k ,\l_{j}^\k],
 \label{eq:linear_interpolation}
\end{equation}
which yields the estimator
\begin{equation}
    \widehat{Q}_{0,\Lm}^\CLMC = \frac{1}{M} \sum_{k=1}^M \sum_{j=1}^{J^\k} \int_{\l_{j-1}^\k}^{\tilde{\l_j}^\k} \frac{1}{\P(L \geq \l)} \dl \; \frac{Q_j^\k - Q_{j-1}^\k}{\l_{j}^\k - \l_{j-1}^\k},
 \label{eq:CLMC_discretized}
\end{equation}
with
\begin{equation*}
    J^\k := \min\{j \geq 1: \ell_j^\k \geq L_r^\k \land \Lm\}, \quad \tilde{\ell}_j^\k := \ell_j^\k \land L_r^\k \land \Lm.
    \label{eq:CLMC-max-level-index}
\end{equation*}
Setting $\l_j^\k = j$ for $j \in \N$ and all $k=1,\dots,M$ to restrict the estimator to the integer level framework and choosing $\P(L_r \geq j)$ to be a discrete distribution over the levels, that is constant over $(j-1, j)$, reduces the CLMC estimator $\eqref{eq:CLMC_discretized}$ to, cf.~\cite{Detommaso2019_CLMC},
\begin{equation*}
    \widehat{Q}_{0,\Lm}^\CLMC = \frac{1}{M} \sum_{k=1}^M \sum_{j=1}^{J^\k} \frac{1}{\P(L_r \geq j)}\; \left(Q_j^\k - Q_{j-1}^\k\right).
 \label{eq:CLMC_uniform}
\end{equation*}
This is directly connected to the (unbiased) estimator by setting $M_j := M\P(L \geq j)$, introduced by Rhee and Glynn in \cite{Rhee2015_unbiasedSDE},
\begin{equation*}
 \widehat{Q}_{0,\Lm}^\CLMC = \sum_{k=1}^M \sum_{j=1}^{J^\k} \frac{1}{M_j} \; \left(Q_j^\k - Q_{j-1}^\k\right).
\end{equation*}
Furthermore, with $0 < \Lm \in \N$, this may be expressed as, cf.~\cite{Giles2015_MLMC},
\begin{equation*}
 \widehat{Q}_{0,\Lm}^\MLMC := \sum_{j=1}^{\Lm} \frac{1}{M_j} \sum_{k=1}^{M_j} \left(Q_j^\k - Q_{j-1}^\k\right),
\end{equation*}
which is the formula of the standard multilevel Monte Carlo estimator, where $\Lm$ is the maximal level and the sample numbers $M_j$ on each level $j=1,..,\Lm$ are not probabilistic, but deterministic.

\begin{remark}
\label{rem:choice_L_r}

 As demonstrated in the derivation of the MLMC estimator, the level random variable $L_r$ is not restricted to the exponential distribution, but merely has to be independent of the stochastic process $(Q(\l);\l \geq 0)$ and have finite expectation with $\P(L_r \geq \l) > 0$ for all $\l \in (0, \Lm)$. 

\end{remark}

\begin{remark}
 \label{rem:adapted_cost_growth_assumption}

 The parameters $c_1, \al$, $c_2, \be$ and $c_3, \ga$ from \eqref{eq:CLMCassumptions} depend not only on the approximations $Q_j^\k = Q(\l_j^\k)$ at refinement $j \in \N$ for $k=1,...,M$ but also on the definition of the derivative $\dQl^\k$ and the samplewise level $\l_j^\k$ in a practical setting. 
 The assumption on the bias and variance decay of the derivative quantity $\dQl$ scales with the change in the level $\dl$ in a way, that changes the constants $c_1$ and $c_2$ correspondingly. This is reflected in the assumption for the cost growth in \eqref{eq:CLMCassumptions}, which is an assumption for the rate of change of the total cost $\text{d}\calC[\l]$ to compute an approximation of $Q(\l)$, with respect to the change in the level $\dl$. The assumption on the cost growth, $\calC[\l] \leq c_3 e^{\ga \l}$, where $\calC[\l]$ is the cost to compute one sample of $Q(\l)$, as in \cite[Theorem~$2.3$]{Detommaso2019_CLMC} does not scale $c_3$ accordingly and allows to construct a practical estimator of CLMC with a specific definition for the samplewise level $\l_j^\k$ at refinement $j \in \N$ for $k=1,...,M$ such that the theoretical upper bound to the estimator's cost in the CLMC complexity theorem becomes arbitrarily small.
\end{remark}

The key difference between the CLMC estimator and the QCLMC estimator is the choice of how to compute $L_r^\k$. While in CLMC $L_r^\k$ are i.i.d. copies of the random variable $L_r$ for $k=1,\dots,M$, in QCLMC we choose $L_r^\k$ to be a deterministic quasi-random sequence for $k=1,\dots,M$ yielding a better approximation of the underlying tail distribution via their improved $F$-discrepancy convergence. These concepts are introduced in the next section and the improvement in the approximation is demonstrated.

\section{Quasi-random sequences and $F$-discrepancy}
\label{sec:Quasi-random sequences and F-discrepancy}
The discrepancy of a set of points $P$ consisting of $x^{(1)},\dots,x^{(M)} \in [0,1)^s$ for $M \in \N$ and $s \in \N$ is given by
\begin{equation}
 D_M(\mathcal{B}; P) = \sup_{B \in \mathcal{B}} \left|\frac{1}{M} \sum_{k=1}^M \mathds{1}_{B}(x^\k) - \lambda(B) \right|,
 \label{eq:discr}
\end{equation}
cf.~, e.g., \cite{Niederreiter1992_QMC}, where $\lambda$ is the Lebesgue measure and $\mathcal{B}$ a non-empty family of Lebesgue-measurable subsets of $[0,1)^s$. For simplicity and since it fits our considerations we assume $s=1$. 
Quasi-random sequences are numbers $x^{(1)},\dots,x^{(M)} \in [0,1)$ specifically designed such that the discrepancy converges to zero at a much faster rate than for pseudo-random number sequences, i.e.,
\begin{equation}
  \begin{aligned}
  &D_M(\mathcal{B}; P) \leq c_{disc} M^{\ka-1} \text{ for some } \ka \geq 0, c_{disc}>0 \text{ independent of $M \in$.}
  \end{aligned}
\label{eq:discr_conv}
\end{equation}
This is no probability convergence statement, because quasi-random numbers are essentially deterministic.
The specific choice of $\mathcal{B}$ in the discrepancy definition \eqref{eq:discr} as the family of all subintervals $[0,x] \subset [0,1)$, where $x \in (0,1)$, leads to the star-discrepancy
\begin{equation*}
 D_M^\ast(P) = \sup_{[0,x] \subset [0,1)} \left|\frac{1}{M} \sum_{k=1}^M \mathds{1}_{[0,x]}(x^\k) - x \right| = \sup_{x \in (0,1)} \left|\frac{1}{M} \sum_{k=1}^M \mathds{1}_{[0,x]}(x^\k) - x \right|.
 \label{eq:star_discrepancy}
\end{equation*}
For a cumulative distribution function (CDF) $F: \R \to [0,1]$, the empirical CDF of $M$ samples $\tilde{x}^{(1)},...,\tilde{x}^{(M)} \in \R$ is given by
\begin{equation*}
 F_M(x) := \frac{1}{M} \sum_{k=1}^M \mathds{1}_{\{\tilde{x}^\k \leq x\}} = \frac{1}{M} \sum_{k=1}^M \mathds{1}_{[\tilde{x}^\k,\infty)}(x) = \frac{1}{M} \sum_{k=1}^M \mathds{1}_{(-\infty,x]}(\tilde{x}^\k).
\end{equation*}
The $F$-discrepancy of $P$ is defined by, cf.~\cite{Fang1994_NumberTheoretic}, 
\begin{equation*}
 D_{F,P} = \sup_{x \in \R} |F_M(x) - F(x)|.
\end{equation*}
Considering the uniform distribution on $[0,1]$ with CDF
\begin{equation*}
 F_U(x) = \begin{cases}
         &0 \quadfor x < 0, \\
         &x  \quadfor 0 \leq x \leq 1, \\
         &1 \quadfor x > 1,
        \end{cases}
\end{equation*}
we observe that it holds
\begin{equation*}
  D_{F_U,P} = D_M^\ast(P).
\end{equation*}
Now, let $F_Y: \R \to [0,1]$, be a continuous distribution function to a random variable $Y$, where the inverse $F_Y^{-1}$ exists, is non-decreasing and continuous as well. Let $P_Y:=\{y^\k; k=1,\dots,M\}$ be a sequence of points obtained through inverse sampling of quasi-random numbers $x^\k$ via the inverse CDF $F_Y^{-1}$, i.e. $y^\k = F_Y^{-1}(x^\k)$ for all $k=1,\dots,M$. Further, assume that there exists $x \in [0,1]$ such that $F_Y(y) = x$ and $F_Y^{-1}(x) = y$ for every $y \in \R$. With these assumptions we compute
\begin{equation*}
 \begin{aligned}
    D_{F_Y,P_Y} & \ = \sup_{y \in \R} \left|F_{M,Y}(y) - F_Y(y)\right| = \sup_{y \in \R} \left|\frac{1}{M} \sum_{k=1}^M \mathds{1}_{\{y^\k \leq y\}} - F_Y(y)\right|\\
    & \ = \sup_{y \in \R} \left|\frac{1}{M} \sum_{k=1}^M \mathds{1}_{\{F_Y^{-1}(x^\k) \leq y\}} - F_Y(y)\right| = \sup_{x \in (0,1)} \left|\frac{1}{M} \sum_{k=1}^M \mathds{1}_{\{F_Y^{-1}(x^\k) \leq F_Y^{-1}(x)\}} - x\right| \\
    & \ = \sup_{x \in (0,1)} \left|\frac{1}{M} \sum_{k=1}^M \mathds{1}_{\{x^\k \leq x\}} - x\right|  =  D_{F_U,P} = D_M^\ast(P).
 \end{aligned}
\end{equation*}
Thus, the $F$-discrepancy for a continuous random variable with continuous inverse is equal to the star discrepancy, cf.~\cite{Fang1994_NumberTheoretic}.
The same holds true when considering the $F$-discrepancy of the tail distribution function $T_Y(y) = 1 - F_Y(y)$, because for the empirical tail distributions we have
\begin{equation*}
  \frac{1}{M} \sum_{k=1}^M \mathds{1}_{[-\infty,y^\k]}(y) = T_{M, Y}(y) = 1 - F_{M, Y}(y).
\end{equation*}
This leads to the following result, which is essential for the upcoming complexity analysis of the QCLMC method.
\begin{lemma}
\label{lma:inverse_transform_and_convergence}
 Let $(\Om, \calA, \P)$ be a complete probability space and $Y: \Om \to \R$ a real-valued random variable with continuous distribution function $F_Y$ and a continuous inverse distribution function $F_Y\inv$. For the distribution function it holds $F_Y(y) = \P(Y \leq y)$ and for the tail distribution function $T_Y(y) = 1 - F_Y(y) = \P(Y \geq y)$. Let $y^\k$ be a sequence generated via the inverse transformation
 \begin{equation*}
  y^\k := F_Y\inv(x^\k),
 \end{equation*}
 from a sequence $x^\k$ distributed in $[0,1)$ for $k=1,\dots,M$ and $M \in \N$. Then, the following convergence result for estimating the tail distribution $T_Y$ of the random variable $Y$ via the empirical tail distribution $T_{M,Y}$ holds:
 \begin{equation}
   \sup_{y \in \R} \left|\frac{1}{M} \sum_{k=1}^M \mathds{1}_{[-\infty, y^\k]}(y) - \P(Y \geq y) \right| \leq c_{disc} M^{\ka - 1}, 
   \label{eq:F_discrepancy_property}
 \end{equation}
for some $\ka \geq 0$ and $c_{disc} > 0$ independent of $M$. 
\end{lemma}
\begin{remark}
 \label{rem:F_discrepancy_exponential}

 The assumptions in Lemma \ref{lma:inverse_transform_and_convergence} hold for the specific case of an exponentially distributed random variable $Y \equiv L_r \sim \Exp(r)$ for some $r>0$. The distribution function is given by $\P(L_r \leq \l) = 1 - e^{-r\l}$ and the tail distribution function is $\P(L_r \geq \l) = e^{-r\l}$. The samples $(L_r^\k; k \in \N)$ are generated via the inverse transformation 
 \begin{equation}
  L_r^\k := - \frac{\ln(1 - x^\k)}{r}.
  \label{eq:exponential_inverse_transform}
 \end{equation}
 Thus, Lemma \ref{lma:inverse_transform_and_convergence} bounds the error of the tail estimate of an exponentially-distributed random variable, approximated by an empirical tail estimate. An illustration of this for an empirical tail estimate obtained by a quasi-random sequence compared to a pseudo-random sequence is given in Figure \ref{fig:distribution_convergence_discrepancy}.

\end{remark}

\begin{figure}
\includegraphics[width=0.49\textwidth]{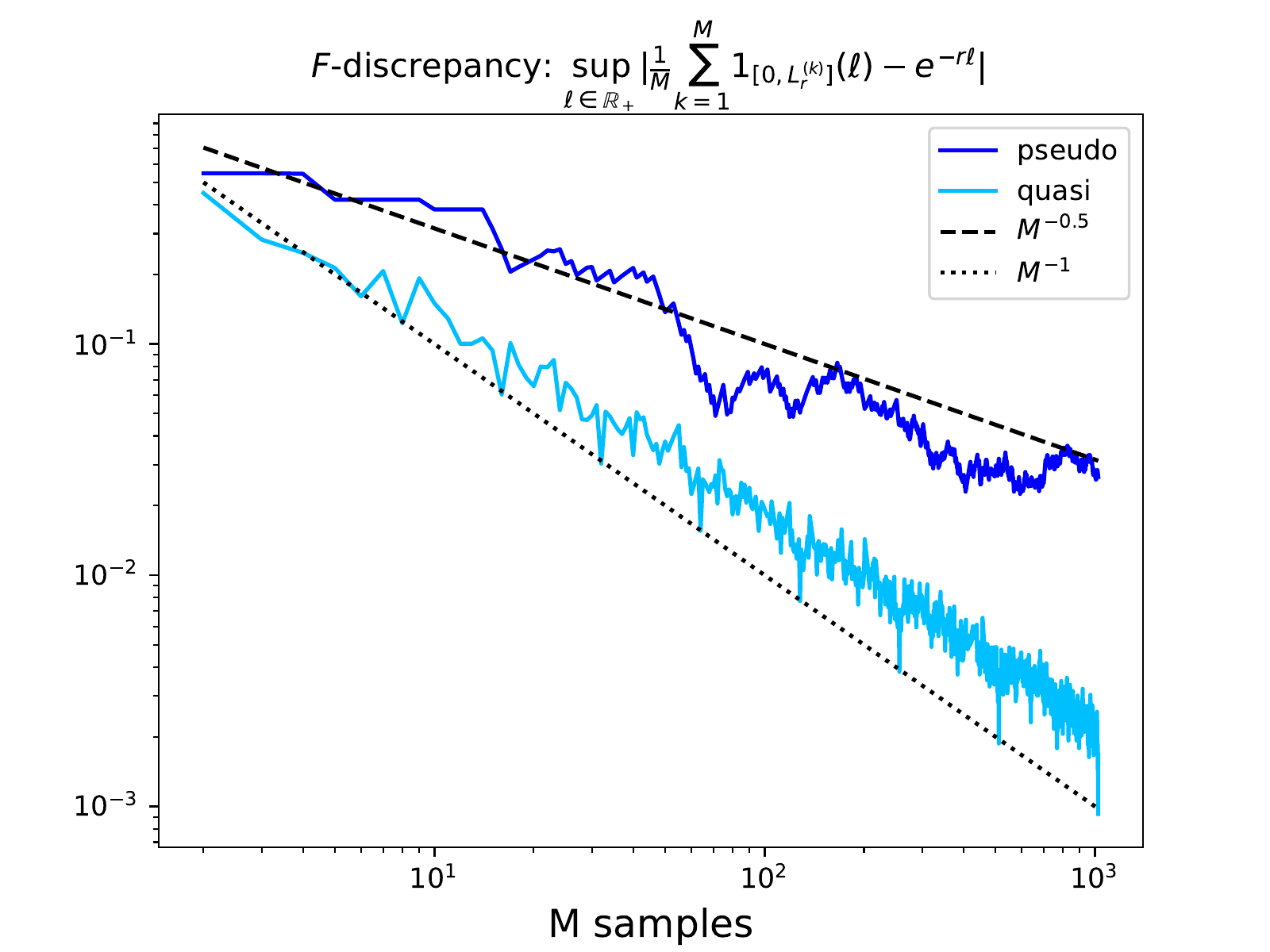}
\includegraphics[width=0.49\textwidth]{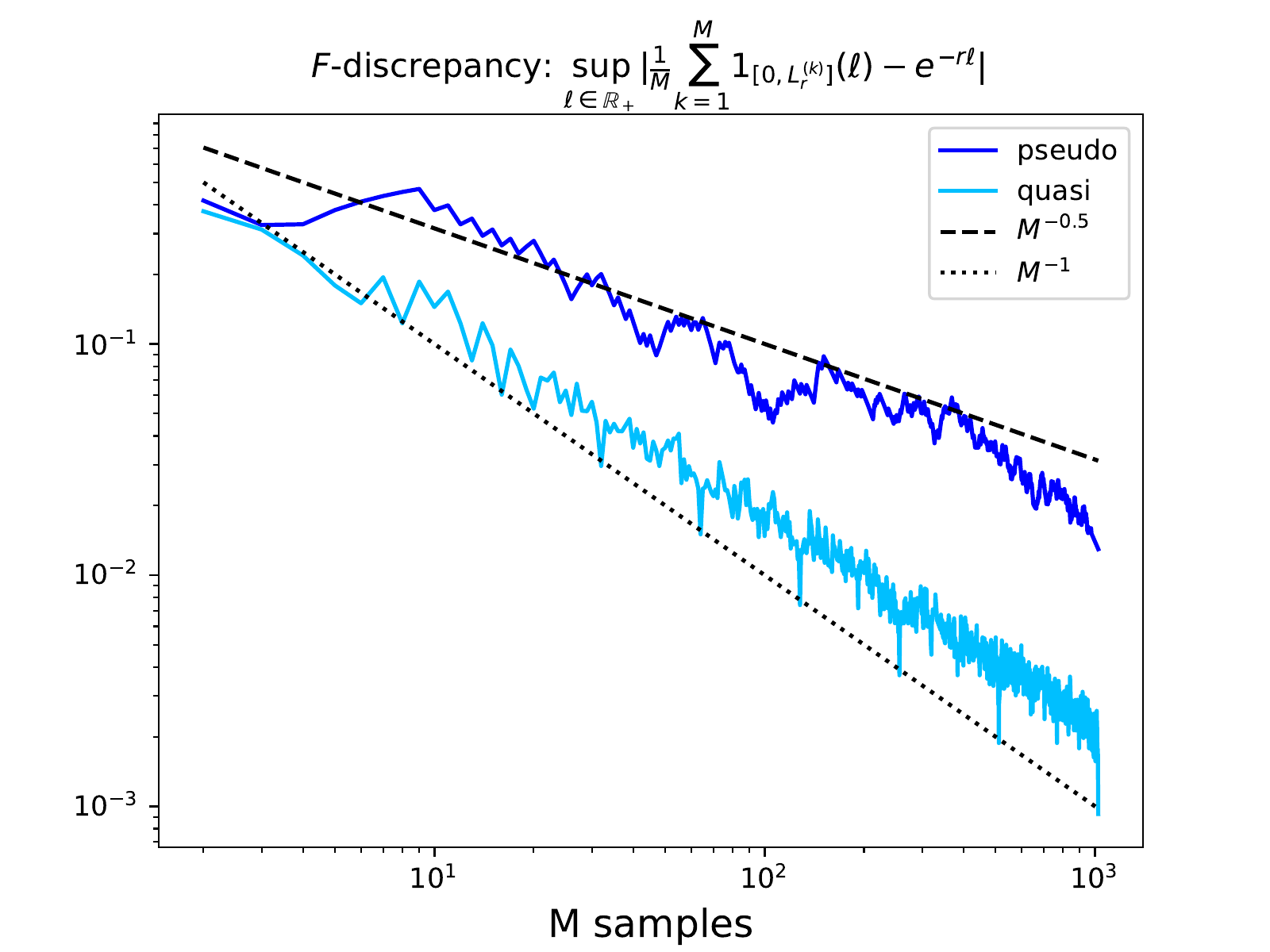}
\includegraphics[width=0.49\textwidth]{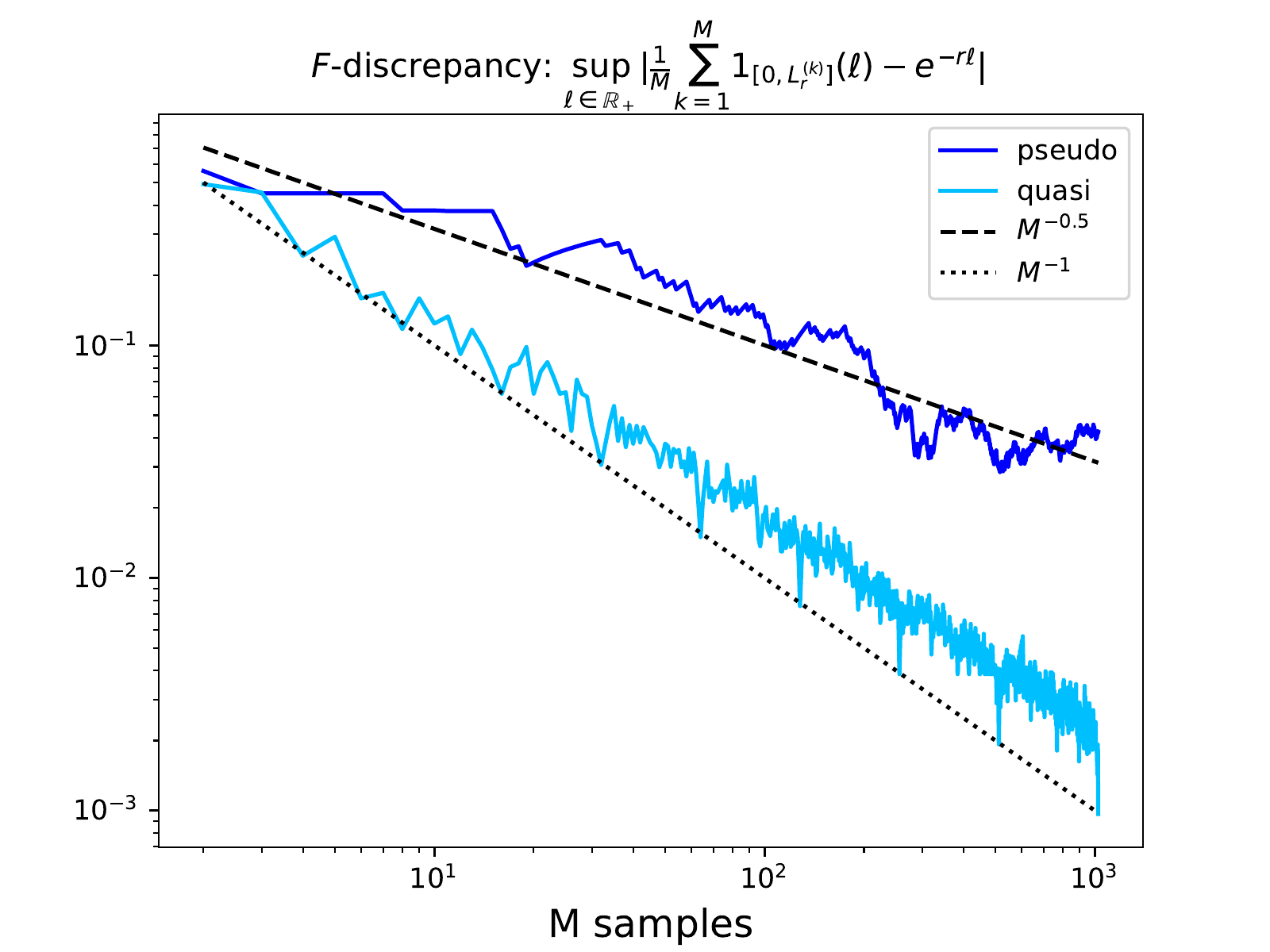}
\includegraphics[width=0.49\textwidth]{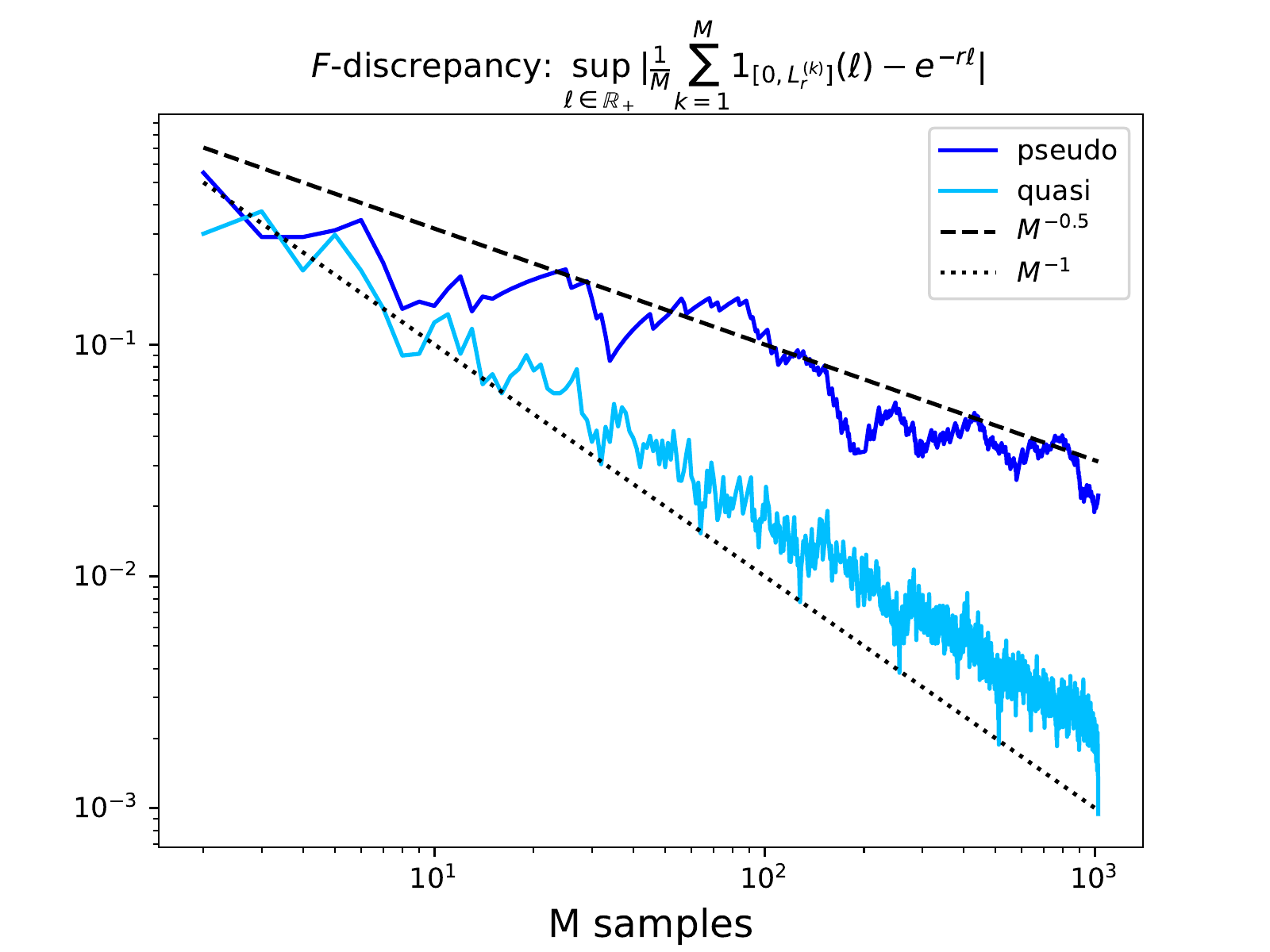}
 \caption{Demonstration of the convergence result of Lemma \ref{lma:inverse_transform_and_convergence} for $r = 1.3$ and four independent runs (different seeds) of quasi-random Sobol numbers \cite{Sobol1967_Sobol} with Owen scrambling, cf.~\cite{Owen1995_Scrambling, Owen1998_Scrambling} generated through the scipy library \cite{Scipy2020}, and pseudo-random numbers generated with the numpy library \cite{Numpy2020}. The quasi-random Sobol numbers have their optimal discrepancy for powers of two (location of downward spikes in the light blue lines), but we observe that for values in between powers of two, the discrepancy still converges with rate one, i.e. $\ka = 0$.}
 \label{fig:distribution_convergence_discrepancy}
\end{figure}

\begin{remark}
\label{rem:sequence_choices} 
  For the $F$-discrepancy result of Lemma \ref{lma:inverse_transform_and_convergence} the full rate of convergence with $\ka = 0$ is achieved in one dimension by Sobol sequences or Hammersley points, among others, cf.~\cite{Dick2014_discrepancy} for an overview. The sequence of numbers $(x^\k; k=1,\dots,M)$ is not restricted to quasi-random numbers in general. Any sequence from the interval $[0,1)$ may be used as long as the $F$-discrepancy convergence property with $\ka = 0$ is satisfied. A sequence of i.i.d. copies of a $[0,1]$-uniformly distributed random variable for $x^\k$ only yields $\ka = \frac{1}{2}$, cf. Figures~\ref{fig:distribution_convergence_discrepancy}.The grid points $x^\k = \frac{2k-1}{2M}$ for $k=1,\dots,M$ have $F$-discrepancy $\frac{1}{2} M^{-1}$, cf.~\cite[Corollary~$1.2$]{Kuipers1974_UniformDistribution} and \cite[Lemma~$1$]{Fang1994_NumberTheoretic}, which is the best achievable discrepancy in one dimension. However, this grid sequence is not nested and can not be reused for growing values of $M$. A great benefit of using quasi-random numbers is the possibility to sample one point after another in case the number of samples $M$ is not known a-priori, e.g., in an on-the-fly-type algorithm. 
 
  Another great benefit of using quasi-random numbers for the low-discrepancy sequence is the possibility of creating independent low-discrepancy sequences via randomization, e.g., for Sobol numbers via Owen scrambling \cite{Owen1995_Scrambling, Owen1998_Scrambling}, to obtain independent QCLMC estimators to estimate the MSE over several QCLMC runs as done in the numerical experiments in Section \ref{subsec:performance_comparison_of_clmc_and_qclmc}. 
  
  Furthermore, the continuous level framework may be extended to multiindex Monte Carlo (MIMC), cf.~\cite[Section $5$]{Detommaso2019_CLMC} for details. In this case, the level variable follows a multivariate probability distribution and quasi-random sequences with an optimal discrepancy property in higher dimensions are beneficial for such an extension of QCLMC. They obtain merely an additional logarithmic dependence on the dimension in the $F$-discrepancy convergence result. 
%
\end{remark}

\section{Quasi continuous level Monte Carlo method}
\label{sec:QCLMC}
As for the CLMC method in Section~\ref{sec:CLMC} we assume that for the level parameter $\l > 0$ we have approximations to the quantity of interest $(Q(\l); \l>0)$, the deterministic maximal level is given by $\Lm \in (0,\infty]$, $M \in \N$ is the total number of samples and $L_{r}$ is a random variable with finite expectation, $\P(L_r \geq \l) > 0$ and that is independent of the stochastic process $(Q(\l))_{\l \geq 0}$. Then, the QCLMC estimator is defined by
\begin{equation}
\begin{aligned}
     \widehat{Q}_{0, \Lm}^\QCLMC = \frac{1}{M} \sum_{k=1}^M \int_0^{\Lm \land \bar{L}} \frac{1}{\P(L_r \geq \l)} \left(\dQl\right)^\k \; \mathds{1}_{[0, L_r^\k]} (\l) \dl,
\end{aligned}
\label{eq:QCLMC}
\end{equation}
where furthermore and most importantly $(L_r^\k; k=1,\dots,M)$ is a deterministic sequence obtained via inverse transformation (see Lemma~\ref{lma:inverse_transform_and_convergence}) with $\ka = 0$ and $\bar{L} := \max\{L_r^\k;\;k=1,\dots,M\}$. Note, that the sequence $(L_r^\k; k=1,\dots,M)$ does not consist of i.i.d. copies of the random variable $L_r$. 

For showing the unbiasedness result of the QCLMC estimator and the complexity theorem we choose $L_r \sim \Exp(r)$ with parameter $r > 0$ as for the CLMC estimator and note that this is also not the only choice in QCLMC. The distribution generally has to satisfy the assumptions in Remark~\ref{rem:choice_L_r} and in Lemma \ref{lma:inverse_transform_and_convergence}. 

Since we deal with a deterministic sequence for $(L_r^\k;\;k=1,\dots,M)$ the use of any distributional properties of random sequences is not reasonable here anymore. Thus, additionally to the expectation $\E[\widehat{Q}_{0, \Lm}^\QCLMC]$ of the estimator we consider the limit $M \to \infty$, see, e.g., \cite[Chapter $3$]{Gentle2003_RandomNumberGeneration}, when investigating the unbiasedness of the QCLMC estimator. Note, that the samples $(\dQl^\k,\;k=1,\dots,M)$ are still i.i.d. copies of the random variable $\dQl$ in QCLMC.
\begin{proposition}
\label{prop:QCLMC_unbiased}
 Assume that $\bar{L} = \ln\left(\tilde{c}^{-\frac{1}{r}} M^{\frac{1 - \ka}{r}}\right)$ with $0 \leq \ka < 1$, $r > 0$ and a constant $\tilde{c} > 0$ independent of $M \in \N$ and suppose there exist positive constants $\al$ and $c_1$ such that for any $\l > 0$
 \begin{equation}
  \left|\E \left[\dQll\right]\right| \leq c_1 e^{-\al \l}.
  \label{eq:QCLMC_meandecay_prop}
 \end{equation}
 Then, in the limit $M \to \infty$, the QCLMC estimator \eqref{eq:QCLMC} is an unbiased estimator for $\E[Q(\Lm) - Q(0)]$, i.e.,
 \begin{equation*}
  \lim_{M \to \infty } \E[\widehat{Q}_{0, \Lm}^\QCLMC] = \E[Q(\Lm) - Q(0)].
 \end{equation*}
 If $\Lm = \infty$ it holds $Q(\Lm) = Q(\infty) = \calQ$ and the QCLMC estimator is an unbiased estimator to the real quantity of interest $\E[\calQ - Q(0)]$.
\end{proposition}
\begin{proof}
First, we compute an equality for the expectation of the estimator \eqref{eq:QCLMC} by adding zero in a suitable way and using that $\P(L_r \geq \l) = e^{-r\l}$ 
\begin{equation*}
 \begin{aligned}
   \E[\widehat{Q}_{0, \Lm}^\QCLMC] = & \ \E\left[\frac{1}{M} \sum_{k=1}^M \int_0^{\Lm \land \bar{L}} e^{r \l} \left(\dQl\right)^\k (\l) \; \mathds{1}_{[0, L_r^\k]} (\l) \dl\right] \\
  = & \  \int_0^{\Lm \land \bar{L}} e^{r \l} \, \E\left[\dQll\right] \; \frac{1}{M} \sum_{k=1}^M \mathds{1}_{[0, L_r^\k]} (\l) \dl \\
  = & \ \int_0^{\Lm \land \bar{L}} \left(e^{r \l} \frac{1}{M} \sum_{k=1}^M \mathds{1}_{[0, L_r^\k]} (\l) - \frac{e^{-r\l}}{e^{-r\l}} + 1 \right) \E\left[\dQll\right] \;  \dl \\
  = & \ \int_0^{\Lm \land \bar{L}} e^{r \l} \left(\frac{1}{M} \sum_{k=1}^M \mathds{1}_{[0, L_r^\k]} (\l) - e^{-r\l}\right) \E\left[\dQll\right] \dl +  \int_0^{\Lm \land \bar{L}}  \E\left[\dQll\right] \;  \dl  =: I + II.
 \end{aligned}
\end{equation*}
We bound the integrand of $I$ from below and above. By the $F$-discrepancy property \eqref{eq:F_discrepancy_property} and Assumption \eqref{eq:QCLMC_meandecay_prop} it holds for every $\l > 0$
\begin{equation*}
- c_1 c_{disc} M^{\ka - 1} e^{(r-\al) \l} \leq e^{r\l} \left(\frac{1}{M} \sum_{k=1}^M \mathds{1}_{[0, L_r^\k]}(\l) - e^{-r\l}\right)\E \left[\dQll\right] \leq c_1 c_{disc} M^{\ka - 1} e^{(r-\al) \l}.
\end{equation*}
We integrate the lower and upper bound from zero to $\Lm \land \bar{L}$ to obtain
\begin{equation*}
  \pm c_1 c_{disc} M^{\ka - 1} \int_0^{\Lm \land \bar{L}} e^{(r-\al) \l} \dl = \pm c_1 c_{disc} M^{\ka - 1} \begin{cases}
          \frac{1}{r-\al}\left(e^{(r-\al)(\Lm \land \bar{L})} -1 \right) &\quadfor r \neq \al, \\
          \Lm \land \bar{L} &\quadfor r = \al.
        \end{cases}
\end{equation*}
Note that $\lim_{M \to \infty} \bar{L} = \lim_{M \to \infty} \ln\left(\tilde{c}^{-\frac{1}{r}} M^\frac{1 - \ka}{r}\right) = \infty$ and thus, in the limit $M \to \infty$, for finite $\Lm < \infty$ the minimum $\Lm \land \bar{L}$ is attained for $\Lm$ and we trivially obtain $\lim_{M \to \infty} I = 0$ by the squeeze theorem and $\lim_{M \to \infty} II = \E[Q(\Lm) - Q(0)]$.
If $\Lm = \infty$ we compute in the case $r = \al$
\begin{equation*}
 \lim_{M \to \infty} \pm c_1 c_{disc} M^{\ka - 1} (\Lm \land \bar{L}) = \lim_{M \to \infty}  \pm c_1 c_{disc} M^{\ka - 1} \ln\left(\tilde{c}^{-\frac{1}{r}} M^\frac{1 - \ka}{r}\right) = 0,
\end{equation*}
for all $0 \leq \ka < 1$. For $r \neq \al$ we obtain
\begin{equation*}
\begin{aligned}
 \lim_{M \to \infty} \pm c_1 c_{disc} M^{\ka - 1} \frac{1}{r-\al}\left(e^{(r-\al)\Lm \land \bar{L}} -1 \right) & \  = \lim_{M \to \infty}  \pm c_1 c_{disc} M^{\ka - 1} \frac{1}{r-\al}\left(e^{(r-\al)\ln\left(\tilde{c}^{-\frac{1}{r}} M^\frac{1 - \ka}{r}\right)} -1 \right) \\
 & \ = \lim_{M \to \infty} \pm \left(\frac{c_1 c_{disc} \tilde{c}^{-\frac{r-\al}{r}}}{r-\al} M^{\ka - 1}M^\frac{(r - \al)(1 - \ka)}{r} - \frac{c_1 c_{disc}}{r-\al} M^{\ka - 1} \right)\\
 & \ = \lim_{M \to \infty} \pm \left(\frac{c_1 c_{disc} \tilde{c}^{-\frac{r-\al}{r}}}{r-\al} M^{(\ka - 1)\frac{\al}{r}} - \frac{c_1 c_{disc}}{r-\al} M^{\ka - 1} \right) = 0,
 \end{aligned}
\end{equation*}
for all $0 \leq \ka < 1$ and $\al > 0$. The squeeze theorem again yields $\lim_{M \to \infty} I = 0$ and $\lim_{M \to \infty} II = \E[Q(\infty) - Q(0)] = \E[\calQ - Q(0)]$.
Thus, in any case we obtain in the limit $M \to \infty$ the final result
\begin{equation*}
 \lim_{M \to \infty} \E[\widehat{Q}_{0, \Lm}^\QCLMC] = \E[Q({\Lm}) - Q(0)],
\end{equation*}
for all $\Lm \in (0, \infty]$.
\end{proof}
\begin{remark}
 \label{rem:max_L_r}
 For any sequence $x^\k$ in $[0,1)$ satisfying the discrepancy convergence property \ref{eq:discr_conv} for the star-discrepancy, it holds $\max\{x^\k; k=1,\dots,M\} = 1 - \tilde{c} M^{\ka - 1}$ for a constant $0 < \tilde{c} \leq c_{disc}$ independent of $M \in \N$ and we obtain for the sequence $L_r^\k$ from Remark \ref{rem:F_discrepancy_exponential} that
\begin{equation}
\begin{aligned}
 \bar{L} := \max\{L_r^\k; k=1,\dots,M\} = & \  -\frac{\ln(1 - \max\{x^\k; k=1,\dots,M\}}{r} \\
 \ & =  -\frac{\ln(1 - (1 - \tilde{c} M^{\ka - 1}))}{r} = \ln\left(\tilde{c}^{-\frac{1}{r}} M^\frac{1 - \ka}{r}\right),
 \end{aligned}
 \label{eq:L_r_bound}
\end{equation}
 showing that the assumption from Proposition \ref{prop:QCLMC_unbiased} is satisfied.
\end{remark}

Next, we prove a complexity theorem for the new QCLMC estimator with explicit treatment of the $(L_r^\k; k=1,\dots, M)$ as a deterministic quasi-random sequence with $\ka = 0$.
\begin{theorem}[QCLMC - complexity theorem]
\label{thm:QCLMC-complexity}
Denote by $(Q(\l);\l \geq 0)$ a stochastic process defined on a probability space $({\Omega}, {\calA}, {\P})$ with $\E\left[\left|\dQl\right|\right] \in L^1((0,\infty);\R)$, corresponding to a family of numerical approximations of $\calQ$ such that $Q(\l) \to \calQ$ as $\l \to \infty$ ${\P}\textup{-almost surely}$. Suppose there exist positive constants $\al$, $\be$, $\ga$, $c_1$, $c_2$, $c_3$ such that for any $\l > 0$:
    \begin{subequations}
        \begin{equation}
                \left|\E \left[\dQll\right]\right| \leq c_1 e^{-\al \l},
            \label{eq:QCLMC_meandecay}
        \end{equation}   
        \begin{equation}
            \V \left[\dQll\right] \leq c_2 e^{-\be \l}, 
            \label{eq:QCLMC_variancedecay}
        \end{equation}   
        \begin{equation}
                \frac{\text{d} \calC[\l]}{\dl} \leq c_3 e^{\ga \l},
            \label{eq:QCLMC_costincrease}
        \end{equation}   
    \end{subequations}
    where $\calC[\l]$ is the total accumulated cost to compute a sample of $Q(\l)$. Further, let $(L_r^\k; k=1,\dots,M)$ be a deterministic quasi-random sequence obtained by inverse transformation (see Lemma~\ref{lma:inverse_transform_and_convergence} with $\ka =  0$), and let $r \in [\min\{\be, 2\al, \ga\}, \max\{\min\{\be,2\al\}, \ga\}]$. Then, there exist $\Lm \in (0, \infty]$ and $M \in \N$, such that for any $\eps \in (0, \frac{1}{e})$ it holds
    \begin{equation*}
        \MSE^\QCLMC  \leq \eps^2 \quad \hbox{ and }\quad \calC[\widehat{Q}_{0, \Lm}^\QCLMC] \leq C \eps^{-2 - \max\{0, \frac{\ga - \min\{\be,2\al\}}{\al} \}} |\ln(\eps)|^{\delta_{r, \be} + \delta_{r, \ga}},
    \end{equation*}
    where $\delta$ is the Dirac function and $C > 0$ is independent of $\Lm, M$ and $\eps$.
\end{theorem}
 Before stating the proof let us note, that with an appropriate choice of $r$, the complexity of the QCLMC estimator given in Theorem \ref{thm:QCLMC-complexity} is the same as for the CLMC estimator given in Equation \eqref{eq:CLMC-complexity} and the MLMC estimator, cf.~\cite{Giles2008_MLMCPath, Giles2015_MLMC}, but with a potentially lower constant in the upper bound to the cost. This potential is investigated for the CLMC and QCLMC method in the numerical experiments in Section \ref{sec:numerical_experiments}.
\begin{proof}

To deliberately use the $F$-discrepancy property of the quasi-random sequence, the proof is based on the standard decomposition of the MSE 
\begin{equation*}
 \MSE_{0,\Lm}^\QCLMC = \V[\widehat{Q}_{0,\Lm}^\QCLMC] + \E[\widehat{Q}_{0, \Lm}^\QCLMC - (\calQ - Q(0))]^2,
 \label{eq:qclmc_mse}
\end{equation*}
into a variance and squared bias term, which is to be bounded by $\eps^2$ for a given $0 < \eps < e\inv$. The proof is split into three parts: In the first part we compute the squared bias of the QCLMC estimator and divide it into a term depending only on $\Lm \land \bar{L}$ and terms depending on $\Lm \land \bar{L}$ and $M$, but each of these terms again depends differently on the total sample size $M$. We choose $\Lm \land \bar{L}$ to bound the first term by $\frac{\eps^2}{2}$ and $M$ to bound the remaining terms by $\frac{\eps^2}{4}$. Then, in the second part we bound the variance of the QCLMC estimator in terms of $\frac{\eps^2}{4}$ by an appropriate choice of $M$. Let us note here, that in each of these bounds appear two types of terms again, each depending differently on the total sample size $M$. Finally, in the third part we bound the total cost of the QCLMC estimator with the aggregated choices of $M$ from the first two parts, finishing the proof. Optimized MSE splits based on the problem parameters $\al, \be, \ga, c_1, c_2, c_3$ may as well be obtained, but are omitted for simplicity of notation.
A key concept in the proof is the insertion of a zero in the integral quantities (as done in the proof of Proposition \ref{prop:QCLMC_unbiased}) in order to use the $F$-discrepancy property of the quasi-random sequence.
Utilizing that $L_r \sim \Exp(r)$, we the estimate
\begin{equation}
 \begin{aligned}
\left|\E\left[\widehat{Q}_{0, \Lm}^\QCLMC - (\calQ - Q(0))\right]\right| \leq & \  \left|\E[\calQ - Q(\Lm \land \bar{L})]\right| + \left|\E[\widehat{Q}_{0, \Lm}^\QCLMC] - \E[Q({\Lm} \land \bar{L}) - Q(0)]\right|\\
\leq & \ \int_{\Lm \land \bar{L}}^\infty \left|\E\left[\frac{\dQ(\l)}{\dl}\right]\right| \dl + \frac{c_{disc}}{M} \int_0^{\Lm \land \bar{L}} e^{r \l} \left|\E\left[\dQll\right]\right| d\l  \\  
\leq & \ c_1 \int_{\Lm \land \bar{L}}^\infty e^{-\al \l} \dl  + \frac{c_{disc}}{M} \,c_1 \int_0^{\Lm \land \bar{L}} e^{(r-\al)\l} d\l \\
= & \ \frac{c_1}{\al} e^{-\al (\Lm  \land \bar{L})} + \frac{c_{disc}\,c_1}{M}  \cdot      \begin{cases}
          \frac{1}{r-\al}\left(e^{(r-\al)(\Lm  \land \bar{L})} -1 \right) &\quadfor r \neq \al, \\
           \Lm \land \bar{L} &\quadfor r = \al,
        \end{cases} 
 \end{aligned}
 \label{eq:qclmc_bias_est}
\end{equation}
with explicit dependence on $c_{disc}, c_1, \al, r, M, \bar{L}$ and $\Lm$. We restrict to the case $r = \al$ and compute for the bound of the squared bias
\begin{equation*}
\begin{aligned}
  \left|\E\left[\widehat{Q}_{0, \Lm}^\QCLMC - (\calQ - Q(0))\right]\right|^2 & \ \leq \frac{c_1^2}{\al^2} e^{-2\al (\Lm  \land \bar{L})} + \frac{c_{disc}}{M} \frac{2c_1^2}{\al} (\Lm  \land \bar{L}) e^{-\al (\Lm  \land \bar{L})} + \frac{c_{disc}^2}{M^2} c_1^2 (\Lm  \land \bar{L})^2.
  \end{aligned}
\end{equation*}
Next, we consider the other case $r \neq \al$, which yields
\begin{equation*}
\begin{aligned}
   \left|\E\left[\widehat{Q}_{0, \Lm}^\QCLMC - (\calQ - Q(0))\right]\right|^2\leq \frac{c_1^2}{\al^2} e^{-2\al (\Lm  \land \bar{L})} & \ + \frac{c_{disc}}{M} \frac{2c_1^2}{\al} \frac{1}{r-\al}\left(e^{(r-2\al)(\Lm  \land \bar{L})} - e^{-\al (\Lm  \land \bar{L})}\right) \\
   & \ + \frac{c_{disc}^2}{M^2} c_1^2 \frac{1}{(r-\al)^2}\left(e^{2(r-\al)(\Lm  \land \bar{L})} + 1 \right),
  \end{aligned}
\end{equation*}
where we bounded strictly negative terms by zero from above. 
As the upper bounds to the squared bias we define the functions 
\begin{equation*}
 \begin{aligned}
  B_1(\Lm  \land \bar{L}) := \frac{c_1^2}{\al^2} e^{-2\al (\Lm  \land \bar{L})},
 \end{aligned}
\end{equation*}
which depends only on $(\Lm  \land \bar{L})$ and
\begin{equation*}
 \begin{aligned}
B_2(\Lm  \land \bar{L}, M) := \begin{cases}
                \frac{c_{disc}^2}{M^2} c_1^2 (\Lm  \land \bar{L})^2 + \frac{c_{disc}}{M} \frac{2c_1^2}{\al} (\Lm  \land \bar{L}) e^{-\al (\Lm  \land \bar{L})} &\quadfor r = \al, \\
                \frac{c_{disc}^2}{M^2} c_1^2 \frac{1}{(r-\al)^2}\left(e^{2(r-\al)(\Lm  \land \bar{L})} + 1 \right) \\
                \qquad + \frac{c_{disc}}{M} \frac{2c_1^2}{\al} \frac{1}{r-\al}\left(e^{(r-2\al)(\Lm  \land \bar{L})} - e^{-\al (\Lm  \land \bar{L})}\right) &\quadfor r \neq \al,
              \end{cases}
 \end{aligned}
\end{equation*}
which depends on $(\Lm  \land \bar{L})$ and $M$. In order for $B_1$ to equal $\frac{\eps^2}{2}$, we choose $(\Lm  \land \bar{L})$ such that
\begin{equation}
 \begin{aligned}
  B_1(\Lm  \land \bar{L}) = \frac{c_1^2}{\al^2} e^{-2\al (\Lm  \land \bar{L})} = \frac{\eps^2}{2} \quad \text{ leading to } \quad \Lm  \land \bar{L} = \frac{1}{\al} \ln\left(\frac{c_1 \, \sqrt{2}}{\al \eps}\right).
 \end{aligned}
 \label{eq:qclmc_L_max}
\end{equation}
To avoid complicated case distinctions we set $\Lm = \frac{1}{\al} \ln\left(\frac{c_1 \, \sqrt{2}}{\al \eps}\right)$ and choose $\bar{L} = \ln(\tilde{c}^{-\frac{1}{r}}M^{\frac{1}{r}}) \geq \frac{1}{\al} \ln\left(\frac{c_1 \, \sqrt{2}}{\al \eps}\right)$ such that $\Lm  \land \bar{L} = \Lm$ leading to 
\begin{equation}
 M \geq \tilde{c} \left(\frac{c_1 \, \sqrt{2}}{\al}\right)^\frac{r}{\al} \eps^{-\frac{r}{\al}} = C_0\, \eps^{-\frac{r}{\al}},
 \label{eq:lower_bound_M_discr}
\end{equation}
for a constant $C_0 := \tilde{c} \left(\frac{c_1 \, \sqrt{2}}{\al}\right)^\frac{r}{\al}$ independent of $M$, $\Lm$ and $\eps$. Later, we see that this choice of $M$ subsumes in with the other choices of $M$ in the rest of this proof and is not restrictive with respect to its dependence on $\eps$.

Using the definition for $\Lm$, we can estimate terms depending on $\Lm$ in terms of $\eps$. For $\eps < e\inv$ it holds $|\ln(\eps)| > 1$ and thus we compute
\begin{equation}
\begin{aligned}
  \Lm = \frac{1}{\al} \ln\left(\frac{c_1 \, \sqrt{2}}{\al \eps}\right) & \ = \frac{1}{\al} \ln\left(\frac{c_1 \, \sqrt{2}}{\al} \eps\inv\right) \\
  & \ = \frac{1}{\al} \ln\left(\frac{c_1 \, \sqrt{2}}{\al} \right)|\ln(\eps)| + \frac{1}{\al}|\ln(\eps)| \leq \frac{1}{\al}\max\left\{\ln\left(\frac{c_1 \, \sqrt{2}}{\al} \right), 1\right\} |\ln(\eps)| = C_1  |\ln(\eps)|, 
  \end{aligned}
   \label{eq:qclmc_helping_bounds_1}
\end{equation}
with $C_1:= \frac{1}{\al}\max\left\{\ln(\frac{c_1 \, \sqrt{2}}{\al} ), 1\right\} > 0$ independent of $M$, $\Lm$ and $\eps$.
Further, we get for a scalar $\rho \in \R$
\begin{equation}
 e^{\rho \Lm} = e^{\frac{\rho}{\al} \ln\left(\frac{c_1 \, \sqrt{2}}{\al \eps}\right)} = \left(\frac{c_1 \, \sqrt{2}}{\al \eps}\right)^\frac{\rho}{\al} = C_2 \eps^{-\frac{\rho}{\al}},
  \label{eq:qclmc_helping_bounds_2}
\end{equation}
with a constant $C_2:= \left(\frac{c_1 \, \sqrt{2}}{\al}\right)^\frac{\rho}{\al} > 0$ independent of $M$, $\Lm$, and $\eps$. Simply combining both relations we obtain
\begin{equation}
\begin{aligned}
  \Lm e^{\rho \Lm} \leq C_1 C_2  \eps^{-\frac{\rho}{\al}}|\ln(\eps)|.
  \end{aligned}
   \label{eq:qclmc_helping_bounds_3}
\end{equation}
Furthermore, exponential terms with negative exponent are bounded by one and negative terms are bounded by zero from above.

With these upper bounds at hand we distinct between five different relations between $r$ and $\al$. Recall, that we wish to bound $B_2(\Lm, M)$ in each case against $\frac{\eps^2}{4}$ by using $\Lm$ from Equation $\eqref{eq:qclmc_L_max}$ and choosing $M$ accordingly. 
\\
\textbf{Case 1: $(r < \al)$}
\begin{equation*}
 \begin{aligned}
  B_2(\Lm, M) = & \  \frac{c_{disc}^2}{M^2} c_1^2  \frac{1}{(r-\al)^2}\left(e^{2(r-\al)\Lm} + 1 \right)  + \frac{c_{disc}}{M} \frac{2c_1^2}{\al}\frac{1}{r-\al}\left(e^{(r-2\al)\Lm} - e^{-\al \Lm}\right) \leq C_3 \frac{1}{M}.
  \end{aligned}
\end{equation*}
\\
\textbf{Case 2: $(r = \al)$}
\begin{equation*}
 \begin{aligned}
  B_2(\Lm, M) = & \  \frac{c_{disc}^2}{M^2} c_1^2\Lm^2  + \frac{c_{disc}}{M} \frac{2c_1^2}{\al}  \Lm e^{-\al \Lm} \leq C_4 \left(\frac{\ln(\eps)^2}{M} + 1 \right) \frac{1}{M}.
  \end{aligned}
\end{equation*}
\\
\textbf{Case 3: $(\al < r < 2 \al)$}
\begin{equation*}
 \begin{aligned}
  B_2(\Lm, M) = & \  \frac{c_{disc}^2}{M^2} c_1^2  \frac{1}{(r-\al)^2}\left(e^{2(r-\al)\Lm} + 1 \right)  + \frac{c_{disc}}{M} \frac{2c_1^2}{\al}\frac{1}{r-\al}\left(e^{(r-2\al)\Lm} - e^{-\al \Lm}\right) \\
  \leq & \ C_5 \left(\frac{1}{M}\eps^{-2} + 1\right) \frac{1}{M}.
  \end{aligned}
\end{equation*}
\\
\textbf{Case 4: $(r = 2\al)$}
\begin{equation*}
 \begin{aligned}
  B_2(\Lm, M) = & \  \frac{c_{disc}^2}{M^2} c_1^2  \frac{1}{\al^2}\left(e^{2\al\Lm} + 1 \right) \frac{c_{disc}}{M} \frac{2c_1^2}{\al^2}\left(1 - e^{-\al \Lm}\right) \leq C_6 \left(\frac{1}{M}\eps^{-2} + 1\right)\frac{1}{M}.
  \end{aligned}
\end{equation*}
We note that no logarithmic $\eps$ contribution appears in the choice for the sample size $M$ for $r = 2\al$, as opposed to the complexity theorem of the standard CLMC estimator.
\\
\textbf{Case 5: $(r > 2\al)$}
\begin{equation*}
 \begin{aligned}
  B_2(\Lm, M) = & \  \frac{c_{disc}^2}{M^2} c_1^2  \frac{1}{(r-\al)^2}\left(e^{2(r-\al)\Lm}  + 1 \right)  + \frac{c_{disc}}{M} \frac{2c_1^2}{\al}\frac{1}{r-\al}\left(e^{(r-2\al)\Lm} - e^{-\al \Lm}\right) \\
  \leq & \  C_7\left(\frac{\eps^{-\frac{r}{\al}}}{M} + 1\right) \frac{\eps^{-\frac{r - 2\al}{\al}}}{M}.
  \end{aligned}
\end{equation*}
All constants $C_3, C_4, C_5, C_6, C_7 > 0$ are independent of $M$, $\Lm$ and $\eps$.
Overall, we obtain 
\begin{equation}
 \begin{aligned}
  B_2(\Lm, M) \leq C_{B_2} \frac{1}{M} 
                  \begin{cases}
                      1 & \quadfor r < \al, \\
                      \frac{\ln(\eps)^2}{M} + 1 & \quadfor r = \al, \\
                      \frac{\eps^{-2}}{M}+ 1  & \quadfor \al < r \leq 2\al, \\
                        \left(\frac{\eps^{-\frac{r}{\al}}}{M} + 1\right)\eps^{-\frac{r - 2\al}{\al}}  & \quadfor r > 2\al. \\
                   \end{cases}
  \end{aligned}
  \label{eq:upper_bound_bias}
\end{equation}
with a constant  $C_{B_2} := \max\{C_3,\,C_4,\,C_5,\,C_6,\,C_7\} > 0$  of $M$, $\Lm$ and $\eps$. Further to bound $B_2$ by $\frac{\eps^2}{4}$ by an appropriate choice of $M$, in all cases $r \geq \al$ the additional contributions from $\eps$ need to be compensated for by the additional factor of $\frac{1}{M}$. We demonstrate this for the case $r > 2\al$ and start by choosing 
\begin{equation*}
 M \geq C_{tmp} \eps^{-2 - \frac{r - 2\al}{\al}}
\end{equation*}
for a constant $C_{tmp}$ to be chosen subsequently. Then, we compute
\begin{equation*}
\begin{aligned}
 C_7\left(\frac{\eps^{-\frac{r}{\al}}}{M} + 1\right) \frac{\eps^{-\frac{r - 2\al}{\al}}}{M} \leq & \ C_7\left(\frac{\eps^{2 + \frac{r - 2\al}{\al}} \eps^{-\frac{r}{\al}}}{C_{tmp}} + 1\right) \frac{\eps^{2 + \frac{r - 2\al}{\al}} \eps^{-\frac{r - 2\al}{\al}}}{C_{tmp}} \\
 = & \  C_7\left(\frac{\eps^{2 + \frac{r - 2\al - r}{\al}}}{C_{tmp}} + 1\right) \frac{\eps^{2}}{C_{tmp}} =  C_7\left(\frac{1}{C_{tmp}} + 1\right) \frac{\eps^{2}}{C_{tmp}},
 \end{aligned}
\end{equation*}
and choose $C_{tmp} := 2 C_7 \left(1 + \sqrt{1 + \frac{1}{C_7}}\right)$ as the solution to the quadratic equation $C_7\left(\frac{1}{C_{tmp}} + 1 \right)\frac{\eps^2}{C_{tmp}} = \frac{\eps^2}{4}$.
Overall, we obtain $B_2 \leq \frac{\eps^2}{4}$, dealing with all different relations between $r$ and $\al$ by choosing $M$ to be
\begin{equation}
 M \geq \widetilde{C}_{B_2} \eps^{-2- \max\{0, \frac{r - 2\al}{\al}\}},
 \label{eq:lower_bound_M_bias}
\end{equation}
with an appropriate constant $\widetilde{C}_{B_2} > 0$ of $M$, $\Lm$ and $\eps$. The initial lower bound for $M$ in Equation \eqref{eq:lower_bound_M_discr} to obtain $\Lm \land \bar{L} = \Lm$ is satisfied by adapting the constant $\widetilde{C}_{B_2}$ to be larger than $C_0$, since it holds $\eps^{-2- \max\{0, \frac{r - 2\al}{\al}\}} \geq \eps^{-\frac{r}{\al}}$ independent of the relation between $r$ and $\al$.

Next, we continue with the bound for the variance term. Note, that the copies of the stochastic process of approximations $\left(Q(\l)^\k; \l \geq 0\right)$ as well as $\left(\dQll^\k; \l > 0\right)$ are i.i.d. for $1 \leq k \leq M$. 
Using the Fubini--Tonelli theorem to exchange the covariance with the integration over the level domain we obtain by linearity of integration and bilinearity of the covariance
\begin{equation}
 \begin{aligned}
    \V & \left[\widehat{Q}_{0, \Lm}^\QCLMC\right] = \Cov\left[\widehat{Q}_{0, \Lm}^\QCLMC, \widehat{Q}_{0, \Lm}^\QCLMC\right] \\
  =& \ \Cov\left[\frac{1}{M} \sum_{k=1}^M \int_0^{\Lm} e^{r \l} \left(\dQl\right)^\k (\l) \; \mathds{1}_{[0, L_r^\k]} (\l) \dl, \frac{1}{M} \sum_{m=1}^M \int_0^{\Lm} e^{r \l'} \left(\frac{\text{d} Q}{\text{d}\l'}\right)^\m (\l') \; \mathds{1}_{[0, L^\m]} (\l') \dl' \right] \\
 =  & \ \frac{1}{M} \int_0^{\Lm} \int_0^{\Lm} \frac{1}{M} \sum_{k=1}^M \mathds{1}_{[0, L_r^\k]} (\l) \mathds{1}_{[0, L_r^\k]} (\l') e^{r \l} e^{r \l'} \Cov\left[\left(\dQll\right),  \left(\frac{\text{d} Q(\l')}{\text{d}\l'} \right)\right]  \dl \dl' \\
  \leq  & \ \frac{1}{M} \int_0^{\Lm} \int_0^{\Lm} \frac{1}{M} \sum_{k=1}^M \mathds{1}_{[0, L_r^\k]} (\l) \mathds{1}_{[0, L_r^\k]} (\l') e^{r \l} e^{r \l'} \V\left(\dQll\right)^\half\V\left(\frac{\text{d}Q(\l')}{\text{d}\l'}\right)^\half  \dl \dl' \\
  \leq & \ \frac{c_2}{M} \int_0^{\Lm} \int_0^{\Lm} \frac{1}{M} \sum_{k=1}^M \mathds{1}_{[0, L_r^\k]} (\l) \mathds{1}_{[0, L_r^\k]} (\l') e^{(r -\frac{\be}{2}) \l} e^{(r -\frac{\be}{2}) \l'} \dl \dl',
 \end{aligned}
 \label{eq:qclmc_var_first}
\end{equation}
where we used the Cauchy--Schwarz inequality on the covariance and the convergence assumption on the variance of $\dQll$ in \eqref{eq:QCLMC_variancedecay} from the complexity theorem.
For fixed $L_r^\k$ it holds $\mathds{1}_{[0, L_r^\k]} (\l) \mathds{1}_{[0, L_r^\k]} (\l') = \mathds{1}_{[0, L_r^\k]} (\max(\l,\l'))$ for $\l,\l' \in \R_{\geq 0}$ and we insert a zero by subtracting and adding $e^{-r \max\{\l, \l'\}}$ to obtain
\begin{equation*}
\begin{aligned}
 \frac{1}{M} \sum_{k=1}^M &  \mathds{1}_{[0, L_r^\k]} (\l) \mathds{1}_{[0, L_r^\k]} (\l') e^{(r -\frac{\be}{2}) \l} e^{(r -\frac{\be}{2}) \l'} \\
  & =  \left( \frac{1}{M} \sum_{k=1}^M \mathds{1}_{[0, L_r^\k]} (\max(\l,\l'))  -  e^{-r \max\{\l, \l'\}} \right) e^{(r -\frac{\be}{2}) \l} e^{(r -\frac{\be}{2}) \l'} + e^{-r \max\{\l, \l'\}} e^{(r -\frac{\be}{2}) \l} e^{(r -\frac{\be}{2}) \l'}. 
\end{aligned}
\end{equation*}
Inserting this back into Equation~\eqref{eq:qclmc_var_first}, we obtain the following two integral terms, 
\begin{equation}
  I := \frac{c_2}{M} \int_0^{\Lm} \int_0^{\Lm}\left( \frac{1}{M} \sum_{k=1}^M \mathds{1}_{[0, L_r^\k]} (\max(\l,\l'))  -  e^{-r \max\{\l, \l'\}} \right)  e^{(r -\frac{\be}{2}) \l} e^{(r -\frac{\be}{2}) \l'}  \dl \dl' 
  \label{eq:integral_I}
\end{equation}
and
\begin{equation}
  II :=  \frac{c_2}{M} \int_0^{\Lm} \int_0^{\Lm} e^{-r \max\{\l, \l'\}} e^{(r -\frac{\be}{2}) \l} e^{(r -\frac{\be}{2}) \l'}  \dl \dl', 
  \label{eq:integral_II}
\end{equation}
that are estimated in Appendix \ref{app:appendix} and lead to the following upper bound of the variance of the QCLMC estimator
\begin{equation*}
 \begin{aligned}
 \V\left[\widehat{Q}_{0, \Lm}^\QCLMC\right] \leq & \  \frac{c_{disc}\, c_2}{M^2} \begin{cases}
      \frac{4}{(2r - \be)^2} \left(e^{(2r-\be)\Lm} - e^{(r-\frac{\be}{2})\Lm} + 1\right) &\quadfor r \neq \frac{\be}{2}, \\
      (\Lm)^2 &\quadfor r = \frac{\be}{2},
  \end{cases}
 \\
 & \ + \frac{c_2}{M} \begin{cases}
                      \frac{2}{(r-\be) (r - \frac{\be}{2})} e^{(r-\be) \Lm} +  \frac{4}{\be(r - \frac{\be}{2})} e^{- \frac{\be}{2} \Lm} -  \frac{4}{(r-\be) \be}  & \quadfor r \neq \frac{\be}{2}, \be, \\
                        -\frac{8}{\be^2} e^{-\frac{\be}{2} \Lm} - \frac{4}{\be} \Lm e^{- \frac{\be}{2} \Lm} + \frac{8}{\be^2} & \quadfor r = \frac{\be}{2}, \\
                       \frac{4}{\be} \Lm + \frac{8}{\be^2} e^{-\frac{\be}{2}\Lm} -  \frac{8}{\be^2} & \quadfor r = \be,
                     \end{cases}
 \end{aligned}
\end{equation*}
with explicit dependence on $c_{disc}, c_2, \be, r, M, \bar{L}$ and $\Lm$. In order to bound the variance in terms of $\eps$ and $M$, we again use the choice of $\Lm$ from Equation \eqref{eq:qclmc_L_max} and the assumption $\eps < e\inv$, that yields $|\ln(\eps)| > 1$. 
First, we trivially bound the variance further by removing all negative terms and by bounding the exponential terms with negative exponent by one to arrive at
\begin{equation*}
 \begin{aligned}
 \V\left[\widehat{Q}_{0, \Lm}^\QCLMC\right] \leq & \  \frac{c_{disc}\, c_2}{M^2} \begin{cases}
      \frac{4}{(2r - \be)^2} \left(e^{(2r - \be)\Lm} + 1\right) &\quadfor r \neq \frac{\be}{2}, \\
      (\Lm)^2 &\quadfor r = \frac{\be}{2},
  \end{cases}
 \\
 & \ + \frac{c_2}{M} \begin{cases}
                      \frac{2}{(r-\be) (r - \frac{\be}{2})} e^{(r-\be) \Lm} +  \frac{4}{\be(r - \frac{\be}{2})} e^{- \frac{\be}{2} \Lm} -  \frac{4}{(r-\be) \be}   & \quadfor r \neq \frac{\be}{2}, \be, \\
                        \frac{8}{\be^2} & \quadfor r = \frac{\be}{2}, \\
                       \frac{4}{\be} \Lm + \frac{8}{\be^2} & \quadfor r = \be,
                     \end{cases}
 \end{aligned}
\end{equation*}
We bound the variance it in each distinct case by using the bounds from Equations $\eqref{eq:qclmc_helping_bounds_1}$ -- $\eqref{eq:qclmc_helping_bounds_3}$. 
\\
\textbf{Case 1: $(r < \frac{\be}{2})$}
\begin{equation*}
\begin{aligned}
 \V\left[\widehat{Q}_{0, \Lm}^\QCLMC\right] & \ \leq \frac{c_{disc}\, c_2}{M^2} \frac{4}{(2r - \be)^2} \left(e^{(2r-\be)\Lm} + 1\right) \\
 & \ \qquad + \frac{c_2}{M} \left(\frac{2}{(r-\be) (r - \frac{\be}{2})} e^{(r-\be) \Lm} +  \frac{4}{\be(r - \frac{\be}{2})} e^{- \frac{\be}{2} \Lm} -  \frac{4}{(r-\be) \be}\right) \leq C_8 \frac{1}{M}.
 \end{aligned}
\end{equation*}
\\
\textbf{Case 2: $(r = \frac{\be}{2})$}
\begin{equation*}
 \V\left[\widehat{Q}_{0, \Lm}^\QCLMC\right] \leq \frac{c_{disc}\, c_2}{M^2} (\Lm)^2 + \frac{c_2}{M} \frac{8}{\be^2} \leq C_9 \left(\frac{\ln(\eps)^2}{M} + 1\right)\frac{1}{M}.
\end{equation*}
\\
\textbf{Case 3: $(\frac{\be}{2} < r < \be)$}
\begin{equation*}
\begin{aligned}
 \V\left[\widehat{Q}_{0, \Lm}^\QCLMC\right] & \ \leq \frac{c_{disc}\, c_2}{M^2} \frac{4}{(2r - \be)^2} \left(e^{(2r-\be)\Lm} + 1\right) \\
 & \ \qquad + \frac{c_2}{M} \left(\frac{2}{(r-\be) (r - \frac{\be}{2})} e^{(r-\be) \Lm} +  \frac{4}{\be(r - \frac{\be}{2})} e^{- \frac{\be}{2} \Lm} -  \frac{4}{(r-\be) \be}\right) \\
 & \ \leq C_{10} \left(\frac{\eps^{-\frac{\be}{\al}}}{M} + 1\right) \frac{1}{M}.
 \end{aligned}
\end{equation*}
\\
\textbf{Case 4: $(r = \be)$}
We compute with $|\ln(\eps)| > 1$
\begin{equation*}
\begin{aligned}
 \V\left[\widehat{Q}_{0, \Lm}^\QCLMC\right] \leq & \ \frac{c_{disc}\, c_2}{M^2} \frac{4}{\be^2} \left(e^{\be \Lm} + 1\right) + \frac{c_2}{M} \left( \frac{4}{\be} \Lm + \frac{8}{\be^2} \right) \leq C_{11} \left(\frac{\eps^{-\frac{\be}{\al}}}{M} + 1 \right) \frac{|\ln(\eps)|}{M}.
 \end{aligned}
\end{equation*}
\\
\textbf{Case 5: $(r > \be)$}
\begin{equation*}
\begin{aligned}
 \V\left[\widehat{Q}_{0, \Lm}^\QCLMC\right] & \ \leq \frac{c_{disc}\, c_2}{M^2} \frac{4}{(2r - \be)^2} \left(e^{(2r-\be)\Lm} + 1\right) \\
 & \ \qquad  + \frac{c_2}{M} \left(\frac{2}{(r-\be) (r - \frac{\be}{2})} e^{(r-\be) \Lm} +  \frac{4}{\be(r - \frac{\be}{2})} e^{- \frac{\be}{2} \Lm} -  \frac{4}{(r-\be) \be}\right) \\
 & \ \leq  C_{12} \left(\frac{\eps^{-\frac{r}{\al}}}{M} + 1 \right) \frac{\eps^{-\frac{r - \be}{\al}}}{M}.
 \end{aligned}
\end{equation*}
All constants $C_8, C_9, C_{10}, C_{11}, C_{12} > 0$ are independent of $M$, $\Lm$ and $\eps$.
Overall, we obtain 
\begin{equation}
 \begin{aligned}
  \V\left[\widehat{Q}_{0, \Lm}^\QCLMC\right] \leq C_{V} \frac{1}{M} 
                  \begin{cases}
                      1 & \quadfor r < \frac{\be}{2}, \\
                      \frac{\ln(\eps)^2}{M} + 1 & \quadfor r = \frac{\be}{2}, \\
                      \frac{\eps^{-\frac{2r - \be}{\al}}}{M} + 1  &\quadfor \frac{\be}{2} < r < \be, \\
                        \left(\frac{\eps^{-\frac{\be}{\al}}}{M} + 1\right)|\ln(\eps)| & \quadfor r = \be, \\
                        \left(\frac{\eps^{-\frac{r}{\al}}}{M} + 1\right) \eps^{-\frac{r - \be}{\al}} & \quadfor r > \be,
                   \end{cases}
  \end{aligned}
  \label{eq:upper_bound_variance}
\end{equation}
with a constant $C_V := \max\{C_8,\,C_9,\,C_{10},\,C_{11},\,C_{12}\} > 0$, independent of $M$, $\Lm$ and $\eps$. Further, to bound the variance by $\frac{\eps^2}{4}$ by an appropriate choice of $M$, in all cases $r \geq \frac{\be}{2}$ the additional contributions from $\eps$ need to be compensated for by the additional factor of $\frac{1}{M}$. This is realized, as shown in the bias part right before Equation \eqref{eq:lower_bound_M_bias}, by choosing 
\begin{equation*}
 M \geq \widetilde{C}_{V} \eps^{- 2 - \max\{0, \frac{r - \min\{\be, 2\al\}}{\al} \}} |\ln(\eps)|^{\delta_{r, \be}},
\end{equation*}
with an appropriate constant $\widetilde{C}_{V} > 0$ independent of $M$, $\Lm$ and $\eps$. The lower bound on $M$ determined by the term $B_2$ (see~\eqref{eq:lower_bound_M_bias}) is no more constraining than this newly established lower bound for $M$, with respect to its dependence on $\eps$, since
\begin{equation*}
\begin{aligned}
 \eps^{-2 - \max\{0, \frac{r - \min\{\be, 2\al\}}{\al}\}} \geq \eps^{-2 - \max\{0, \frac{r - 2\al}{\al}\}}.
 \end{aligned}
\end{equation*}
For the upper bound to the MSE we therefore obtain
\begin{equation*}
 \MSE_{0,\Lm}^\QCLMC \leq \V[\widehat{Q}_{0, \Lm}^\QCLMC] + B_2(\Lm \land \bar{L}, M) + B_1(\Lm \land \bar{L}) \leq \frac{\eps^2}{4} + \frac{\eps^2}{4} + \frac{\eps^2}{2} = \eps^2,
\end{equation*}
by choosing
\begin{equation}
 M \geq \left\lceil C_{MSE} \eps^{- 2 - \max\{0, \frac{r - \min\{\be, 2\al\}}{\al} \}} |\ln(\eps)|^{\delta_{r, \be}} \right\rceil,
 \label{eq:lower_bound_M_mse}
\end{equation}
where $\lceil \cdot \rceil$ denotes the Gauss bracket, for a constant $C_{MSE}:= \max\{\widetilde{C}_{B_2}, \widetilde{C}_{V}\}> 0$ independent of $M$, $\Lm$ and $\eps$.
Finally, we compute an upper bound for the cost of the estimator by
\begin{equation*}
\begin{aligned}
 \calC(\widehat{Q}_{0, \Lm}^\QCLMC) = & \ \sum_{k=1}^M \,\int_0^{\Lm} \mathds{1}_{[0, L^{k}]}(\l) \, \frac{\text{d} \calC[\l]}{\dl} \dl \leq c_3 \sum_{k=1}^M \,\int_0^{\Lm} \mathds{1}_{[0, L^{k}]}(\l) \, e^{\ga \l} \dl \\
 = & \  c_3 \, M\,\int_0^{\Lm} \frac{1}{M} \sum_{k=1}^M\mathds{1}_{[0, L^{k}]}(\l) \, e^{\ga \l} \dl = c_3 \,M\,\int_0^{\Lm} \left( \frac{1}{M} \sum_{k=1}^M\mathds{1}_{[0, L^{k}]}(\l) - e^{-r \l} + e^{-r \l} \right) e^{\ga \l} \dl \\
 \leq  & \ c_3 \,M\,\sup_{\l > 0} \left|\frac{1}{M} \sum_{k=1}^M\mathds{1}_{[0, L^{k}]}(\l) - e^{-r \l}\right|\int_0^{\Lm} e^{\ga \l} \dl +  c_3 \,M\,\int_0^{\Lm} e^{-r \l}  e^{\ga \l} \dl \\
 \leq & \ c_3 c_{disc} \frac{1}{\ga}\left(e^{\ga \Lm} - 1\right) + c_3 M    \begin{cases}
                    \frac{1}{\ga - r} \left(e^{(\ga - r) \Lm} - 1 \right) & \quadfor r \neq \ga, \\
                    \Lm & \quadfor r = \ga,
                    \end{cases}  \\
 \leq & \  C_2 \eps^{-\frac{\ga}{\al}} + C_{13} M \eps^{-\max\{0, \frac{\ga - r}{\al}\}}  |\ln(\eps)|^{\delta_{r , \ga}}
\end{aligned}                                               
\end{equation*}
for a constant $C_{13} > 0$ independent of $M$, $\Lm$ and $\eps$. Bounding the Gauss bracket in Equation~\eqref{eq:lower_bound_M_mse} by adding one to it and inserting this for $M$ we obtain for the cost
\begin{equation*}
 \begin{aligned}
    \calC(\widehat{Q}_{0, \Lm}^\QCLMC) \leq & \ C_2 \eps^{-\frac{\ga}{\al}}  + C_{13} \eps^{-\max\{0, \frac{\ga - r}{\al}\}}  |\ln(\eps)|^{\delta_{r , \ga}} \\
    & \ + C_{13}\eps^{-\max\{0, \frac{\ga - r}{\al}\}}  |\ln(\eps)|^{\delta_{r , \ga}} C_{MSE} \eps^{- 2 - \max\{0, \frac{r - \min\{\be, 2\al\}}{\al} \}} |\ln(\eps)|^{\delta_{r, \be}} \\
    & \ \leq C \eps^{-2 - \max\{0, \frac{\ga - \min\{\be, 2\al\}}{\al} \}} |\ln(\eps)|^{\delta_{r, \be} + \delta_{r, \ga}}.
 \end{aligned}
\end{equation*}
for some constant $C > 0$, independent of $M$, $\Lm$ and $\eps$, finishing the proof. In the last inequality we used the assumption $r \in [\min\{\be, 2\al, \ga\}, \max\{\min\{\be, 2\al\}, \ga\}]$ that yields $\eps^{-\max\{0, \frac{\ga - r}{\al}\}} \eps^{- 2 - \max\{0, \frac{r - \min\{\be, 2\al\}}{\al} \}} = \eps^{-2 - \max\{0, \frac{\ga - \min\{\be, 2\al\}}{\al} \}}$ and furthermore the trivial bounds for $0 < \eps < e\inv$
\begin{equation*}
 \eps^{-\frac{\ga}{\al}} \leq \eps^{-2 - \max\{0, \frac{\ga - \min\{\be, 2\al\}}{\al} \}},  \quad \eps^{-\max\{0, \frac{\ga - r}{\al}\}}  \leq \eps^{- \max\{0, \frac{\ga - \min\{\be, 2\al\}}{\al} \}} \quadand |\ln(\eps)|^{\delta_{r , \ga}}  \leq \eps^{-2}.
\end{equation*}
\end{proof}

\begin{remark}
 \label{rem:complexity_L_infty}
 For simplicity of notation the proof is given for the choice $\Lm < \infty$ such that the bias term is bounded by $\frac{\eps}{\sqrt{2}}$ and an appropriate value of $M$ (see~\eqref{eq:lower_bound_M_discr}) to ensure $\Lm \land \bar{L} = \Lm$ is considered. Note, however, that the complexity theorem still holds in the case $\Lm = \infty$, where $\Lm \land \bar{L} = \bar{L}$, with the same choice of $M$. Most importantly this means, that for $\Lm = \infty$ all error contributions in the MSE, namely the squared bias and the variance, decrease automatically with growing $M$ while the cost does not blow up, because the maximal level $\bar{L} = \ln(\tilde{c}^{-\frac{1}{r}} M^\frac{1}{r})$ grows with $M$ at just the right speed. Further note, that taking the limit $M \to \infty$ in the complexity theorem as done to proof the unbiasedness properties of the QCLMC estimator in Proposition \ref{prop:QCLMC_unbiased} is not sensible, since $\lim_{M\to \infty} \MSE = 0$ with $\lim_{M\to\infty} \calC(\widehat{Q}_{0, \Lm}^\QCLMC) = \infty$. 
\end{remark}

\begin{remark}
 \label{rem:iid_sequence}
 It is important to note, that in order to prove the Complexity Theorem \ref{thm:QCLMC-complexity} using the $F$-discrepancy property as done in this work, the sequence $(L_r^\k;\;k=1,..,M)$ must satisfy $\ka = 0$. This is not the case for an i.i.d. sequence with $\ka = \frac{1}{2}$. The additional dependence on $\eps$ stemming from the use of the $F$-discrepancy property in the upper bounds to $B_2$ (see~\eqref{eq:upper_bound_bias}) and upper bound to the variance (see~\eqref{eq:upper_bound_variance}), cannot be compensated by $M^{\ka - 1}$ with $\ka = \frac{1}{2}$ instead of $\kappa = 0$.
\end{remark}

\section{Numerical experiments}
\label{sec:numerical_experiments}
The derived upper bounds to the MSE from the proofs of the complexity theorems for QCLMC and CLMC allow us to compare the methods to one another in terms of their computational time to error performance on the basis of the underlying stochastic model parameters. In order to do so we introduce a random PDE model as our stochastic model problem and approximate it by a spatial discretization via $h$-adaptive finite elements. Further, in Algorithm \ref{alg:qclmc} we formulate a practical (Q)CLMC algorithm, state how to obtain the sample adaptive meshes via a-posteriori error estimation and we demonstrate how to numerically estimate the model parameters, that are the basis of the performance comparison. %

As a followup experiment we compare the performance of CLMC and QCLMC by estimating the real achieved MSE over a series of runs of the (Q)CLMC algorithm for a growing sequence of sample sizes.

The numerical experiments are implemented in Python, where all finite element computations are implemented in FEniCS \cite{Fenics2015}. The linear systems are solved with its integrated optimized direct $LU$-decomposition. The computations are done on an Intel(R) Core(TM) i$7$-$4770$ CPU running at $3.4$ GHz with $4$ cores and $2$ threads per core.
\subsection{Random PDE model and its discretization}
\label{subsec:random_pde_model_and_its_discretization}
For the comparison of performances between the CLMC and QCLMC method we consider the quantity of interest $\calQ = \|\cdot\|_{H^1(\calD)}$ to be the $H^1$-norm of the solution of a random PDE, that we introduce next. Let $(\Omega,\calA,\P)$ be a complete probability space and $\calD \subset \R^d$, $d=1,2,3$ be a bounded and connected Lipschitz domain. The linear, random elliptic PDE with solution $u: \Omega \times \calD \to \R$ is given by
\begin{equation}
    -\nabla \cdot (a(\omega,x) \nabla u(\omega,x)) = f(x) \quad \text{in} \; \Omega \times \calD,
\label{eq:PDE}
\end{equation}
where $f: \calD \to \R$ is the source term and $a : \Omega \times \calD \to \R$ is the random coefficient.
The boundary $\partial \calD$  is assumed to be Lipschitz continuous and equipped with homogeneous Dirichlet boundary conditions
\begin{equation*}
    \begin{aligned}
    u(\omega,x) &= 0 \quad \text{on} \; \Omega \times \partial \calD.
    \end{aligned}
\end{equation*}
This is a simple mathematical model for subsurface flow through porous media and has been a common model problem in various works on uncertainty quantification (see, e.g.,\cite{Cliffe2011_MLMCPDE, Teckentrup2013_MLMCPDE, Barth2018_Elliptic}).
For simplicity we choose $\calD:=[0,1]^2$ and set $f\equiv 1$ in Equation~\eqref{eq:PDE}. 
We consider a log-Gauss random field as the random coefficient $a$ with a covariance function of the Mat\'ern class, i.e.,
\begin{equation}
  \Cov(x,y) : = v\,\frac{2^{1-\nu}}{\Gamma(\nu)} \left(\frac{\sqrt{2\nu} \, \|x - y\|_2}{\lambda}\right)^{\nu} K_\nu \left( \frac{\sqrt{2\nu} \, \|x - y\|_2}{\lambda} \right) \quadfor x,y \in \R^d,
 \label{eq:matern_kernel}
\end{equation}
where $\| \cdot \|_2$ denotes the Euclidean norm on $\calD$ and where $v > 0$ is the variance, $\lambda > 0$ the correlation length and $\nu > 0$ a parameter steering the roughness of the field. The functions $\Gamma$ and $K_\nu$ are the Gamma function and modified Bessel function, cf.~\cite{Abramowitz1964_Handbook}, respectively.
As an approximation to $a$ we consider a truncated Karhunen--Lo\`eve expansion (cf.~\cite{Alexanderian2015_KL}), for $R \in \N$
\begin{equation}
    a_R(x, \omega) = \exp\left(\sum_{m=1}^R \sqrt{\mu_m} \phi_m(x) \xi_m \right),
    \label{eq:log_gauss_coefficient}
\end{equation}
where $\xi_m \overset{d}{=} \calN(0,1)$ are standard normal-distributed random variables and the eigenvalues $\mu_m$ and eigenfunctions $\phi_m$ of the covariance kernel \eqref{eq:matern_kernel} are approximated via the Nystr\"om method (cf.~\cite{Rasmussen2007_Gaussian}), for $1 \leq m \leq R$. For the numerical discretization of the PDE \eqref{eq:PDE} we consider the Finite Element method (FE), see, e.g., \cite{Knabner2003_NumericalMethods, Hackbusch2017_EllipticEquations, Brenner2008_FEM}, with sample-dependent adaptive meshes. Exemplary visualizations of the log-Gauss random coefficient are given in Figure \ref{fig:log_gauss_coefficient} and a numerical approximation to a pathwise solution of the PDE \eqref{eq:PDE} on an adaptive mesh is given in Figure \ref{fig:log_gauss_solution}.

\begin{figure}[tbhp]
\centering
 \includegraphics[width=0.49\textwidth]{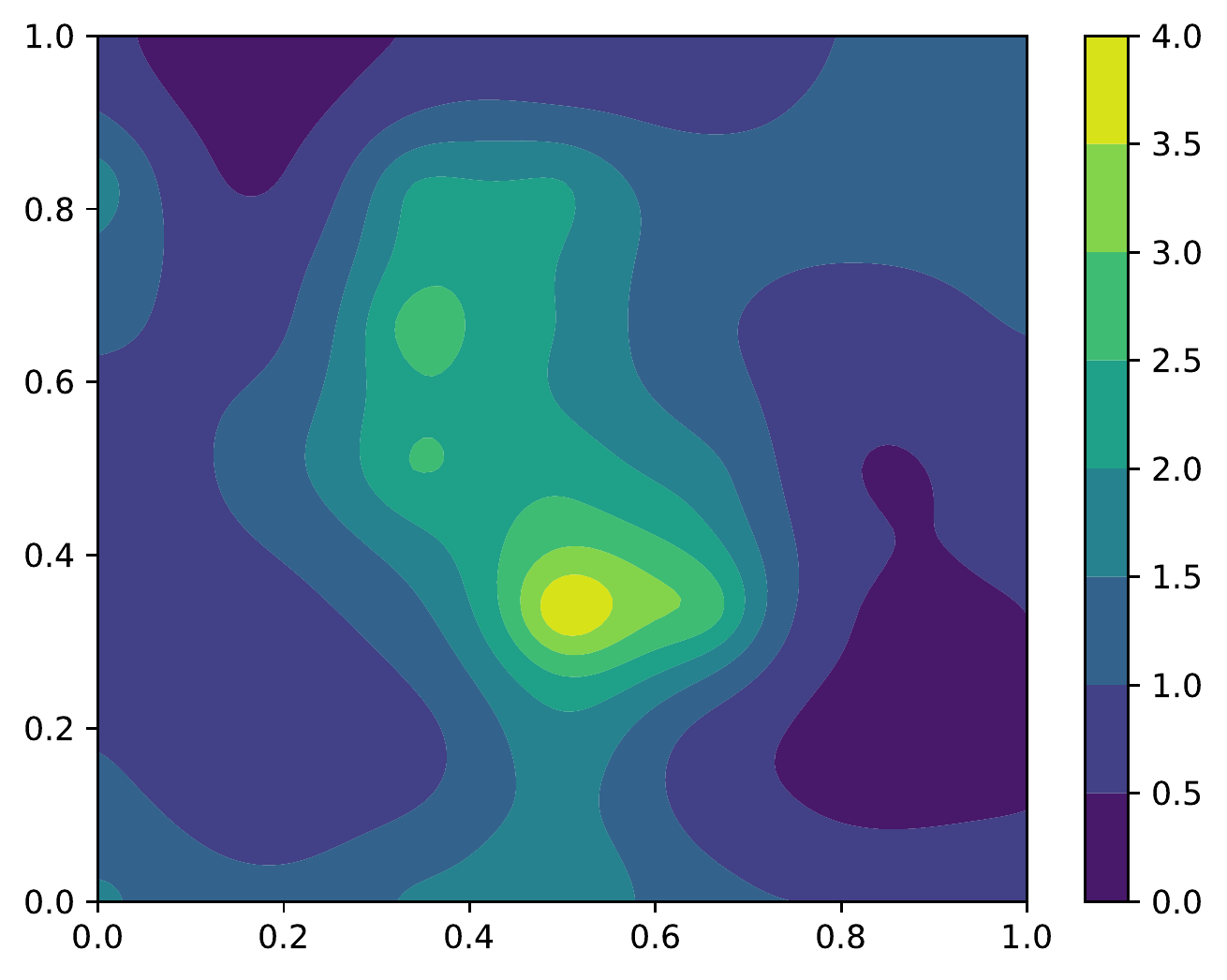}
 \includegraphics[width=0.49\textwidth]{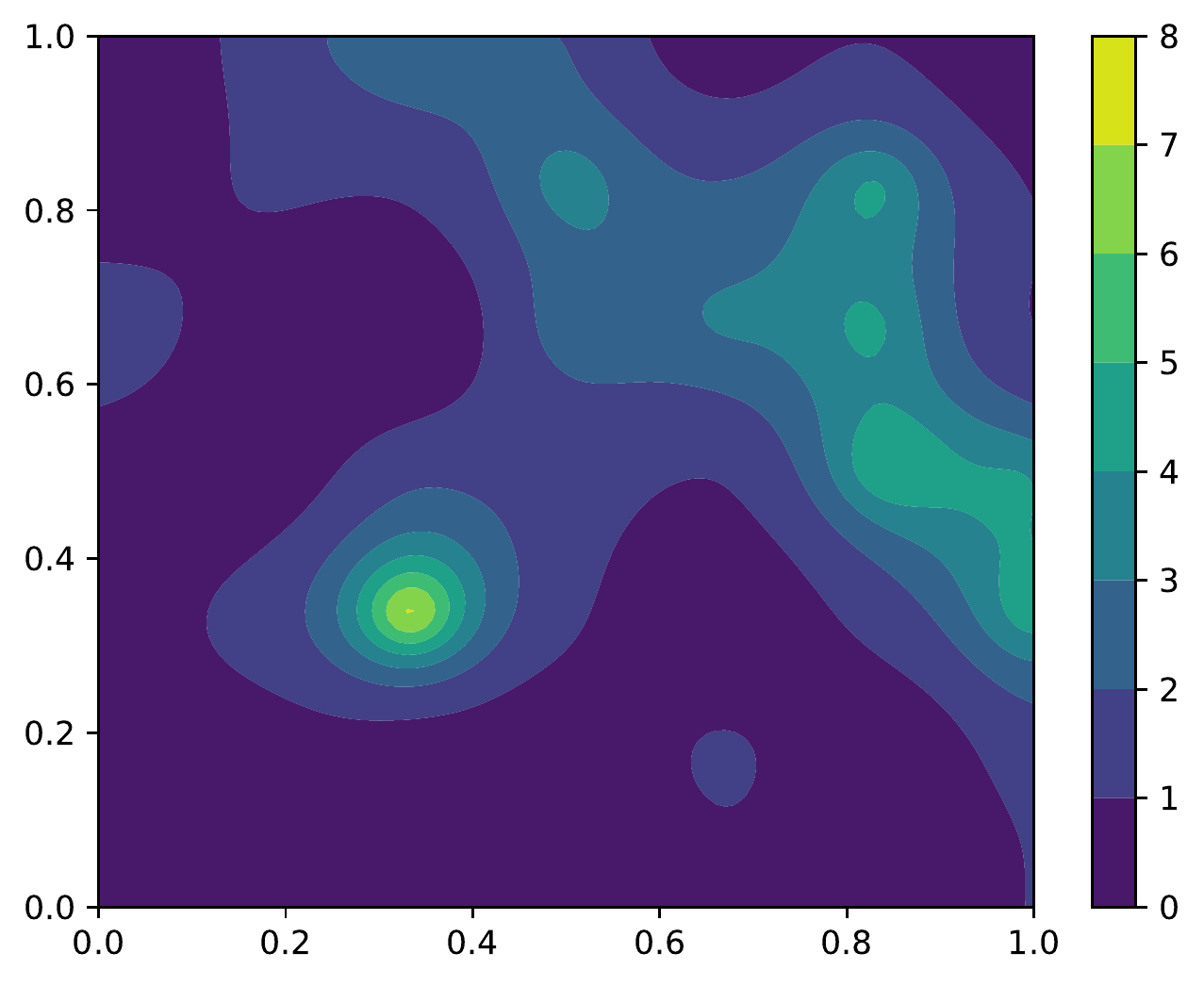}
 \includegraphics[width=0.49\textwidth]{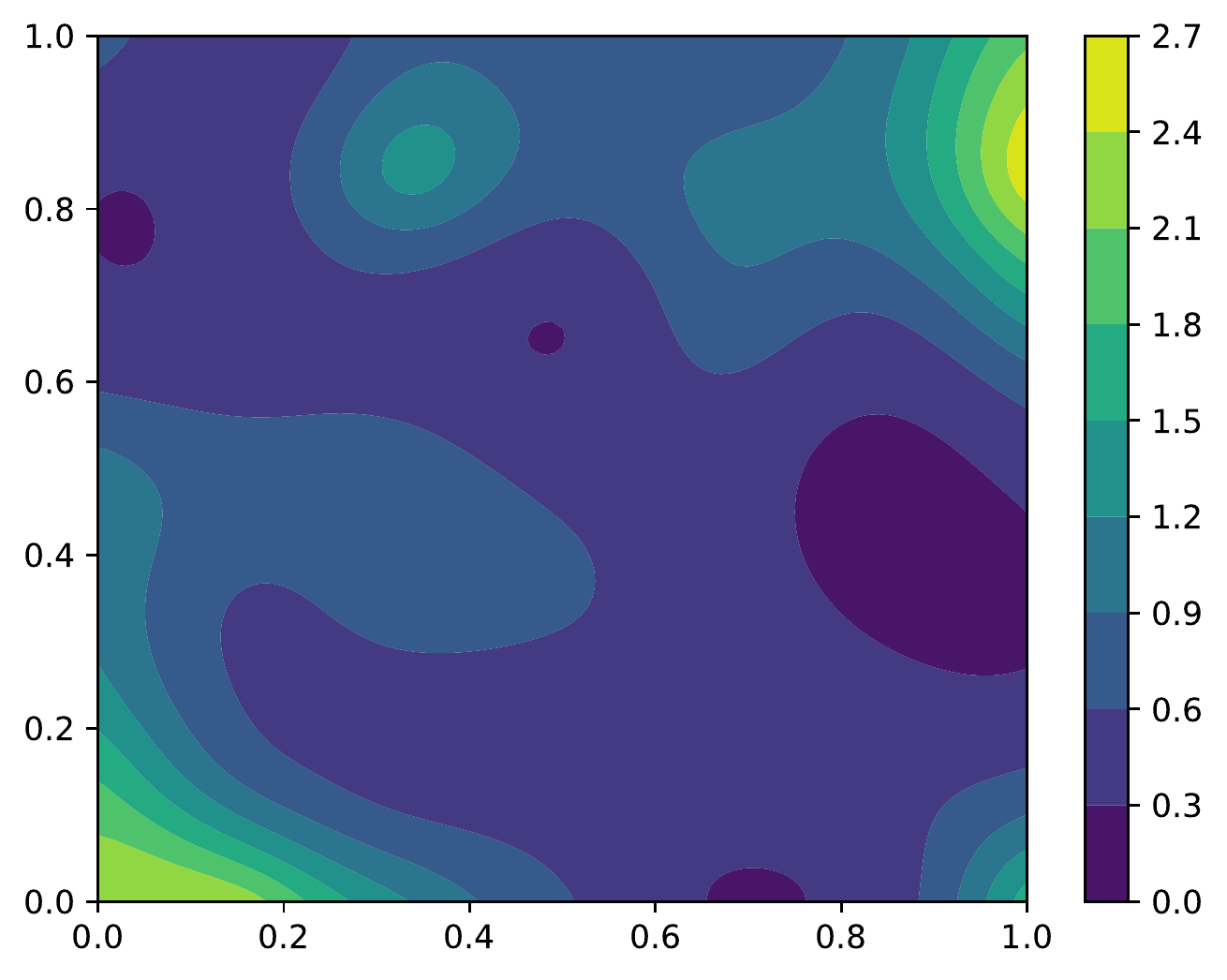}
 \includegraphics[width=0.49\textwidth]{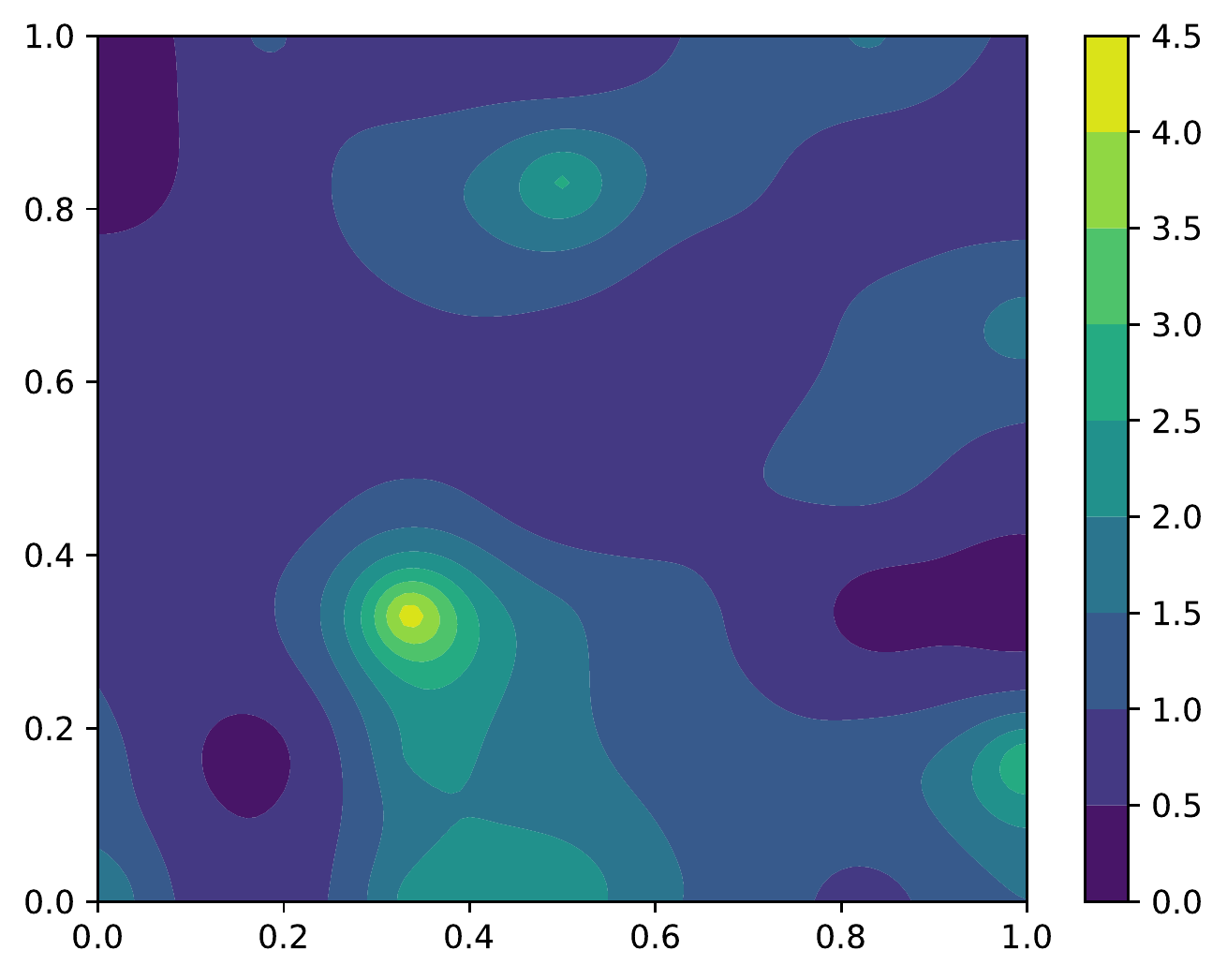}
 \caption{Visualization of samples of the log-Gauss random coefficient for $\nu=1.5$, $\lambda = 0.1$, $v=0.5$ (upper left), $\nu=1.5$, $\lambda = 0.1$, $v=1$ (upper right), $\nu=1.5$, $\lambda = 0.2$, $v=0.5$ (lower left), $\nu=1$, $\lambda = 0.1$, $v=0.5$ (lower right). The KL-expansion \eqref{eq:log_gauss_coefficient} was truncated after $R=36$ terms in each case.}
 \label{fig:log_gauss_coefficient}
\end{figure}
\begin{figure}[tbhp]
\centering
 \includegraphics[width=0.34\textwidth]{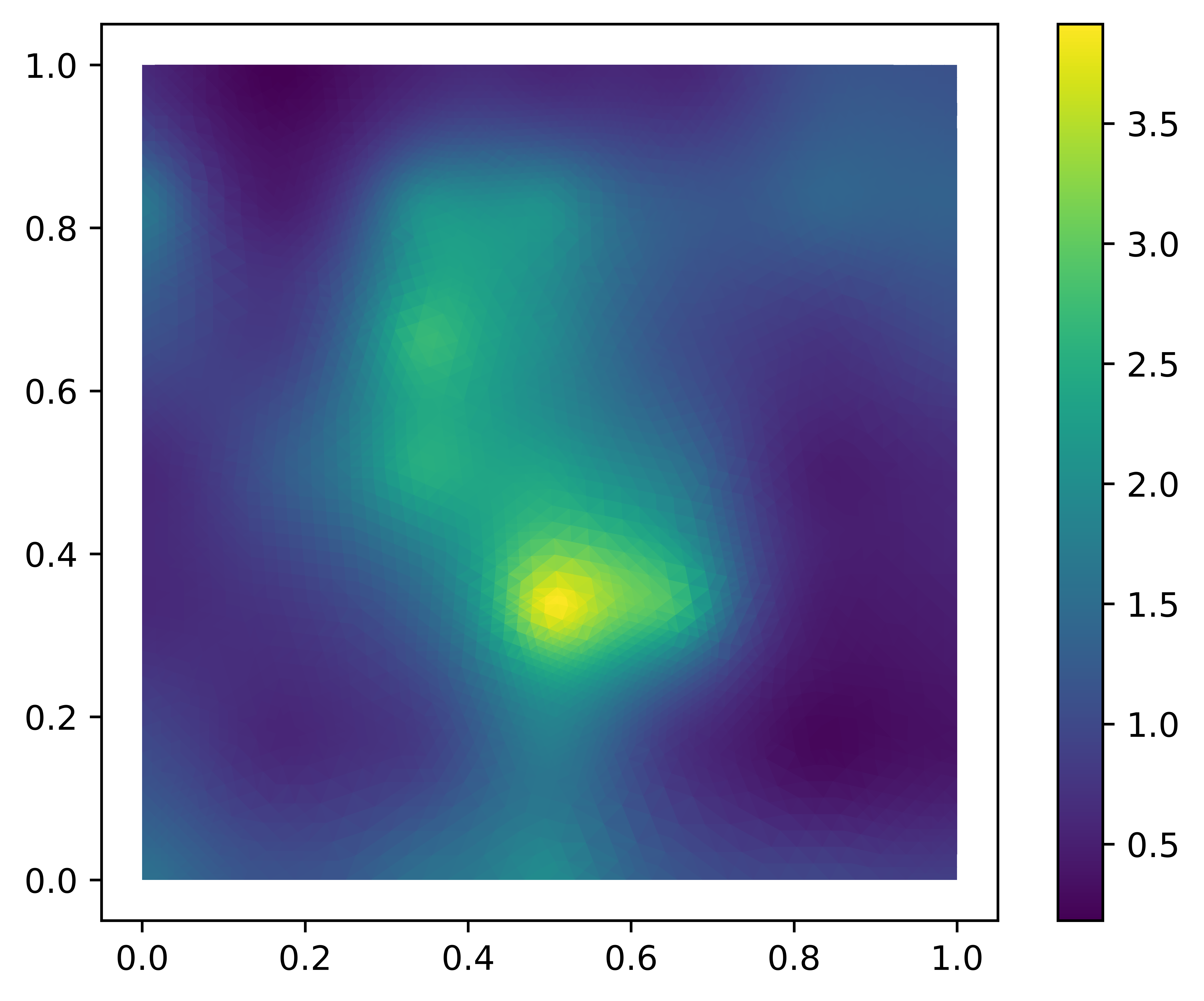}
 \includegraphics[width=0.29\textwidth]{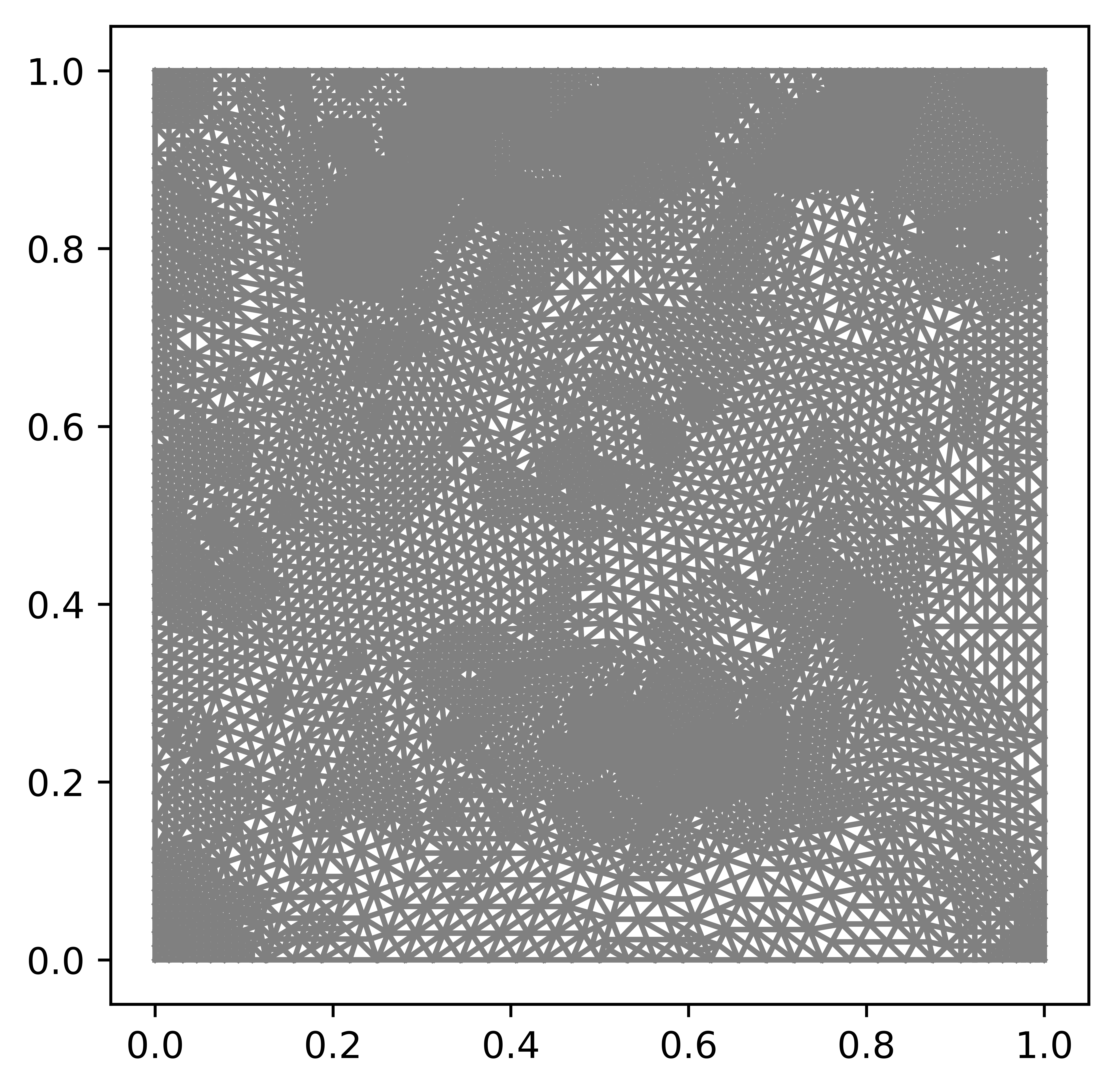}
 \includegraphics[width=0.35\textwidth]{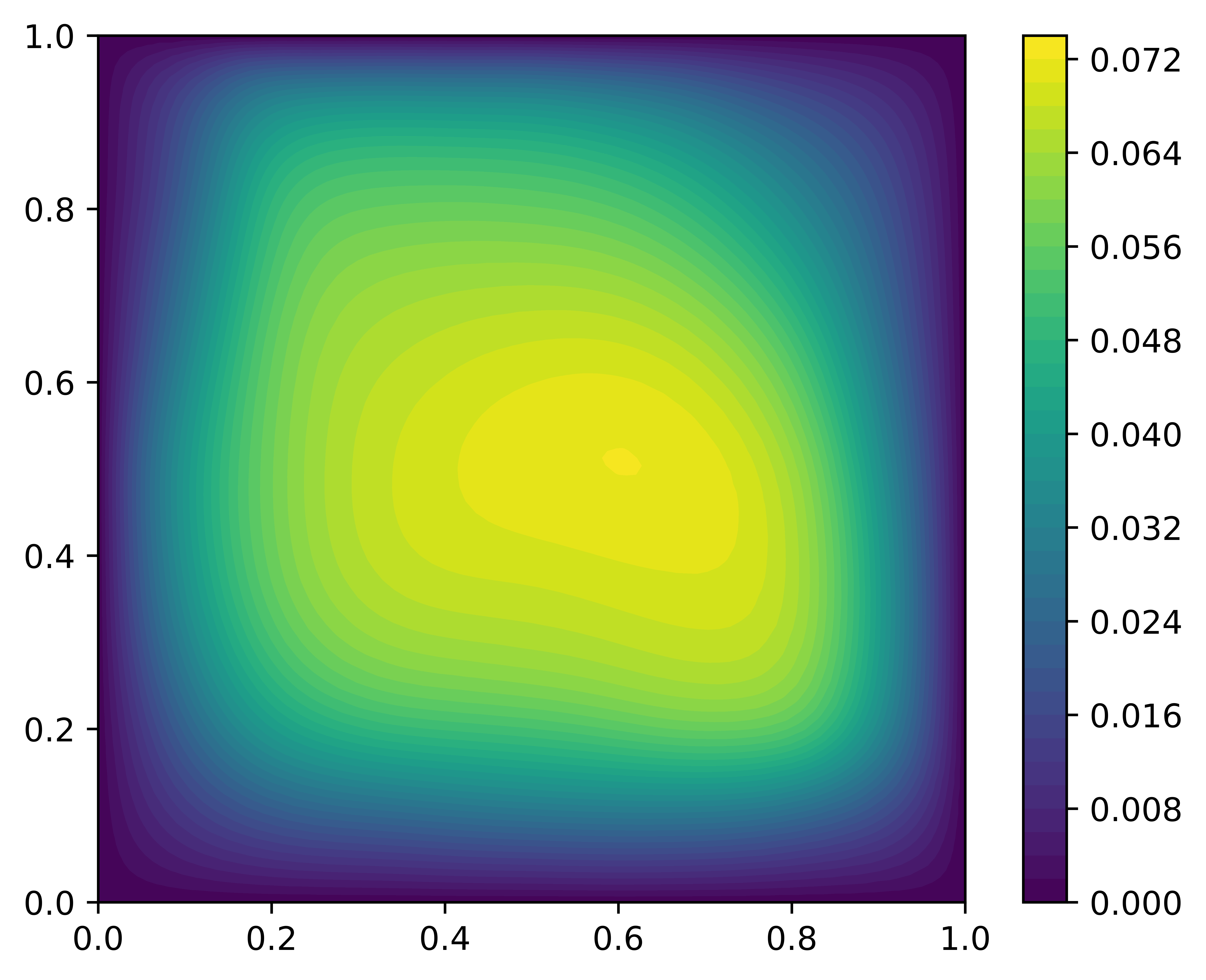}
 \caption{Single sample of the log-Gauss random coefficient for $\nu=1.5$, $\lambda = 0.1$, $v=0.5$ (left) on adaptive mesh (middle) generated by $5$ iterative refinement steps, see Remark \ref{rem:refinement}, and corresponding PDE solution (right). The KL-expansion was truncated after $R=36$ terms.}
 \label{fig:log_gauss_solution}
\end{figure}
%
%
\subsection{Practical estimator, a-posteriori error and parameter estimates}
\label{subsec:pratical_estimator_a_posteriori_error_and_parameter_estimates}

The continuous stochastic process $\dQl^\k$ for each sample $k \in \N$ has to be approximated in order to be computable numerically. As described in \cite{Detommaso2019_CLMC} and as done in the beginning when deriving the MLMC estimator from the CLMC estimator, a straightforward approximation is given via linear interpolation
\begin{equation*}
 \dQl^\k (\l) := \frac{Q_j^\k - Q_{j-1}^\k}{\l_{j}^\k - \l_{j-1}^\k} \quadfor \in (\l_{j-1}^\k ,\l_{j}^\k],
\end{equation*}
with samples $Q_j^\k$ as approximations to the quantity of interest at levels $\l_j^\k$ for $j \geq 1$, $k \in \N$. This is only one possible choice of many, e.g., a particular regression function or polynomial interpolant may be used to match the global trend of the process $Q(\l)$, cf.~\cite[Section~3.3]{Detommaso2019_CLMC} for details.
Inserting the linear interpolation into the QCLMC (respectively CLMC) estimator we obtain
\begin{equation}
    \widehat{Q}_{0,\Lm}^\QCLMC = \frac{1}{M} \sum_{k=1}^M \sum_{j=1}^{J^\k} \int_{\l_{j-1}^\k}^{\tilde{\l_j}^\k} \frac{1}{\P(L_r \geq \l)} \dl \; \frac{Q_j^\k - Q_{j-1}^\k}{\l_{j}^\k - \l_{j-1}^\k},
 \label{eq:qclmc_discretized}
\end{equation}
with $M \in \N$ and
\begin{equation*}
    J^\k := \min\{j \geq 1: \l_j^\k \geq L_r^\k \land \Lm\}, \quad \tilde{\l}_j^\k := \l_j^\k \land L_r^\k \land \Lm.
    \label{eq:qclmc_max_level_index}
\end{equation*}
Different to the derivation of the MLMC estimator in Equation \eqref{eq:CLMC_discretized}, the level random variable $L_r$ is exponentially-distributed to some parameter $r > 0$ and the integral in the estimator \eqref{eq:qclmc_discretized} computes to
\begin{equation*}
 \int_{\l_{j-1}^\k}^{\tilde{\l_j}^\k} \frac{1}{\P(L_r \geq \l)} \dl = \int_{\l_{j-1}^\k}^{\tilde{\l_j}^\k} e^{r\l} \dl = \frac{\exp(r \tilde{\l}_j^\k) - \exp(r \l_{j-1}^\k)}r,
 \label{eq:qclmc_weights}
\end{equation*}
and the practical (Q)CLMC estimator is given by
\begin{equation*}
    \widehat{Q}_{0,\Lm}^\QCLMC = \frac{1}{M} \sum_{k=1}^M \sum_{j=1}^{J^\k} \frac{\exp(r \tilde{\l}_j^\k) - \exp(r \l_{j-1}^\k)}{r\left(\l_{j}^\k - \l_{j-1}^\k\right)} \left(Q_j^\k - Q_{j-1}^\k\right).
\end{equation*}
The samplewise continuous level of refinement for each sample $k$ is defined by
    $\l_j^\k := - \log (e_j^\k/e_0^\k )$ for $j=0,\dots,J$, for $J \in \N$ and $k=1,\dots,M$,
naturally providing values $\l_0^\k = 0$ for all $k = 1,\dots,M$. In order to use sample adaptive meshes in the (Q)CLMC method, we use a standard energy norm error estimator for each sample $k \in \N$ (we refer to~\cite{Graetsch2005_APostError} for an overview of different a-posteriori error estimation techniques).
The values $(e_j^\k; j=0,\dots,J)$ are computable a-posteriori error estimators that satisfy
\begin{equation}
    \left| \calQ^\k - Q_j^\k \right| = \left| \| u^\k \|_{H^1(\calD)} - \|u_{j}^\k\|_{H^1(\calD)}\right| \leq C_{est}^\k \left(\sum_{K \in \calK_j^\k} \left( \varphi_K^\k \right)^2 \right)^\frac{1}{2} =: e_j^\k,
    \label{eq:error_estimator}
\end{equation}
for each sample $k \in \N$ with a constant $C_{est}^\k > 0$ independent of $u^\k$ and the FE approximation $u_{j}^\k$. The elementwise error indicator $\varphi_K^\k$ is  given by the formula
\begin{equation*}
 \left(\varphi_K^\k \right)^2 =  h_K^2 \,\|f_j + \div (a_j \nb u_j)\|_{L^2(K)}^2 + \frac{1}{2}\sum_{\ga  \in \calE_K} h_\ga \|[\vv{n}_\ga|_K \cd (a_j \nabla u_{j})]_{\ga}\|_{L^2(\ga)}^2,
 \label{eq:a-posteriori_error}
\end{equation*}
where we omitted the dependence of the right hand side terms on $k$ for a better readability. The quantities $f_j$ and $a_j$ are approximations to $f$ and $a$, $h_K$ and $h_\ga$ are the element diameter and edge length of element $K$ and edge $\ga$ and $\vv{n}_\ga$ is the outward pointing unit normal vector to edge $\ga$. Further, $\calE_K$ is the collection of all edges of element $K \in \calK_j^\k$ and $[\cdot]_{\ga}$ denotes the jump of a quantity over the edge $\ga$. Details on the derivation of the estimator are found in, e.g., \cite{Graetsch2005_APostError,Ainsworth1997_APostError}.

As in~\cite{BeschleBarth2023_QCLMC}, we estimate the underlying model parameters and constants for CLMC and QCLMC, $\al, \be, \ga, c_1,c_2,c_3$ from Theorem \ref{thm:QCLMC-complexity} numerically, since for the considered model problem and various real-world applications they are not available theoretically. We apply the natural logarithm to Equations \eqref{eq:QCLMC_meandecay}, \eqref{eq:QCLMC_variancedecay} and \eqref{eq:QCLMC_costincrease} to obtain the linear relationships
$
    \ln(\E [\scalebox{0.8}{$\dQl$}]) \leq \tilde{c}_1 -\al \l,$ $\ln(\V [\scalebox{0.8}{$\dQl$}]) \leq \tilde{c}_2 - \be \l$ and
    $\ln(\frac{\text{d}\calC}{\dl}) \leq  \tilde{c}_3 + \ga \l,
$
where $\tilde{c}_i = \ln(c_i)$ for $i=1,2,3$.  
Using the definition for $\dQl$ from above and a similar definition for $\frac{\text{d}\calC}{\dl}$, the mean, variance and cost quantities are estimated by sample averages at refinement steps $j = 0,...,J$ with corresponding approximations to the levels $\l_j \approx \frac{1}{M}\sum_{k=1}^M \l_j^\k$. Finally, the parameters and constants from the linear relationships are obtained by linear fitting.
\begin{remark}
 \label{rem:refinement}
 
  The adaptive refinement procedure (throughout this work) of $J \in \N$ refinement steps is the classical Dörfler marking strategy from \cite{Doerfler1996_AdaptiveAlgorithm}. Starting on an initial unstructured uniform mesh, all elements that exceed $50\%$ of the total a-posteriori error bound according to Equation~\eqref{eq:error_estimator} are refined in each step, i.e. for $j=0,\dots,J-1$.
 
\end{remark}
\begin{table}[h]
\begin{tabular}{|l|l|l|l|l|l|l|l|}
\hline
   Mat\'ern parameters       & $c_1$ & $\al$  & $c_1^2$ & $c_2$ & $\be$ & $\ga$ & r \\ \hline
$\nu=1$, $\lambda = 0.1$, $v=0.5$  & 5.21e-02 & 1.85 & 2.72e-03 & 4.13e-04 & 3.69   & 1.83 & 2.76 \\ \hline
$\nu=1.5$, $\lambda = 0.1$, $v=0.5$   & 5.52e-02 & 1.84 & 3.05e-03 & 5.13e-04 & 3.69 & 1.8 & 2.74 \\ \hline
$\nu=1.5$, $\lambda = 0.2$, $v=0.5$  &  5.84e-02 & 1.86 &  3.42e-03 & 9.67e-04 & 3.73  &  1.79 &  2.76  \\ \hline
$\nu=1.5$, $\lambda = 0.1$, $v=1$  & 9.14e-02 & 1.71 & 8.36e-03 & 1.98e-03 & 3.39   & 1.78 & 2.59 \\ \hline
\end{tabular}
\caption{Estimates for the parameters from Equations \eqref{eq:QCLMC_meandecay}, \eqref{eq:QCLMC_variancedecay} and \eqref{eq:QCLMC_costincrease} for different values of the hyperparameters $\nu, \lambda$ and $v$ of the log-Gauss random coefficient \eqref{eq:matern_kernel} for $M=500$ independent PDE samples, generated by a pseudo-random number generator. The error estimator for the adaptive refinement procedure is defined via Equation \eqref{eq:error_estimator} and $J = 11$ adaptive refinement steps are used, see Remark \ref{rem:refinement}}
\label{tab:parameter_estimates_log_gauss}
\end{table}
\subsection{Comparison of upper bounds to the MSE}
\label{subsec:comparison_of_upper_bounds_to_the_mse}
We compare the theoretical performance of QCLMC and unbiased CLMC (cf.~\cite{Detommaso2019_CLMC, BeschleBarth2023_QCLMC}) based on the derived upper bounds to the MSE from the proofs of their respective complexity theorems. The parameter estimates given in Table~\ref{tab:parameter_estimates_log_gauss} provide the convergence regime. All experiments share the properties $\ga \approx 2$ as opposed to $\ga \approx 1$, which is usually expected by an optimal direct solver to solve a $2$-dimensional PDE problem. But as indicated in Remark \ref{rem:adapted_cost_growth_assumption} the rate $\ga$ in Equation \eqref{eq:QCLMC_costincrease} scales with the average growth of the computed levels $\l_j^\k$. The a-posteriori error estimator from Equation \eqref{eq:error_estimator} is actually an upper bound to the samplewise strong error $\|u^\k - u_j^\k\|_{H^1(\calD)}$ and hence converges with halve the rate as the samplewise weak error $|\|u^\k\|_{H^1(\calD)} - \|u_j^\k\|_{H^1(\calD)}|$ leading to a decreased growth of the levels $\l_j^\k$ for each sample $k \in \N$ over the refinements $j \in \N$. Overall, it still holds $\gamma < \min\{\be, 2\al\}$, since $\al$ and $\be$ are scaled the same way. Further, for all upcoming numerical experiments the truncation index for the KL expansion \eqref{eq:log_gauss_coefficient} is $R=36$ and we choose $r = ({\ga + \min\{\be, 2\al\}})/{2}$ to satisfy the assumption on $r$ for QCLMC and CLMC from their complexity theorems. 
Next, we state the bias and variance bounds from the proof of the complexity theorem for QCLMC in dependence of the problem parameters $\al, \be, c_1, c_2, r$.
The bias is given by
\begin{equation}
 \begin{aligned}
\left|\E[\widehat{Q}_{0, \Lm}^\QCLMC] - (\calQ - Q(0))\right| \leq & \  \frac{c_{disc} c_1}{M} \frac{1}{r-\al}\left(e^{(r-\al)(\Lm \land \bar{L})} -1 \right) + \frac{c_1}{\al} e^{-\al (\Lm \land \bar{L})} \\
= & \ \text{discrepancy bias term} + \text{standard bias term},
 \end{aligned}
 \label{eq:bias_QCLMC}
\end{equation}
where the first term stands for the additional bias introduced by the $F$-discrepancy of the quasi-random sequence (see Lemma~\ref{lma:inverse_transform_and_convergence}). The variance is bounded by
\begin{equation}
 \begin{aligned}
 \V\left[\widehat{Q}_{0, \Lm}^\QCLMC\right] \leq & \ \frac{c_{disc} c_2}{M^2} \left(\frac{4}{(2r - \be)^2} \left(e^{(2r-\be)(\Lm \land \bar{L})} - e^{(r-\frac{\be}{2})(\Lm \land \bar{L})} + 1\right)\right) \\
 & \ + \frac{c_2}{M}  \left(\frac{2}{(r-\be) (r - \frac{\be}{2})} e^{(r-\be) (\Lm \land \bar{L})} +  \frac{4}{\be(r - \frac{\be}{2})} e^{- \frac{\be}{2} (\Lm \land \bar{L})} +  \frac{4}{(\be - r) \be} \right) \\
 = & \ \text{discrepancy variance term} + \text{variance convergence term},
 \end{aligned}
 \label{eq:var_QCLMC}
\end{equation}
where we see the split in an additional term introduced by the $F$-discrepancy of the quasi-random sequence and the term stemming from the assumption on the convergence of the variance decay in the complexity theorem  (see~\eqref{eq:QCLMC_variancedecay}).
The MSE, see Equation \eqref{eq:MSE}, of QCLMC consists of the variance and the squared bias of the estimator
\begin{equation}
 \MSE^\QCLMC = \V\left[\widehat{Q}_{0, \infty}^\QCLMC\right] + \left|\E[\widehat{Q}_{0, \infty}^\QCLMC - (\calQ - Q(0))]\right|^2,
 \label{eq:mse_QCLMC}
\end{equation}
and it is bounded by the respective variance upper bound (see~\eqref{eq:var_QCLMC}) and bias upper bound (see~\eqref{eq:bias_QCLMC}).
For CLMC the bias is bounded by, cf.~\cite{Detommaso2019_CLMC, BeschleBarth2023_QCLMC},
\begin{equation}
 \begin{aligned}
\left|\E[\widehat{Q}_{0, \Lm}^\CLMC - (\calQ - Q(0)]\right| = \left|\E[\calQ - Q(\Lm)]\right| \leq & \ \frac{c_1}{\al} e^{-\al \Lm} \\
= & \ \text{standard bias term},
 \end{aligned}
 \label{eq:bias_CLMC}
\end{equation}
which vanishes in the case $\Lm=\infty$. For the variance we have 
\begin{equation}
 \begin{aligned}
 \V & \left[\widehat{Q}_{0, \Lm}^\CLMC\right]   \leq\frac{c_2}{M} \left[\frac{1}{(r - \frac{\be}{2})^2} \left(\frac{2r - \be}{r - \be} e^{(r - \be) \Lm}  + \frac{4 r - 2\be}{\be} e^{-\frac{\be}{2} \Lm} \right) +  \frac{4}{(\be - r)\be} \right]\\
  & \ + \frac{c_1^2}{M} \left[\frac{1}{(r-\al)^2} \left(\ \frac{2r - 2\al}{r - 2\al} e^{(r - 2\al)\Lm}  - \frac{(r-\al)^2}{\al^2}e^{-2\al \Lm}  +  \frac{2 r^2 - 2 r\al}{\al^2}e^{-\al \Lm}\right) + \frac{r}{(2\al - r)\al^2}  \right], 
 \end{aligned}
  \label{eq:var_CLMC}
\end{equation}
which in the case $\Lm=\infty$ and $r < \min\{\be,2\al\}$ boils down to
\begin{equation}
 \begin{aligned}
 \V\left[\widehat{Q}_{0, \infty}^\CLMC\right] \leq & \ \frac{1}{M}\frac{4c_2}{(\be - r)\be} +  \frac{1}{M}\frac{c_1^2 r}{(2\al - r)\al^2} \\
 = & \ \text{variance convergence term} + \text{bias convergence term}.
 \end{aligned}
 \label{eq:var_CLMC_infty}
\end{equation}
Here, the first term corresponds to the assumption on the variance decay and the second to the assumption on the bias decay, see~\eqref{eq:CLMCassumptions}.
The MSE of CLMC is given by
\begin{equation}
 \MSE^\CLMC = \begin{cases}
                \V\left[\widehat{Q}_{0, \Lm}^\CLMC\right] + \left|\E[\calQ - Q(\Lm)]\right|^2 &\quadfor \Lm \neq \infty,\\
                \vspace{0.01cm}\\
                \V\left[\widehat{Q}_{0, \infty}^\CLMC\right] &\quadfor \Lm = \infty,
              \end{cases}
 \label{eq:mse_CLMC}
\end{equation}
which is bounded in the respective case by the variance upper bound (see~\eqref{eq:var_CLMC} and \eqref{eq:var_CLMC_infty}) and bias upper bound (see~\eqref{eq:bias_CLMC}).
For our first performance comparison of both methods, we compare the upper bounds to the MSE based on the parameter estimates for the different hyperparameter settings for the log-Gauss PDE coefficient as listed in Table \ref{tab:parameter_estimates_log_gauss} for a range of sample sizes $M = 16 \cdot 2^i$ for $i=0,1,\dots,9$.
For CLMC we set $\Lm = \infty$ resulting in the unbiased version.
 For QCLMC we set $\Lm = \infty$ as well and thus the upper bounds for the bias and variance are independent of $\Lm$, but depend on $\bar{L}$. For QCLMC we compute for each hyperparameter setting in the PDE coefficient an average of the upper bounds over $100$ independent runs of a quasi-random sequence $L_r^\k$ yielding different values for $\bar{L}$ in each run.
 The independence of the quasi-random sequence in QCLMC was realized by Owen scrambling (see~\cite{Owen1995_Scrambling, Owen1998_Scrambling}) of a Sobol sequence.
 For QCLMC the maximal generated level $\bar{L}$ in each run, see Figure \ref{fig:L_max_and_bias_QCLMC} (left), dictates the bias, illustrated in Figure \ref{fig:L_max_and_bias_QCLMC} (right). We see that the additional bias introduced by the quasi-random sequence is significant, but decays faster than $M^{-\frac{1}{2}}$ resulting in a diminishing contribution to the upper bound to the MSE by the squared bias, emphasizing the automatic compensation of the bias error by $\bar{L}$ as described in Remark \ref{rem:complexity_L_infty}. The upper bounds to the variances of both methods are given in Figure \ref{fig:variance_and_MSE_CLMC_QCLMC} (left), where we observe that QCLMC achieves a much lower upper bound to the variance than CLMC. The upper bound to the variance of CLMC is dominated by the bias convergence term. The discrepancy variance term decays with at a faster rate and is therefore not dominant in the QCLMC estimate. This means that the upper bound to the variance of the QCLMC estimator is essentially only the variance convergence term, whereas the CLMC estimator is dominated by the bias convergence term. The resulting upper bound to the MSE for both methods is shown in Figure \ref{fig:variance_and_MSE_CLMC_QCLMC} (right). We observe a smaller upper bound to the MSE for QCLMC in comparison to CLMC, as a direct consequence of the variance reduction and the natural bias compensation. 
 In Figures~\ref{fig:bias_variance_and_MSE_CLMC_QCLMC_1} to~\ref{fig:bias_variance_and_MSE_CLMC_QCLMC_3} we see similar effects. The constant in the upper bound to the variance of QCLMC is influenced by the constant $c_2$ from the variance decay assumption \eqref{eq:QCLMC_variancedecay}, where the upper bound to the variance of CLMC is influenced by both $c_2$ and $c_1^2$ from the bias decay assumption \eqref{eq:QCLMC_meandecay} and both converge in $M$ with rate one. The constant $c_1$ enters in QCLMC only in the upper bound to the bias, and the squared bias converges like $M^{-0.66 \cd 2} < M^{-1}$ for the given examples. This means that for a larger quotient of the constants $\frac{c_1^2}{c_2}$ given in Table~\ref{tab:parameter_estimates_log_gauss} we see a better result for QCLMC compared to CLMC. This behaviour may be explained, since $c_2 > c_1^2$ resembles a high variance of the problem relative to the squared bias, leading to large sample sizes necessary to reduce the variance of the estimator and the effect of accurately sampling the level distribution $L_r$ by fewer samples becomes less significant.
%
%
\begin{figure}[tbhp]
    \includegraphics[width=0.49\textwidth]{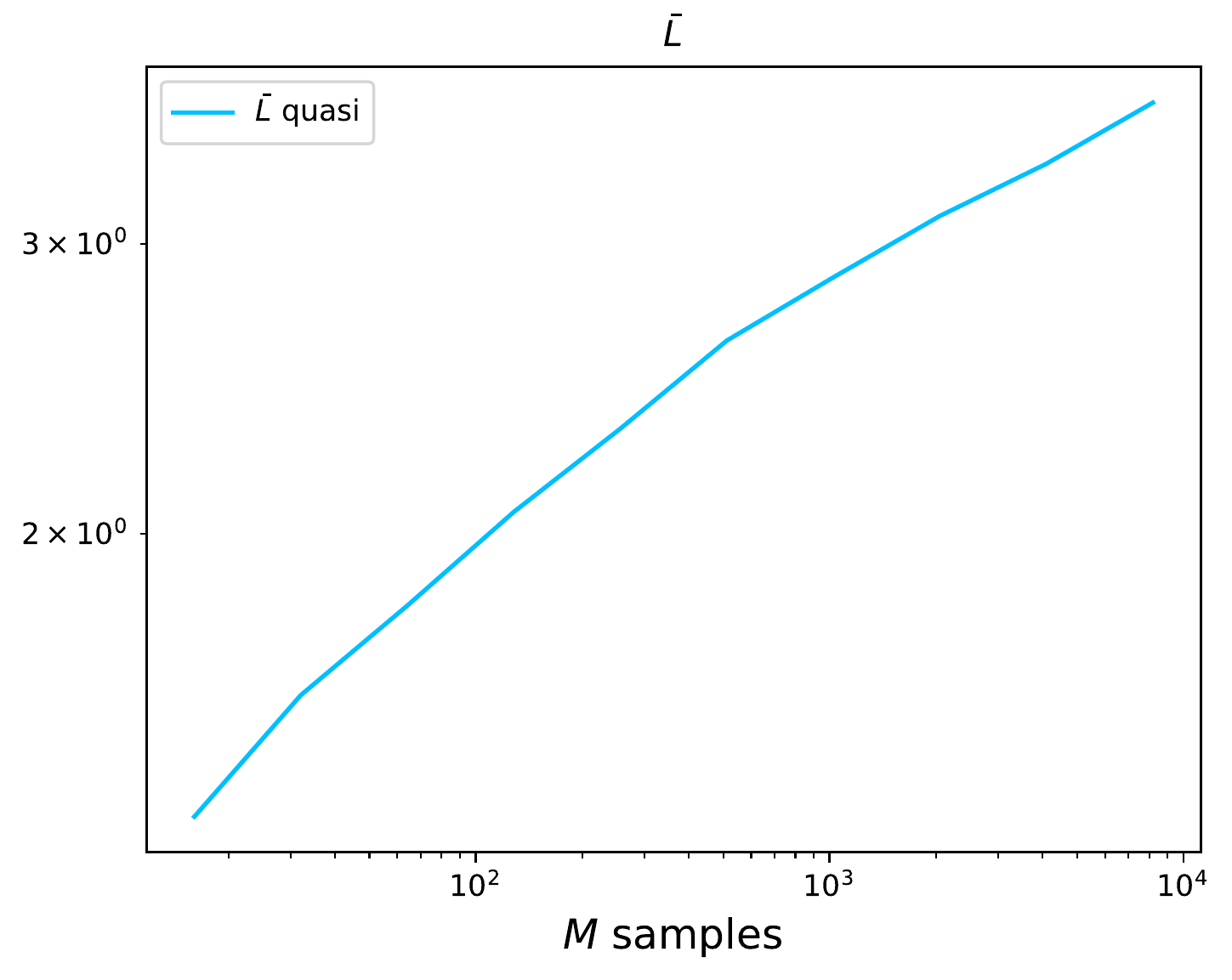}
    \includegraphics[width=0.49\textwidth]{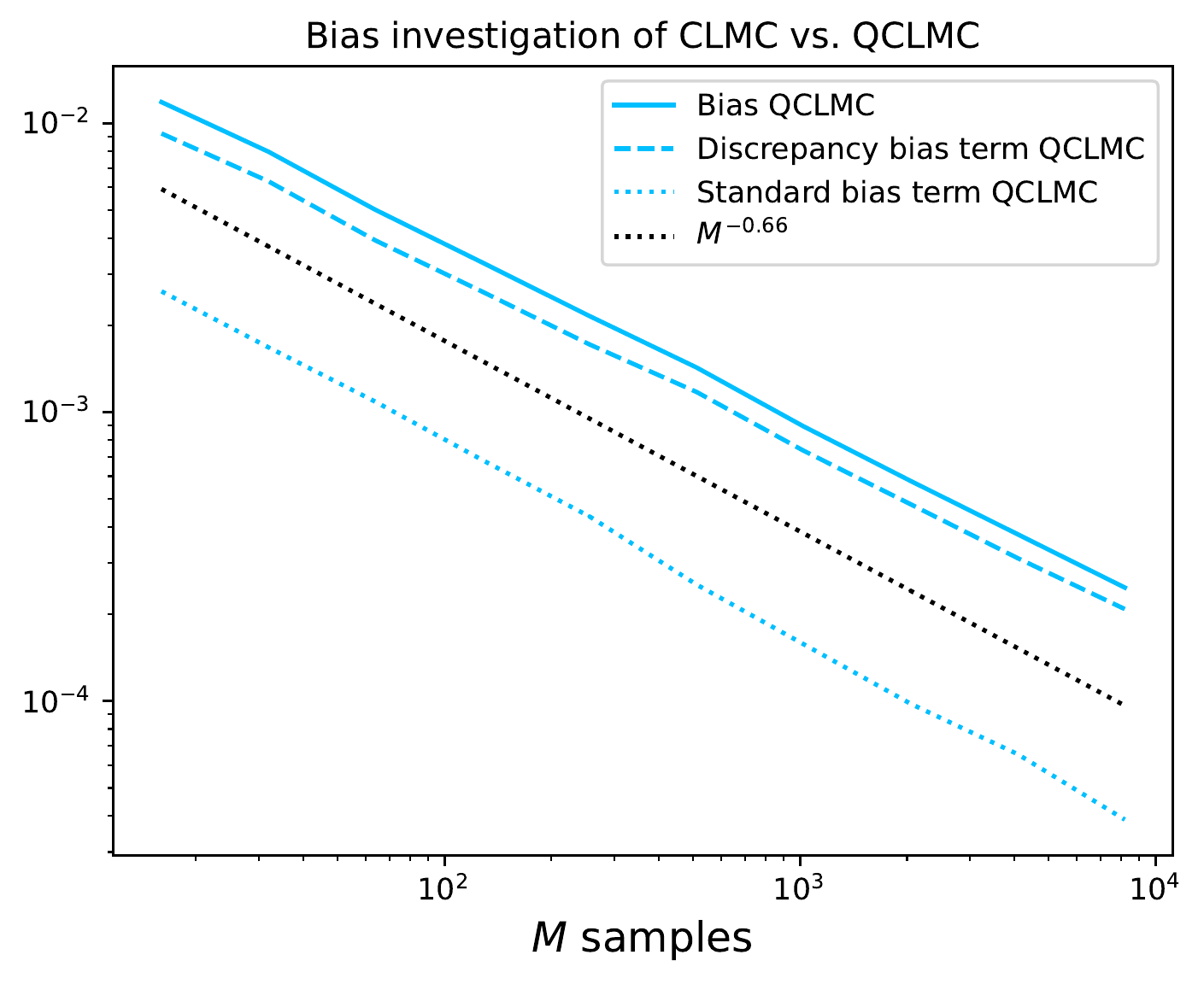}
  \caption{Mean of maximal levels $\bar{L}$ generated by quasi-random Sobol sequence for $M$ samples (left) and corresponding mean of the upper bound to the bias of QCLMC (right) over $100$ independent runs realized via Owen Scrambling.}
\label{fig:L_max_and_bias_QCLMC}
\end{figure}
\begin{figure}[tbhp]
    \includegraphics[width=0.49\textwidth]{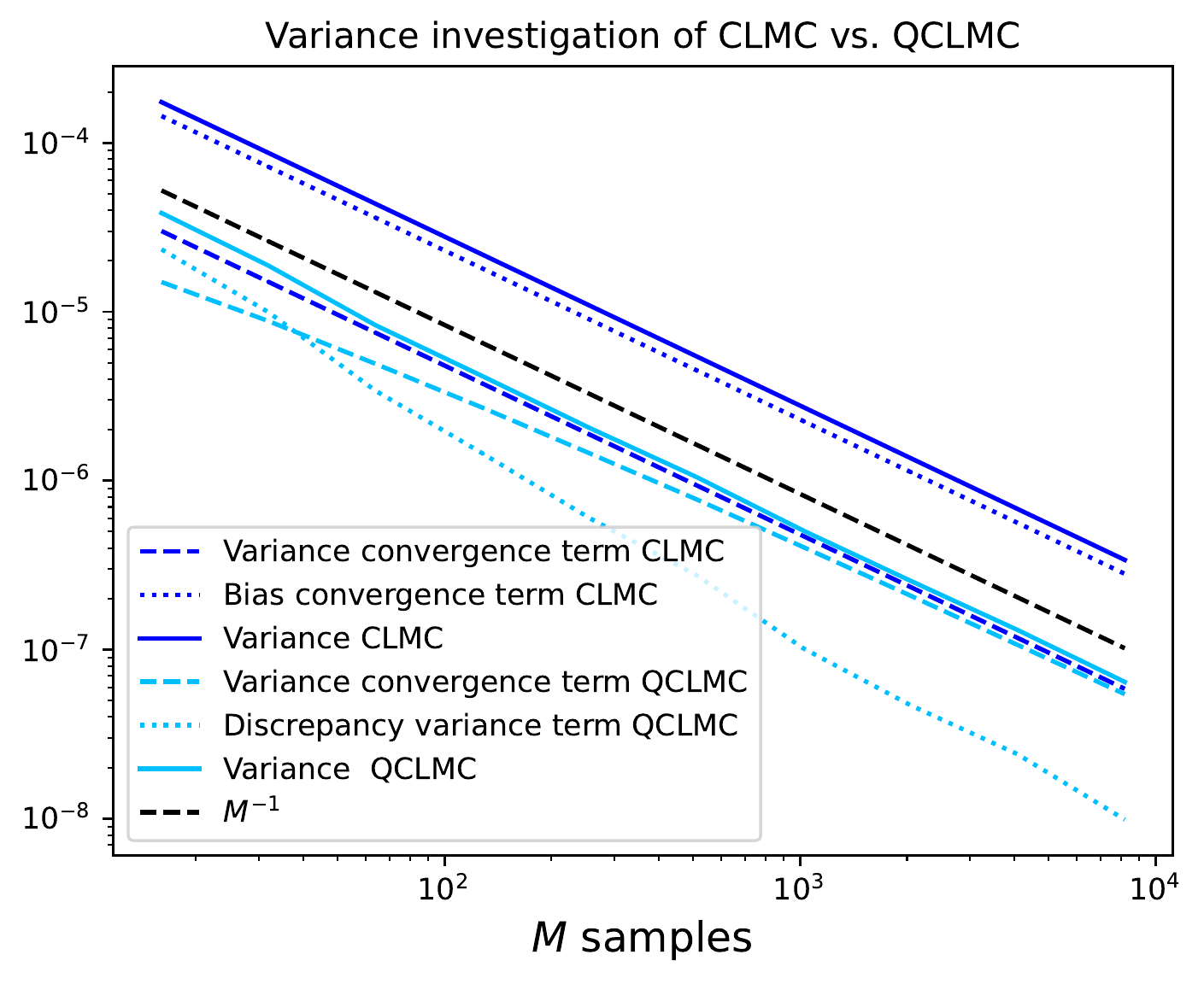}
    \includegraphics[width=0.49\textwidth]{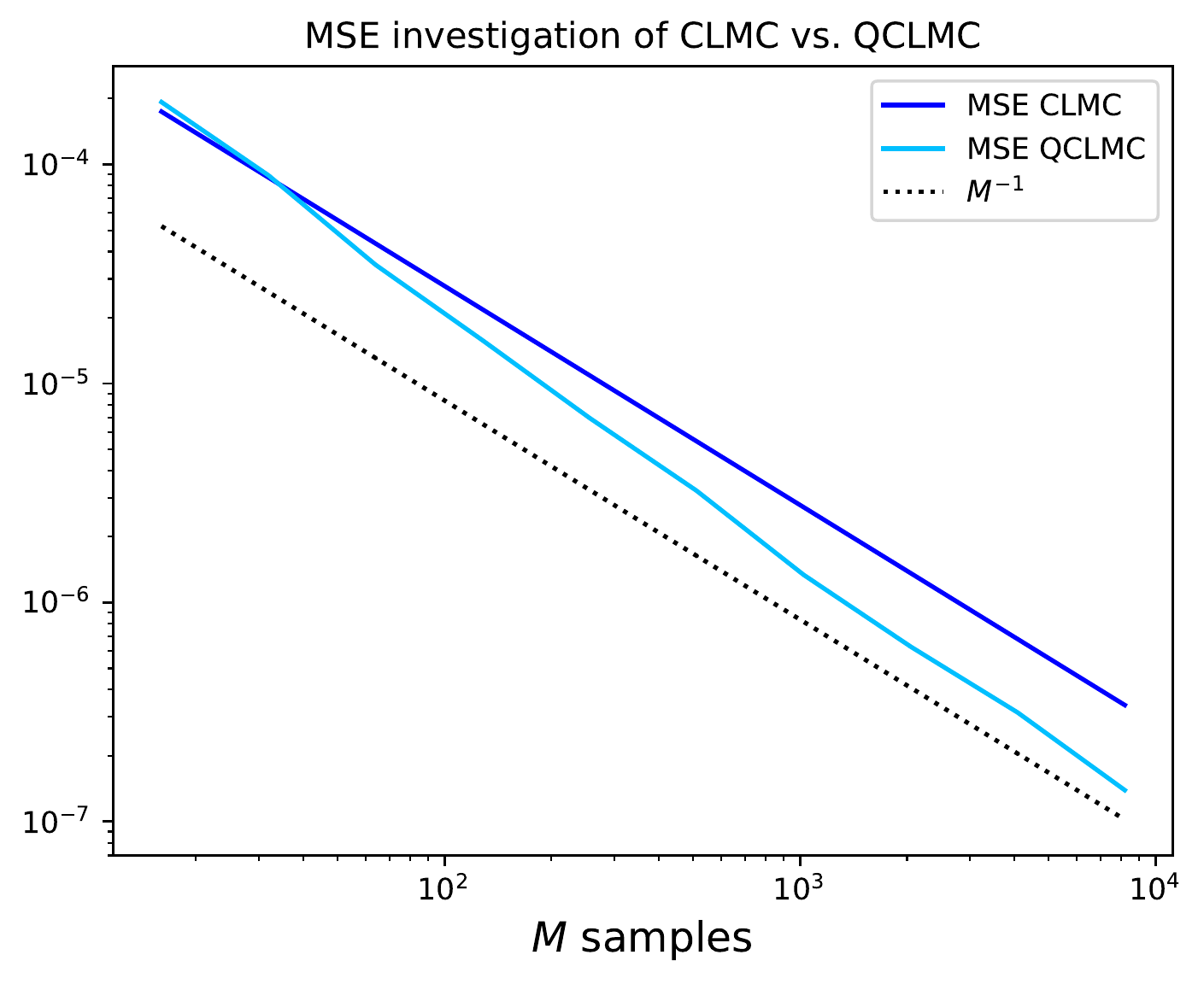}
\caption{Upper bounds to the variance (left) and MSE (right) for CLMC and QCLMC estimated over $100$ independent runs. Hyperparameters for log-Gauss field \eqref{eq:matern_kernel}: $\nu=1$, $\lambda = 0.1$, $v=0.5$.}
\label{fig:variance_and_MSE_CLMC_QCLMC}
\end{figure}

\begin{figure}[tbhp]
    \includegraphics[width=0.32\textwidth]{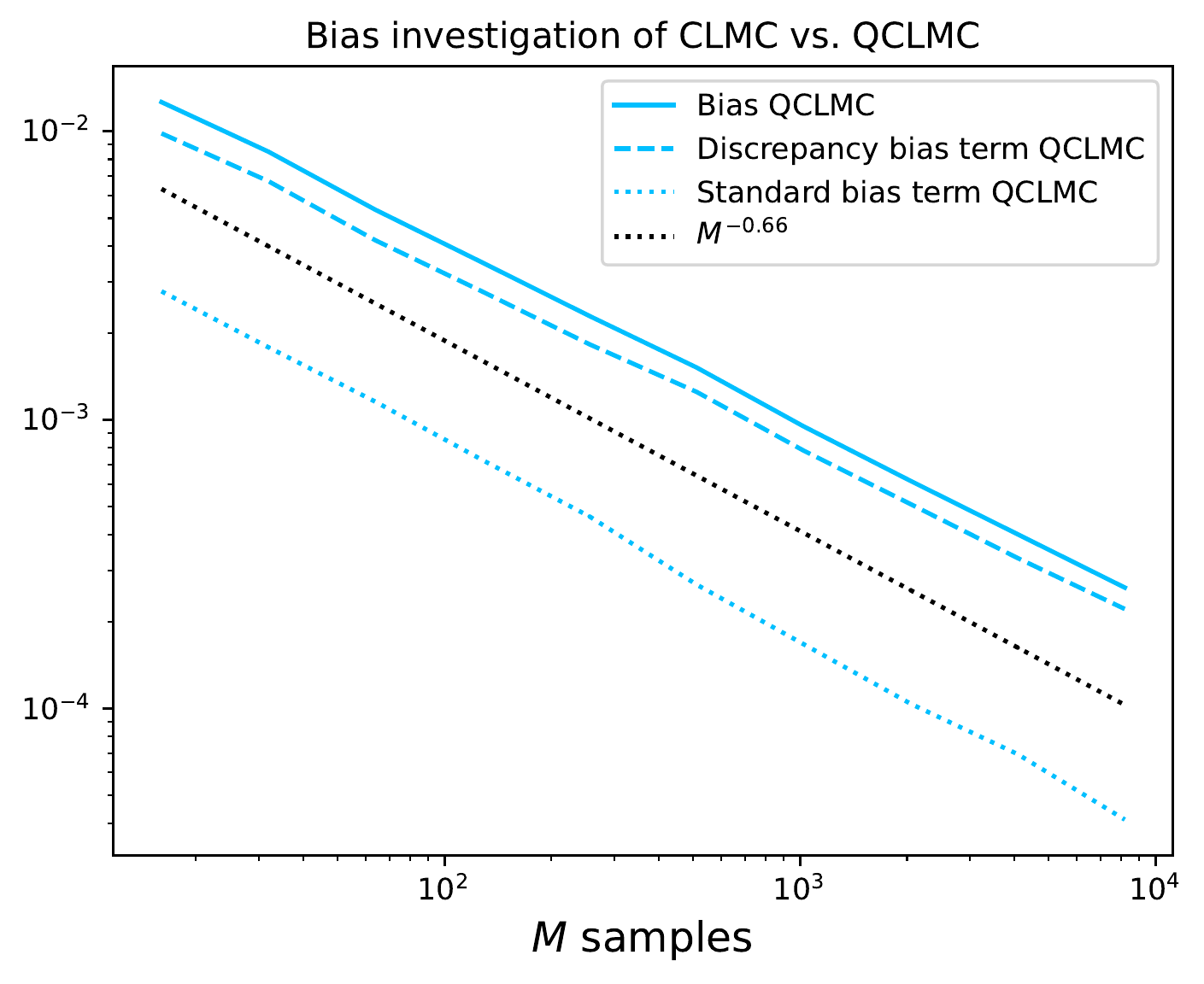}
    \includegraphics[width=0.32\textwidth]{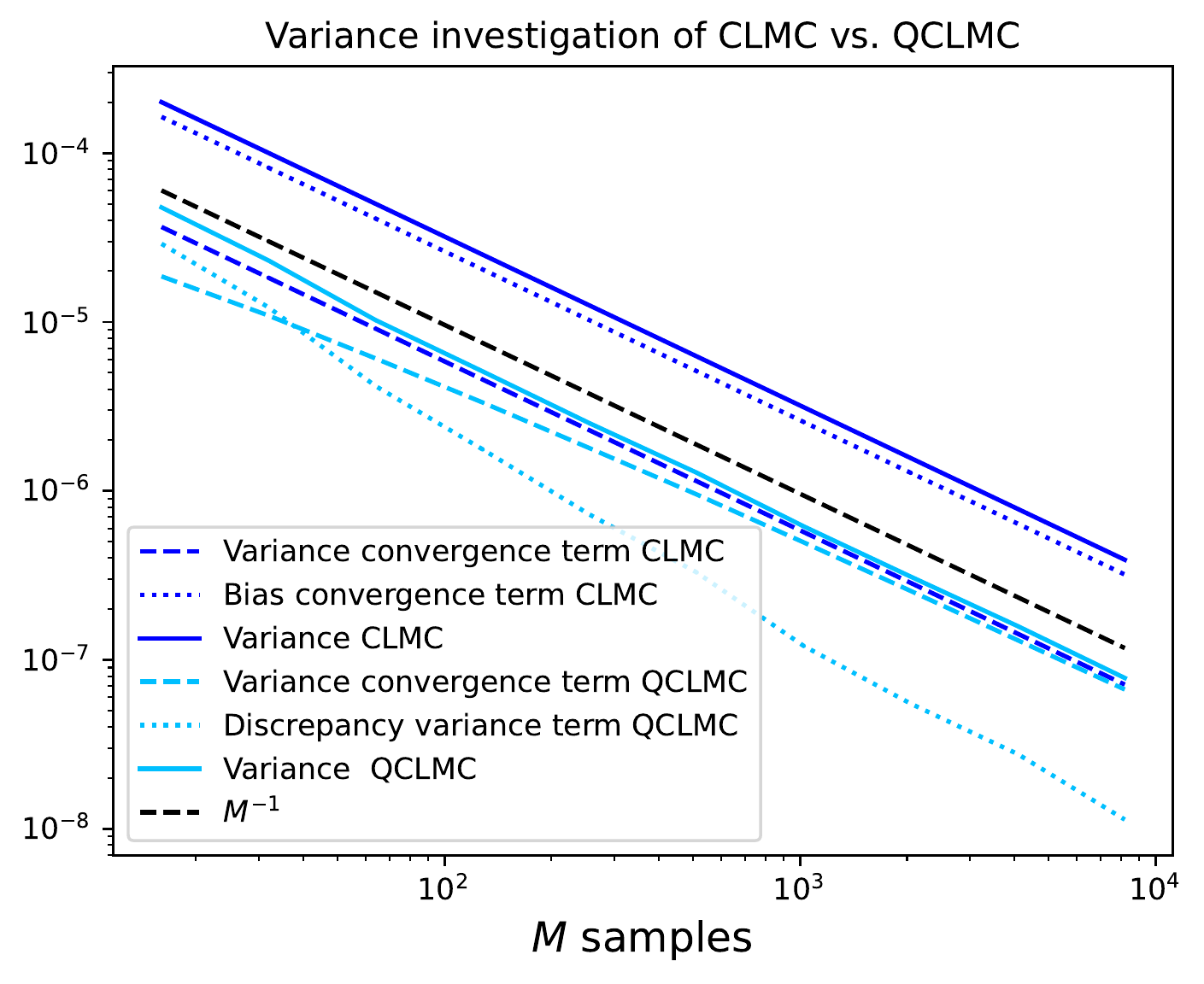}
    \includegraphics[width=0.32\textwidth]{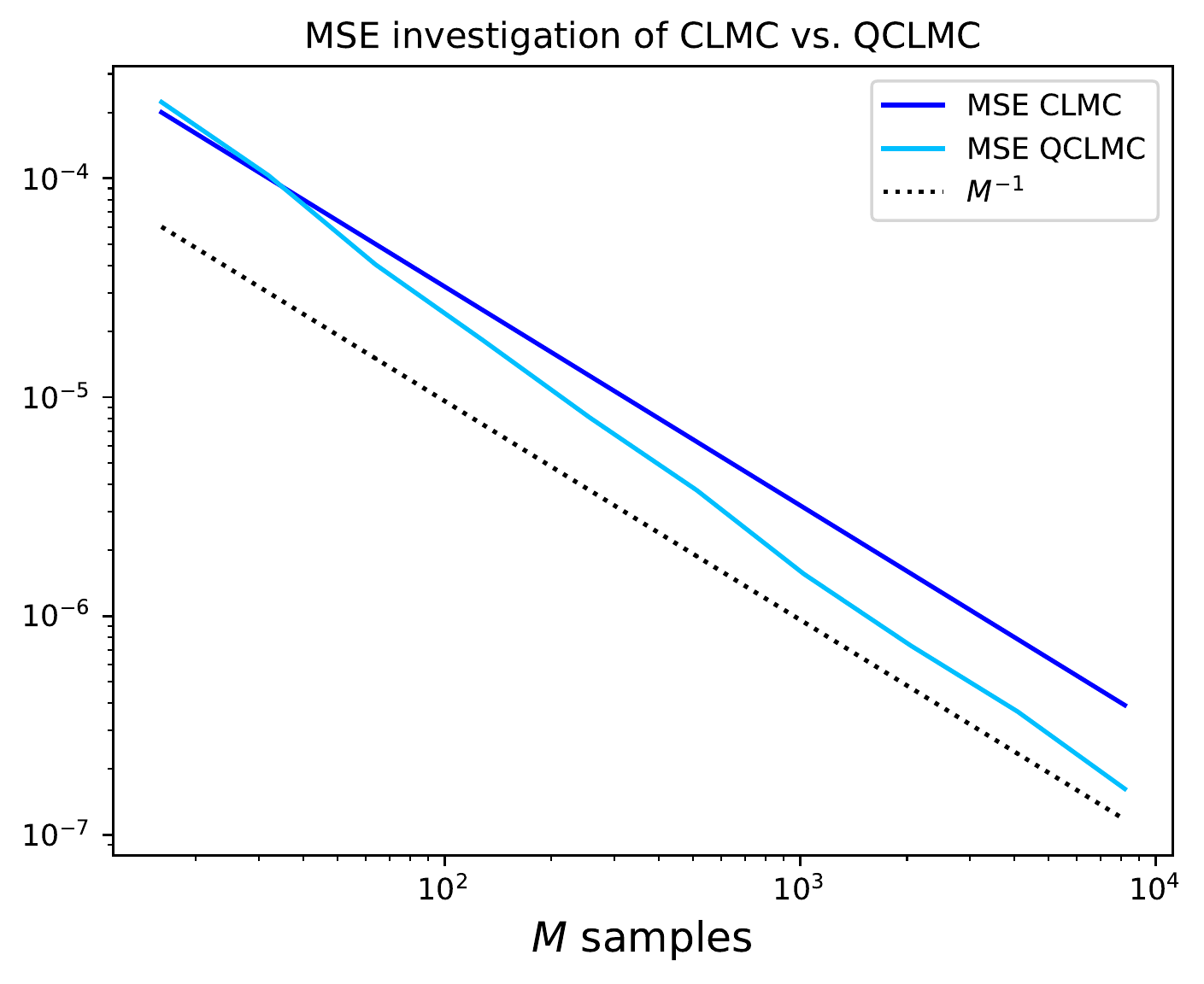}
\caption{Upper bounds to the bias (left), variance (middle) and MSE (right) for CLMC and QCLMC estimated over $100$ independent runs. Hyperparameters for log-Gauss field \eqref{eq:matern_kernel}: $\nu=1.5$, $\lambda = 0.1$, $v=0.5$.}
\label{fig:bias_variance_and_MSE_CLMC_QCLMC_1}
\end{figure}

\begin{figure}[tbhp]
    \includegraphics[width=0.32\textwidth]{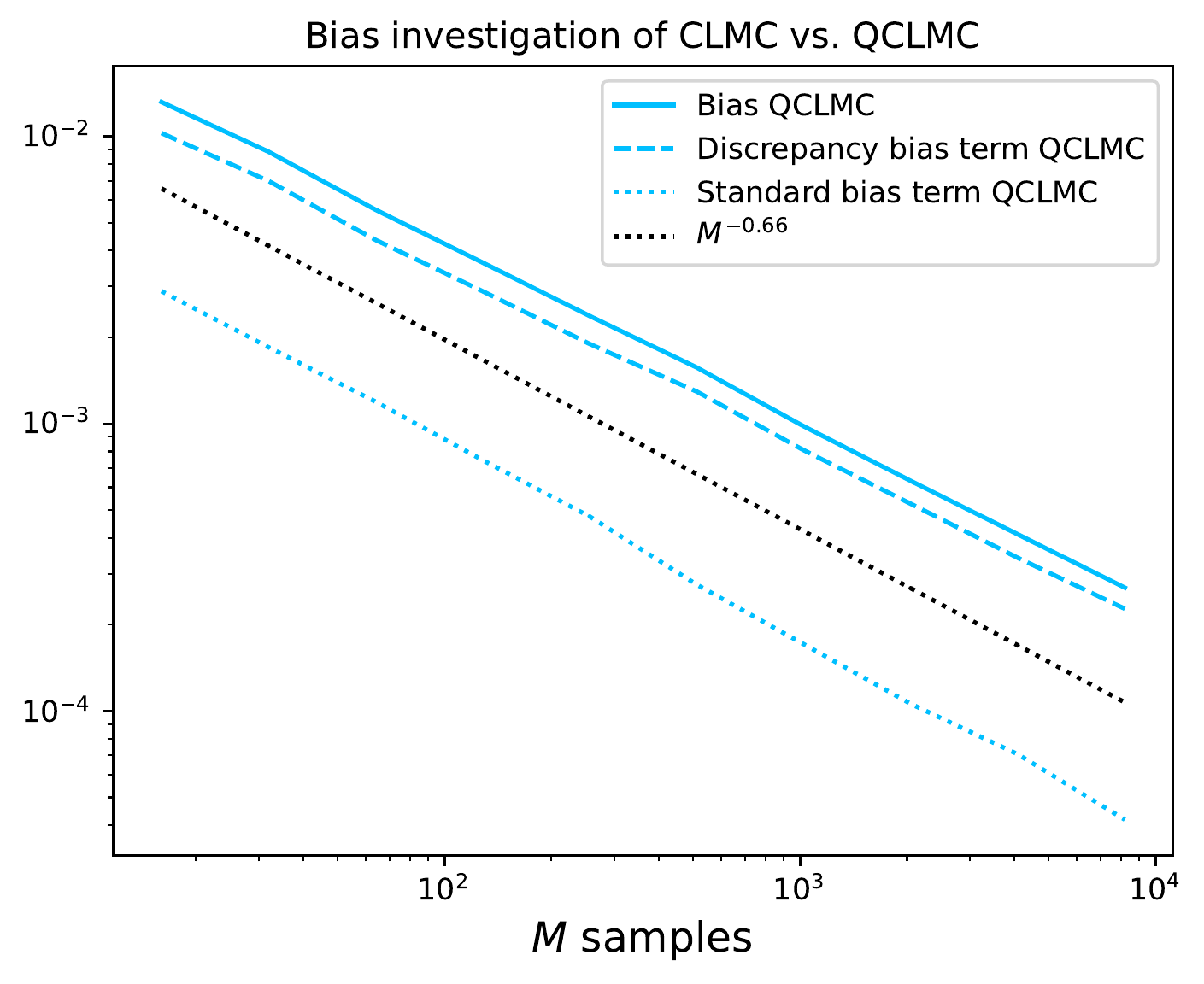}
    \includegraphics[width=0.32\textwidth]{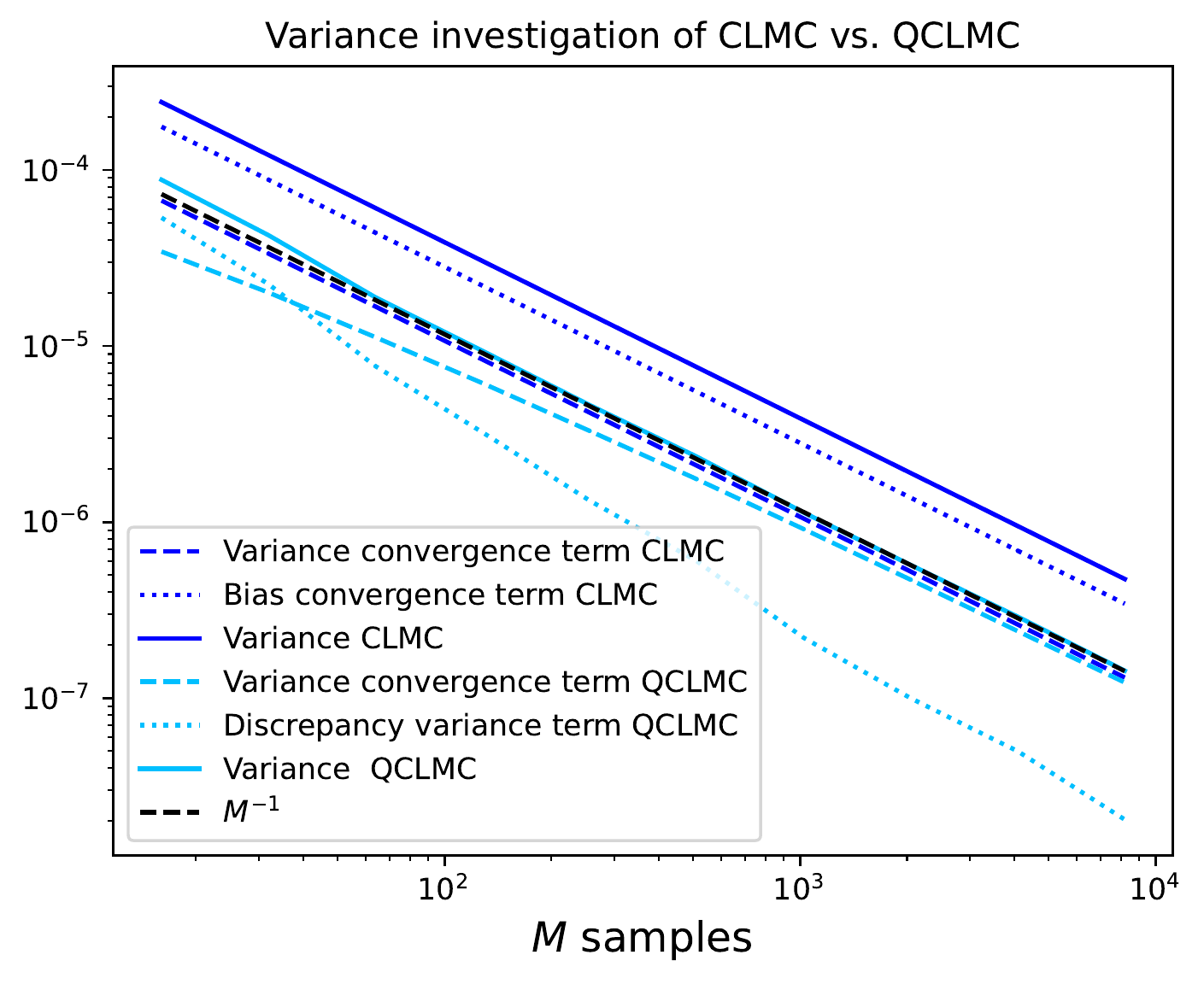}
    \includegraphics[width=0.32\textwidth]{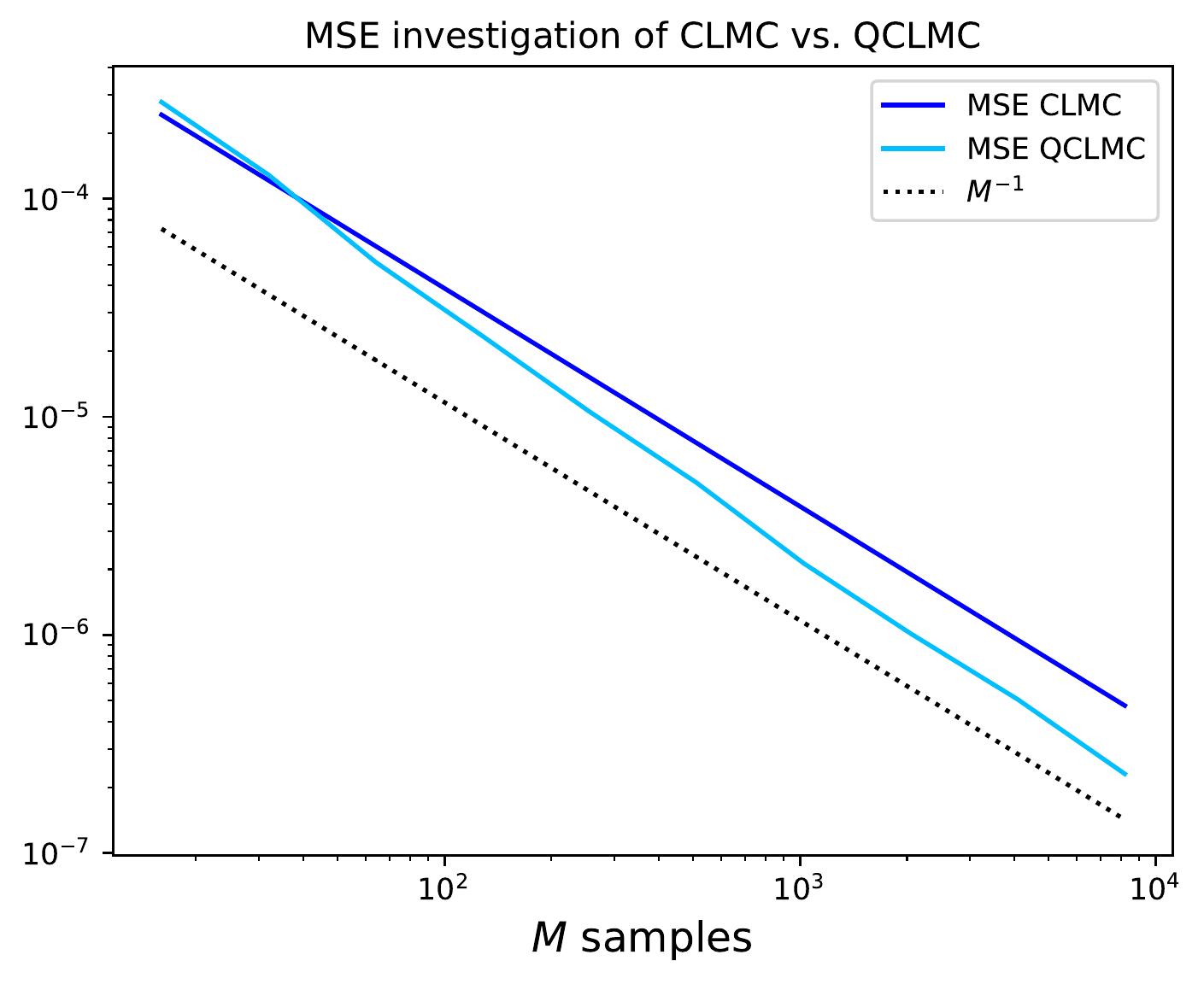}
\caption{Upper bounds to the bias (left), variance (middle) and MSE (right) for CLMC and QCLMC estimated over $100$ independent runs. Hyperparameters for log-Gauss field \eqref{eq:matern_kernel}: $\nu=1.5$, $\lambda = 0.2$, $v=0.5$.}
\label{fig:variance_and_MSE_CLMC_QCLMC_2}
\end{figure}

\begin{figure}[tbhp]
      \includegraphics[width=0.32\textwidth]{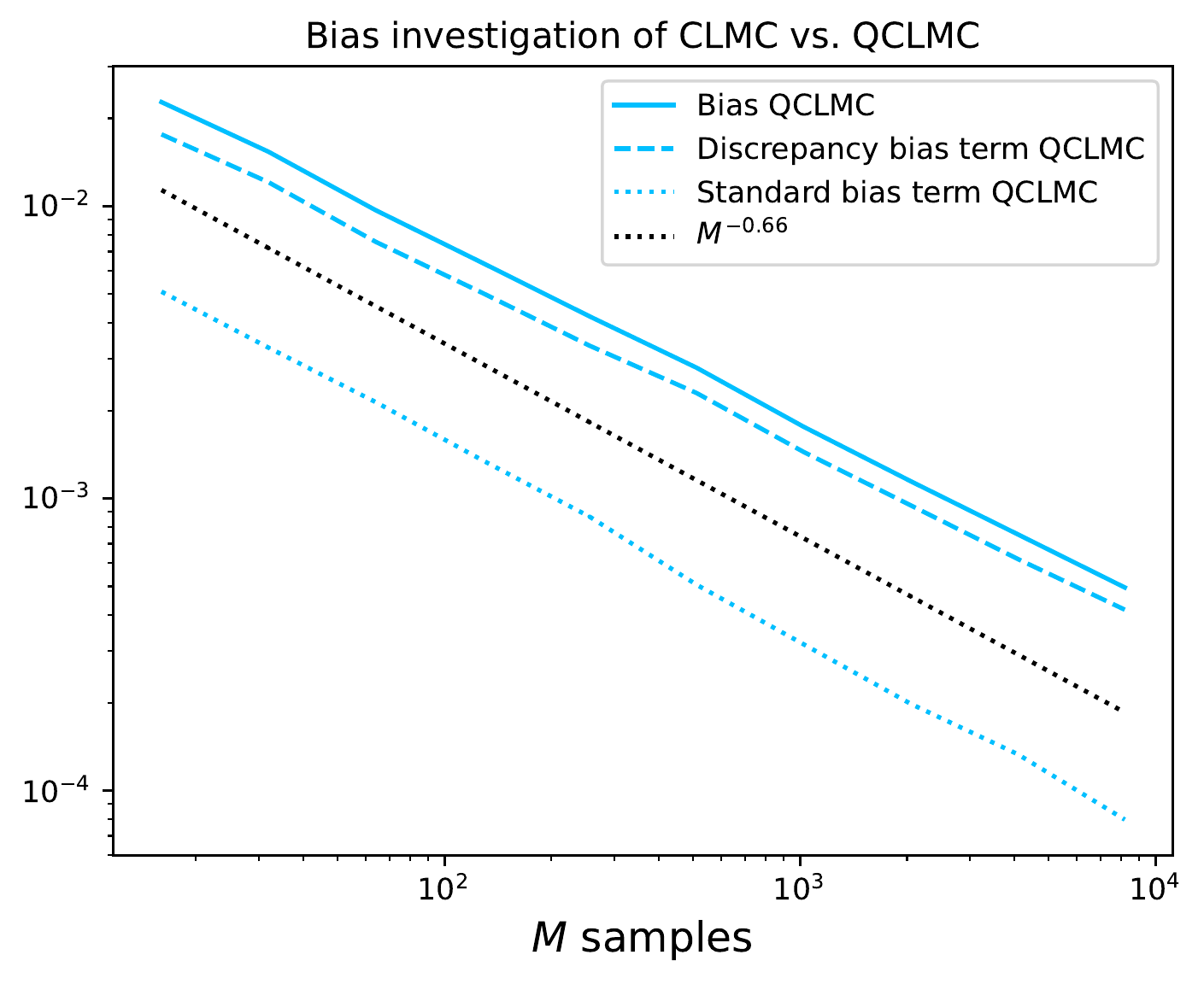}
    \includegraphics[width=0.32\textwidth]{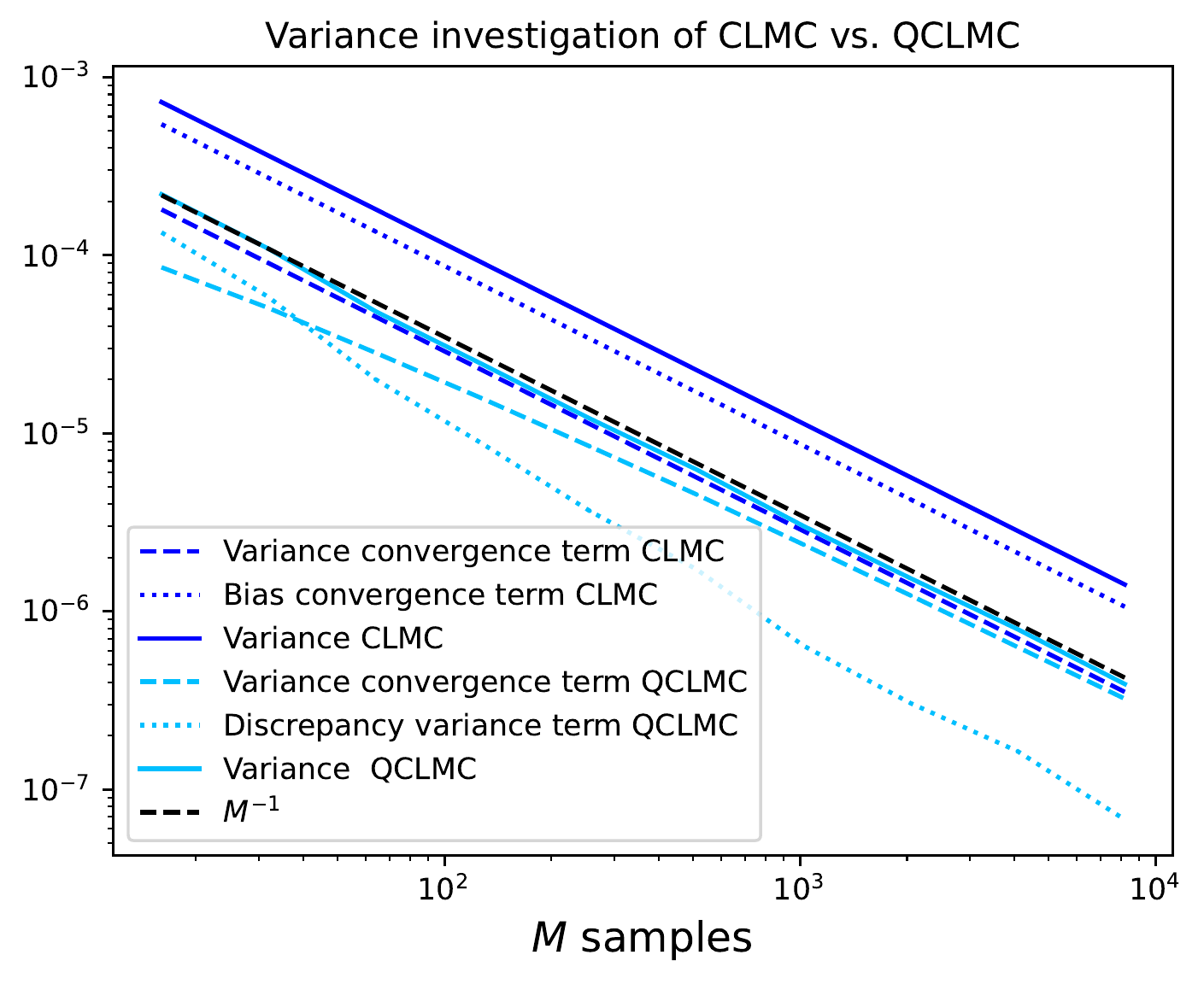}
    \includegraphics[width=0.32\textwidth]{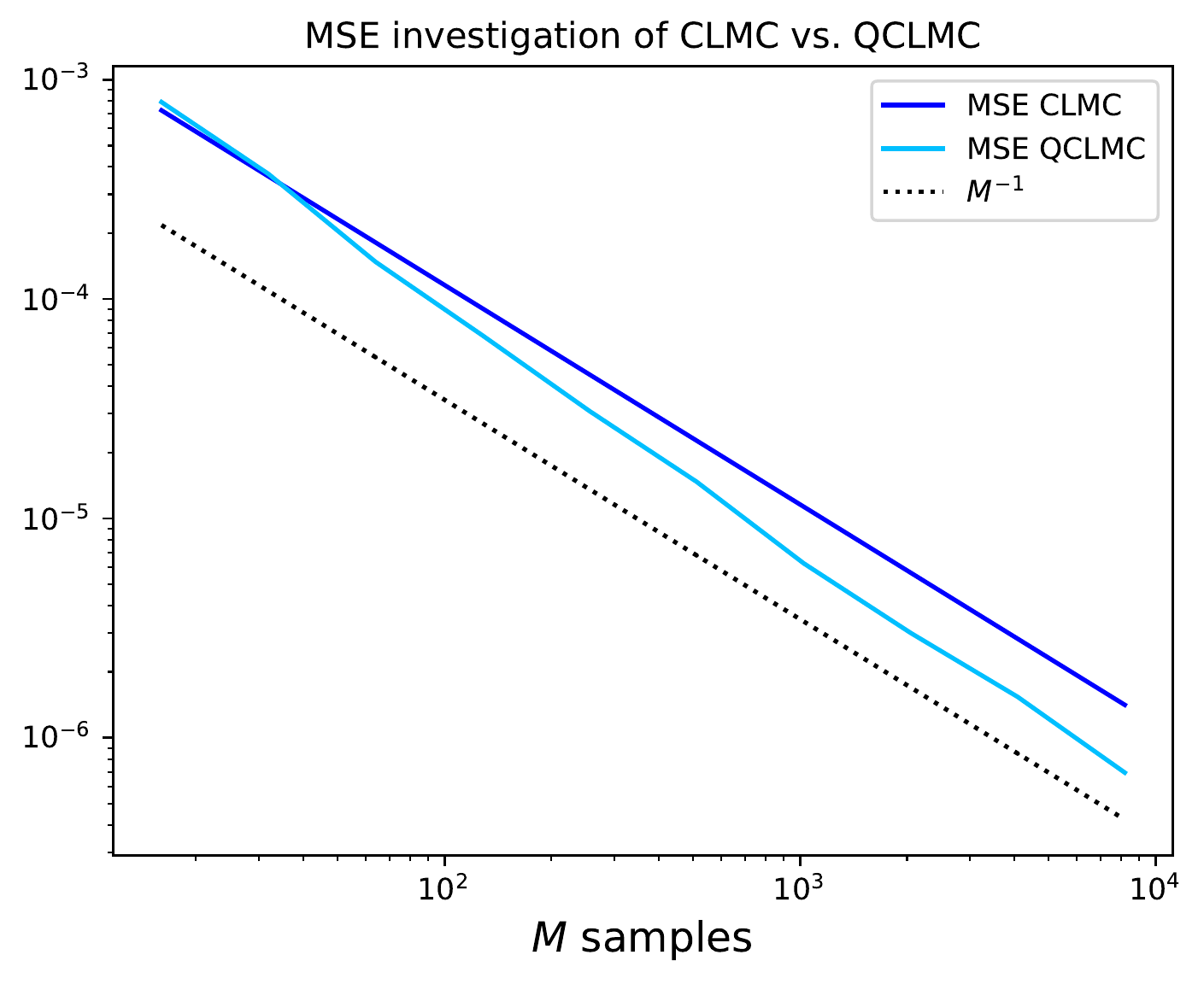}
\caption{Upper bounds to the bias (left), variance (middle) and MSE (right) for CLMC and QCLMC estimated over $100$ independent runs. Hyperparameters for log-Gauss field \eqref{eq:matern_kernel}: $\nu=1.5$, $\lambda = 0.1$, $v=1$.}
\label{fig:bias_variance_and_MSE_CLMC_QCLMC_3}
\end{figure}

\subsection{Performance comparison of CLMC and QCLMC}
\label{subsec:performance_comparison_of_clmc_and_qclmc}
With the practical estimator defined in Section \ref{subsec:pratical_estimator_a_posteriori_error_and_parameter_estimates} at hand, an algorithm for CLMC and QCLMC to compare their real performance is formulated in Algorithm \ref{alg:qclmc}. The algorithm is defined for $\Lm = \infty$ and $\ga < \min\{\be, 2\al\}$.
\begin{algorithm}
\caption{(Q)CLMC}
\label{alg:qclmc}
\begin{algorithmic}
\REQUIRE $r \in (\ga, \min\{\be, 2 \al\})$ the exponential distribution parameter, $M \in \N$ the total number of samples
\FOR{$k = 1:M$}
 \STATE Draw and save sample $L_r^\k \sim \Exp(r)$
 \STATE $j \gets -1$
 \STATE $\l_{tmp} \gets 0$
 \WHILE{$\l_{tmp} \leq L_r^\k$}
        \STATE $j \gets j+1$
        \STATE Evaluate and save sample $Q_j^\k$
        \STATE Evaluate a-posteriori error estimate $e_j^\k$ of $|\calQ^\k - Q_j^\k|$
        \STATE Compute and save level $\l_j^\k = -\ln(e_j^\k / e_0^\k)$
        \STATE $\l_{tmp} \gets \l_j^\k$
 \ENDWHILE
 \STATE Save $J^\k \gets j$
 \STATE Save $\tilde{\l}_j^\k = \min\{L_r^\k, \l_j^\k\}$
 \ENDFOR
\STATE $\widehat{Q}_{0,\infty}^\QCLMC = \frac{1}{M} \sum_{k=1}^M \sum_{j=1}^{J^\k} \frac{\exp(r \tilde{\l}_j^\k) - \exp(r \l_{j-1}^\k)}{r\left(\l_{j}^\k - \l_{j-1}^\k\right)} \left(Q_j^\k - Q_{j-1}^\k\right)$
 \end{algorithmic}
 Note, that for CLMC the samples $(L_r^\k;\;k=1,...,M)$ are drawn with a pseudo-random number generator, e.g., in Python with \verb+numpy.random+ \cite{Numpy2020} and for QCLMC with a quasi-random number generator and the inverse transformation from Remark \ref{rem:F_discrepancy_exponential}, e.g., in Python with \verb+scipy.qmc+ \cite{Scipy2020}, where independent sequences may be generated by Owen scambling, cf.~\cite{Owen1995_Scrambling, Owen1998_Scrambling}.
\end{algorithm}
This algorithm is used to evaluate the performance of CLMC and QCLMC for a sequence of sample sizes $M_i = 16 \cdot 2^i$ for $i=0,1,\dots,9$. We compute $K = 100$  independent runs for each of the sample sizes and each method and estimate the respective achieved MSE by
\begin{equation*}
 \MSE^{\text{(Q)CLMC}} = \E\left[\left(\widehat{Q}_{0,\infty}^{\text{(Q)CLMC}} - \E[\calQ - Q(0)]\right)^2\right] \approx \frac{1}{K}\sum_{k=1}^K \left(\left(\widehat{Q}_{0,\infty}^{\text{(Q)CLMC}}\right)^\k - \widehat{Q}_{ref} \right)^2,
\end{equation*}
where the reference solution $\widehat{Q}_{ref} \approx \E[\calQ - Q(0)]$ is computed by an optimized MLMC algorithm to a very small tolerance, cf.~\cite{Giles2008_MLMCPath, BeschleBarth2023_QCLMC}. The convergence results are given in Figure \ref{fig:MSE_convergence_CLMC_QCLMC}, where the $95\%$ confidence intervals are computed via the central limit theorem. We observe that both methods achieve their expected cost (measured in sample sizes) to MSE convergence rate of $-1$. As already indicated by the previous experiments, we observe a significant improvement of the MSE for the QCLMC method in comparison to the CLMC method for the same number of samples $M$. In contrast to the evaluated upper bounds to the MSE in Section \ref{subsec:comparison_of_upper_bounds_to_the_mse}, where the improvement of the upper bounds occurs only for larger values of $M$, the real estimated MSE for QCLMC is significantly reduced compared to CLMC right from the start. Comparing the MSE curves for CLMC and QCLMC in Figure \ref{fig:MSE_convergence_CLMC_QCLMC} (left) to the MSE upper bounds given in Figure \ref{fig:variance_and_MSE_CLMC_QCLMC} (right) we observe that the upper bound to the MSE for CLMC is tighter than the one for QCLMC, which gets tight for larger values of $M$. We observe the same when comparing Figure \ref{fig:MSE_convergence_CLMC_QCLMC} (right) to Figure \ref{fig:bias_variance_and_MSE_CLMC_QCLMC_3} (right). 
In Tables \ref{tab:mse_values_1} and \ref{tab:mse_values_2} the corresponding achieved MSE values are given for each method and each sample size, together with the quotient of improvement by the QCLMC method. The average quotient of improvement for the values in Table \ref{tab:mse_values_1} is about $5.6$ and the quotient between the constants is $\frac{c_1^2}{c_2} \approx 6.6$. The average quotient of improvement for the values in Table \ref{tab:mse_values_2} is about $4$ and the quotient between the constants is $\frac{c_1^2}{c_2} \approx 4.2$.

We conclude that QCLMC always significantly outperforms CLMC for the provided numerical examples and the factor by which the MSE is improved may be related to the ratio $\frac{c_1^2}{c_2}$. Overall, by looking at the upper bounds to the bias and variance, see Equations \eqref{eq:bias_QCLMC} and \eqref{eq:var_QCLMC} for QCLMC and Equation \eqref{eq:var_CLMC_infty} for CLMC, and the conducted numerical experiments in this work, it is reasonable to expect a similar performance of both methods in case $\frac{c_1^2}{c_2} \leq 1$, because the error contributions by terms including $c_2$ converge at a rate $M^{-1}$ for both methods. More importantly, we expect that QCLMC outperforms CLMC in cases where $\frac{c_1^2}{c_2} > 1$, because the error contributions by terms including $c_1^2$ converge faster than $M^{-1}$ for QCLMC. Due to the above mentioned advantages and essentially the same involved effort in the implementation, we generally recommend to use QCLMC over CLMC.

%

\begin{figure}[tbhp]
    \includegraphics[width=0.49\textwidth]{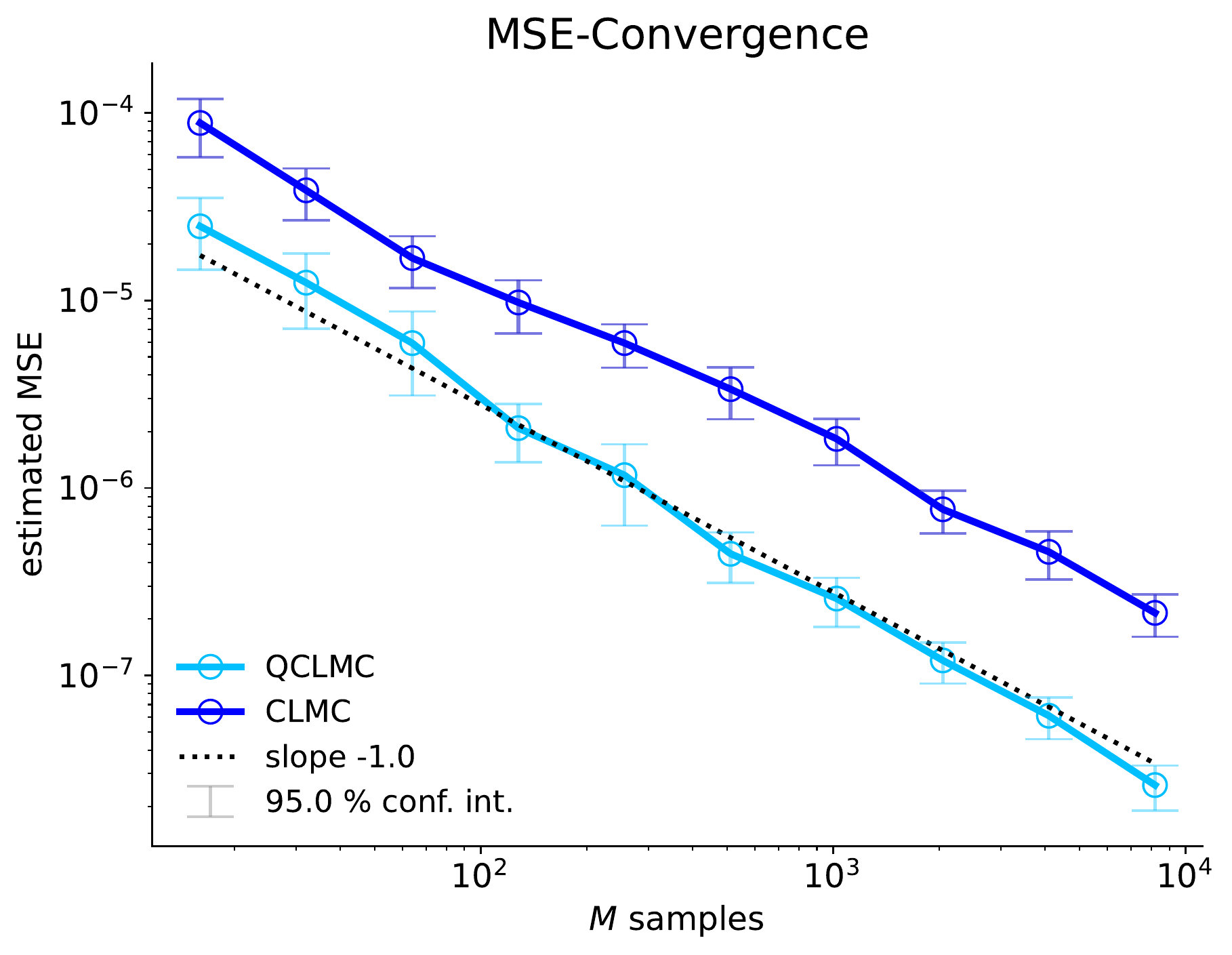}
    \includegraphics[width=0.49\textwidth]{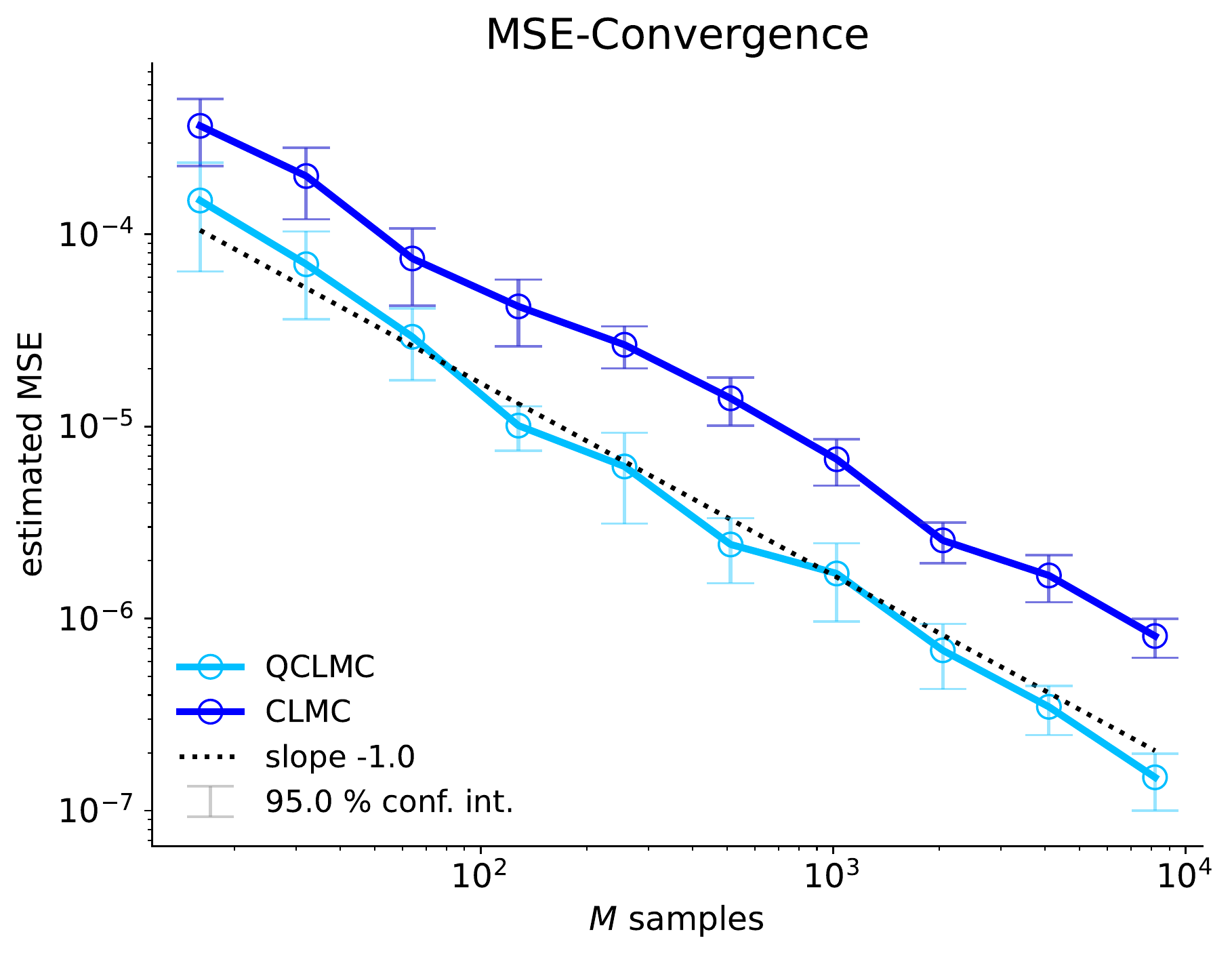}
\caption{Estimated MSE (y-axis) for CLMC and QCLMC over $100$ independent runs for sample sizes $M_i = 16 \cdot 2^i$ for $i=0,\dots,9$ (x-axis). Hyperparameters for log-Gauss field \eqref{eq:matern_kernel}: $\nu=1$, $\lambda = 0.1$, $v=0.5$ (left) and $\nu=1.5$, $\lambda = 0.1$, $v=1$ (right).}
\label{fig:MSE_convergence_CLMC_QCLMC}
\end{figure}

\begin{table}[tbhp]
\begin{tabular}{|l|l|l|l|l|l|l|l|l|l|l|}
\hline
   method / $M$      & $16$ & $32$  & $64$ & $128$ & $256$ & $512$ & $1024$ & $2048$ & $4096$ & $8192$  \\ \hline
$\MSE^\CLMC$  & 8.8e-05 & 3.7e-05 &1.7e-05 & 9.7e-06 & 5.9e-06   & 3.4e-06 & 1.8e-06 & 7.7e-07& 4.5e-07& 2.1e-07 \\ \hline
 $\MSE^\QCLMC$ & 2.5e-05 & 1.2e-05 & 5.9e-06 & 2.1e-06 & 1.2e-06   & 4.5e-07 & 2.5e-07 & 1.2e-07& 6.1e-08& 2.6e0-8 \\ \hline
 $\frac{\MSE^\CLMC}{\MSE^\QCLMC}$ & $3.6$ & $3.1$ & $2.8$ & $4.6$ & $5.1$ & $7.5$ & $7.1$ & $6.5$ & $7.5$ & $8.2$ \\ \hline
\end{tabular}
\caption{Estimated MSE values for CLMC and QCLMC for different values of $M$ together with their quotient, to be able to compare the performances. Hyperparameters for log-Gauss field \eqref{eq:matern_kernel}: $\nu=1$, $\lambda = 0.1$, $v=0.5$.}
\label{tab:mse_values_1}
\end{table}

\begin{table}[tbhp]
\begin{tabular}{|l|l|l|l|l|l|l|l|l|l|l|}
\hline
   method / $M$      & $16$ & $32$  & $64$ & $128$ & $256$ & $512$ & $1024$ & $2048$ & $4096$ & $8192$  \\ \hline
$\MSE^\CLMC$  & 3.7e-04 & 2.0e-04 &7.5e-05 & 4.2e-05 & 2.7e-05   & 1.4e-05 & 6.8e-06 & 2.6e-06& 1.7e-06& 8.1e-07 \\ \hline
 $\MSE^\QCLMC$ & 1.5e-04 & 7.0e-05 & 2.9e-05 & 1.0e-05 & 6.2e-06   & 2.4e-06 & 1.7e-06 & 6.8e-07& 3.4e-07& 1.5e0-7 \\ \hline
 $\frac{\MSE^\CLMC}{\MSE^\QCLMC}$ & $2.4$ & $2.9$ & $2.6$ & $4.2$ & $4.3$ & $5.8$ & $3.9$ & $3.7$ & $4.8$ & $5.4$ \\ \hline
\end{tabular}
\caption{Estimated MSE values for CLMC and QCLMC for different values of $M$ together with their quotient, to be able to compare the performances. Hyperparameters for log-Gauss field \eqref{eq:matern_kernel}: $\nu=1.5$, $\lambda = 0.1$, $v=1$.}
\label{tab:mse_values_2}
\end{table}

%


\bibliographystyle{abbrv}
\bibliography{arxiv_ESAIMM2ANPaper_references.bib}
%


\appendix
\section{Integral computations}
\label{app:appendix}
Here, we show the bounds for the integral terms $I $\eqref{eq:integral_I} and $II$ \eqref{eq:integral_II} from the proof of the QCLMC complexity theorem. For $I$ we compute
\begin{equation*}
 \begin{aligned}
    I = & \ \frac{c_2}{M} \int_0^{\Lm} \int_0^{\Lm} \left( \frac{1}{M} \sum_{k=1}^M \mathds{1}_{[0, L_r^\k]} (\max(\l,\l'))  -  e^{-r \max\{\l, \l'\}} \right) e^{(r -\frac{\be}{2}) \l} e^{(r -\frac{\be}{2}) \l'}  \dl \dl' \\
    = & \ \frac{c_2}{M}  \int_0^{\Lm} \left( \frac{1}{M} \sum_{k=1}^M \mathds{1}_{[0, L_r^\k]} (\l')  -  e^{-r \l'} \right) \int_0^{\l'} e^{(r -\frac{\be}{2}) \l} e^{(r -\frac{\be}{2}) \l'}   \dl \dl' \\
    & \ + \frac{c_2}{M}  \int_0^{\Lm} \int_{\l'}^{\Lm} \left( \frac{1}{M} \sum_{k=1}^M \mathds{1}_{[0, L_r^\k]} (\l)  -  e^{-r \l} \right) e^{(r -\frac{\be}{2}) \l} e^{(r -\frac{\be}{2}) \l'}   \dl \dl' \\
    \leq & \ \frac{c_2}{M}  \sup_{\l'>0} \left|\frac{1}{M} \sum_{k=1}^M \mathds{1}_{[0, L_r^\k]} (\l')  -  e^{-r \l'}\right|  \int_0^{\Lm} \int_0^{\l'} e^{(r -\frac{\be}{2}) \l} e^{(r -\frac{\be}{2}) \l'}   \dl \dl' \\
    & \ + \frac{c_2}{M}  \sup_{\l>0} \left|\frac{1}{M} \sum_{k=1}^M \mathds{1}_{[0, L_r^\k]} (\l)  -  e^{-r \l}\right| \int_0^{\Lm} \int_{\l'}^{\Lm} e^{(r -\frac{\be}{2}) \l} e^{(r -\frac{\be}{2}) \l'}  \dl \dl' \\
    \leq & \  \frac{c_{disc} c_2}{M^2} \int_0^{\Lm} \int_0^{\Lm} e^{(r -\frac{\be}{2}) \l} e^{(r -\frac{\be}{2}) \l'}   \dl \dl' \\
    =  & \ \frac{c_{disc} c_2}{M^2} \left(\int_0^{\Lm} e^{(r-\frac{\be}{2})\l} \dl\right)^2 \\
    =  & \ \frac{c_{disc} c_2}{M^2}  \begin{cases}
           \left(\frac{1}{r-\frac{\be}{2}} e^{(r-\frac{\be}{2})\Lm} - \frac{1}{r-\frac{\be}{2}}\right)^2 & \quadfor \be \neq \frac{r}{2} \\
           (\Lm)^2 & \quadfor \be = \frac{r}{2},
          \end{cases}
 \end{aligned}
\end{equation*}
using the Fubini--Tonelli theorem to compute the double integral as the square of the respective single integral. Further, we compute the square in the case $r \neq \frac{\be}{2}$ to obtain
\begin{equation*}
  \begin{aligned}
  \left(\frac{1}{r-\frac{\be}{2}} e^{(r-\frac{\be}{2})\Lm} - \frac{1}{r-\frac{\be}{2}}\right)^2 & \ = \frac{4 }{(2r - \be)^2} e^{(2r-\be)\Lm} -\frac{4 }{(2r - \be)^2} e^{(r-\frac{\be}{2})\Lm} + \frac{4 }{(2r - \be)^2} \\
  & \ = \frac{4 }{(2r - \be)^2} \left(e^{(2r-\be)\Lm} - e^{(r-\frac{\be}{2})\Lm} + 1\right).
  \end{aligned}
\end{equation*}
For $II$, we compute the following double integral using the Fubini--Tonelli theorem:
\begin{equation*}
 \begin{aligned}
 II =   & \  \frac{c_2}{M} \int_0^{\Lm} \int_0^{\Lm} e^{-r \max\{\l, \l'\}} e^{(r -\frac{\be}{2}) \l} e^{(r -\frac{\be}{2}) \l'}  \dl \dl' \\
  =  & \ c_2 \int_0^{\Lm} e^{(r-\frac{\be}{2}) \l} \int_0^{\Lm}  e^{-r \max(\l, \l')}  e^{(r - \frac{\be}{2}) \l'}  \dl' \dl.
 \end{aligned}
\end{equation*}
The inner integral computes in the case $r \neq \frac{\be}{2}$ as
\begin{equation*}
 \begin{aligned}
   \int_0^{\Lm} e^{-r \max(\l, \l')}e^{(r - \frac{\be}{2}) \l'}  \dl' & = e^{-r\l} \int_0^{\l}  e^{(r - \frac{\be}{2}) \l'}  \dl'  + \int_\l^{\Lm} e^{-\frac{\be}{2}\l'} \dl' \\
  & = \frac{1}{r - \frac{\be}{2}} e^{-r\l} \left(e^{(r-\frac{\be}{2}) \l} - 1\right) + \frac{1}{-\frac{\be}{2}} \left(e^{- \frac{\be}{2} \Lm} - e^{- \frac{\be}{2} \l}\right) \\
  & = \frac{1}{r - \frac{\be}{2}} e^{-\frac{\be}{2} \l} - \frac{1}{r - \frac{\be}{2}}e^{-r\l}  +\frac{2}{\be} e^{ - \frac{\be}{2}\l} -  \frac{2}{\be} e^{- \frac{\be}{2} \Lm}  \\
  & =  \frac{2r}{(r - \frac{\be}{2})\be} e^{-\frac{\be}{2} \l} - \frac{1}{r - \frac{\be}{2}}e^{-r\l}   -  \frac{2}{\be} e^{- \frac{\be}{2} \Lm}.
 \end{aligned}
\end{equation*}
Inserting this on top again leaves us to compute $3$ more integrals and a further case distinction, where we start with $r \neq \be$
\begin{equation*}
 \begin{aligned}
  \frac{2r}{(r - \frac{\be}{2})\be} \int_0^{\Lm} & e^{(r-\be) \l} \dl - \frac{1}{r - \frac{\be}{2}} \int_0^{\Lm}  e^{-\frac{\be}{2}\l}  \dl -  \frac{2}{\be} e^{- \frac{\be}{2} \Lm}  \int_0^{\Lm} e^{(r-\frac{\be}{2}) \l} \dl \\
  & =  \frac{2r}{(r - \frac{\be}{2})\be(r-\be)} e^{(r-\be) \Lm} - \frac{2r}{(r - \frac{\be}{2})\be(r-\be)} \\
  & \qquad + \frac{2}{(r - \frac{\be}{2})\be} e^{-\frac{\be}{2}\Lm} - \frac{2}{(r - \frac{\be}{2})\be}  -  \frac{2}{\be(r - \frac{\be}{2})} e^{(r - \be) \Lm} + \frac{2}{\be(r - \frac{\be}{2})} e^{- \frac{\be}{2} \Lm} \\
  & =  \frac{2}{(r-\be) (r - \frac{\be}{2})} e^{(r-\be) \Lm} +  \frac{4}{\be(r - \frac{\be}{2})} e^{- \frac{\be}{2} \Lm} -  \frac{4}{(r-\be) \be},
 \end{aligned}
\end{equation*}
and for the case $r = \be$ we compute
\begin{equation*}
 \begin{aligned}
   \frac{2r}{(r - \frac{\be}{2})\be} \int_0^{\Lm} & e^{(r-\be) \l} \dl - \frac{1}{r - \frac{\be}{2}} \int_0^{\Lm}  e^{-\frac{\be}{2}\l}  \dl -  \frac{2}{\be} e^{- \frac{\be}{2} \Lm}  \int_0^{\Lm} e^{(r-\frac{\be}{2}) \l} \dl\\
  & =  \frac{4}{\be} \Lm + \frac{4}{\be^2} e^{-\frac{\be}{2}\Lm} - \frac{4}{\be^2} -  \frac{4}{\be^2} + \frac{4}{\be^2} e^{- \frac{\be}{2} \Lm} = \frac{4}{\be} \Lm + \frac{8}{\be^2} e^{-\frac{\be}{2}\Lm}-  \frac{8}{\be^2}.
 \end{aligned}
\end{equation*}
For the special case $r = \frac{\be}{2}$ we obtain
\begin{equation*}
 \begin{aligned}
   \int_0^{\Lm} & e^{-\frac{\be}{2} \max(\l, \l')} \dl'  = e^{-\frac{\be}{2}\l} \int_0^{\l} \dl'  + \int_\l^{\Lm} e^{-\frac{\be}{2}\l'} \dl' \\
  & =  \l \, e^{-\frac{\be}{2}\l}  + \frac{1}{-\frac{\be}{2}} \left(e^{- \frac{\be}{2} \Lm} - e^{- \frac{\be}{2} \l}\right)  = \l \, e^{-\frac{\be}{2}\l}  +\frac{2}{\be} e^{ - \frac{\be}{2}\l} -  \frac{2}{\be} e^{- \frac{\be}{2} \Lm}, 
 \end{aligned}
\end{equation*}
and inserting this on top leads to
\begin{equation*}
 \begin{aligned}
   \int_0^{\Lm} & \l \, e^{-\frac{\be}{2}\l} \dl + \frac{2}{\be}\int_0^{\Lm}   e^{ - \frac{\be}{2}\l}\dl -  \frac{2}{\be} e^{- \frac{\be}{2} \Lm}  \int_0^{\Lm} \dl \\
  & = -\frac{4}{\be^2} e^{-\frac{\be}{2} \Lm} - \frac{2}{\be} \Lm e^{-\frac{\be}{2} \Lm} + \frac{4}{\be^2} - \frac{4}{\be^2}e^{-\frac{\be}{2} \Lm} + \frac{4}{\be^2} - \frac{2}{\be} \Lm e^{- \frac{\be}{2} \Lm} \\
  & = -\frac{8}{\be^2} e^{-\frac{\be}{2} \Lm} - \frac{4}{\be} \Lm e^{- \frac{\be}{2} \Lm} + \frac{8}{\be^2}.
 \end{aligned}
\end{equation*}
Overall, we obtain the bounds
\begin{equation*}
  I  \leq  \frac{c_{disc}\, c_2}{M^2} \begin{cases}
      \left(\frac{4}{(2r - \be)^2} \left(e^{(2r-\be)\Lm} - e^{(r-\frac{\be}{2})\Lm} + 1\right)\right) &\quadfor r \neq \frac{\be}{2}, \\
      (\Lm)^2 &\quadfor r = \frac{\be}{2},
  \end{cases}
 \label{eq:qclmc_var_int_I}
\end{equation*}
and
\begin{equation*}
  II  \leq \frac{c_2}{M} \begin{cases}
                      \left(\frac{2}{(r-\be) (r - \frac{\be}{2})} e^{(r-\be) \Lm} +  \frac{4}{\be(r - \frac{\be}{2})} e^{- \frac{\be}{2} \Lm} -  \frac{4}{(r-\be) \be} \right)   & \quadfor r \neq \frac{\be}{2}, \be, \\
                        -\frac{8}{\be^2} e^{-\frac{\be}{2} \Lm} - \frac{4}{\be} \Lm e^{- \frac{\be}{2} \Lm} + \frac{8}{\be^2} & \quadfor r = \frac{\be}{2}, \\
                       \frac{4}{\be} \Lm + \frac{8}{\be^2} e^{-\frac{\be}{2}\Lm} - \frac{8}{\be^2} & \quadfor r = \be.
                     \end{cases}
  \label{eq:qclmc_var_int_II}
\end{equation*}
\end{document}